\documentclass[11pt]{article}
\usepackage{graphicx}
\usepackage{amsmath,amssymb,amsthm,amsfonts,mathrsfs}
\usepackage{amssymb}
\usepackage{enumerate}
\usepackage[numbers,sort&compress]{natbib}
\usepackage{indentfirst}
\usepackage{amsmath}

\allowdisplaybreaks[4]

\numberwithin{equation}{section}

\DeclareMathSymbol{\C}{\mathalpha}{AMSb}{"43}

\newtheorem{Thm}{Theorem}[section]
\newtheorem{Lem}[Thm]{Lemma}

\newtheorem{Rem}[Thm]{Remark}

\newcommand{\e}\varepsilon
\newcommand{\vs}{\vspace}

\newcommand{\bsub}{\begin{subequations}}
\newcommand{\esub}{\end{subequations}$\!$}

\begin{document}

\title{On the existence of positive solution for a Neumann problem with double critical exponents in half-space
%\thanks {The research was supported by the Natural
%Science Foundation of China (11701203) and
%the NSF of Hubei Province (No.2018CFB268).
%}
%\thanks{The research was supported by the Natural Science Foundation of China (12271196, 11931012).
%}
}

\author{Yinbin Deng
\thanks {Contributing author. School of Mathematics and Statistics \& Hubei Key Laboratory of Mathematical Sciences, Central China Normal University, Wuhan 430079, China.  Email: \texttt{ybdeng@ccnu.edu.cn}.}
\ \  Longge Shi
\thanks{Corresponding author. School of Mathematics and Statistics \& Hubei Key Laboratory of Mathematical Sciences, Central China Normal University, Wuhan 430079, China.
Email: \texttt{shilongge@mails.ccnu.edu.cn}.}}

\date{}

\smallbreak \maketitle

\begin{abstract}
In this paper, we consider the existence and nonexistence of positive solution for the
 following critical Neumann problem
\begin{equation}\label{01as}
	\left\{
	\begin{aligned}
		-\Delta {u}-\frac{1}{2}(x \cdot{\nabla u})&= \lambda u+a{|u|^{{2}^{*}-2}u}& \ \ \mbox{in} \ \ \ {{\mathbb{R}}^{N}_{+}}, \\
		 \frac{{\partial u}}{{\partial n}}&={ {\mu|u|^{q-2}u}}+|u|^{{2}_{*}-2}u \ & \mbox{on}\ {{\partial {{\mathbb{R}}^{N}_{+}}}},
	\end{aligned}
	\right.
\end{equation}
where $ \mathbb{R}^{N}_{+}=\{(x{'}, x_{N}): x{'}\in {\mathbb{R}}^{N-1}, x_{N}>0\}$ is the upper half-space, $N\geq3$, $\lambda, \mu\in\mathbb{R}$ are parameters, $a\in\{0, 1\}$,  $n$ is the outward normal vector at the boundary ${{\partial {{\mathbb{R}}^{N}_{+}}}}$, $2\leq q <{2}_{*}$, $2^{*}=\frac{2N}{N-2}$ is the usual critical exponent for the Sobolev embedding $D^{1,2}({\mathbb{R}}^{N}_{+})\hookrightarrow {L^{{2}^{*}}}({\mathbb{R}}^{N}_{+})$ and ${2}_{*}=\frac{2(N-1)}{N-2}$ is the critical exponent for the Sobolev trace embedding  $D^{1,2}({\mathbb{R}}^{N}_{+})\hookrightarrow {L^{{2}_{*}}}(\partial \mathbb{R}^{N}_{+})$.
By applying the Mountain Pass Theorem without (PS) condition and the delicate estimates for Mountain Pass level, we obtain the existence of a positive solution under different assumptions on $\lambda$,  ${\mu}$ and  $q$. Meanwhile, some nonexistence results for problem  (\ref {01as}) is also obtained by an improved Pohozaev identity and Hardy inequality
 according to the value of the parameters $ \lambda$, ${\mu}$ and $q$. Particularly, for $N\geq3$, $\lambda \ge 0$, $\mu>0$ and $q\in(2, 2^*)$, we obtain that problem \eqref{01as} has a positive solution if and only if  $\lambda\in[0, \frac N2)$; On the other hand,
 for $N\ge 3$ and $\mu =0$, we find a lower bound $\Lambda^*\in [\frac N4, \frac N2)$ depending on $N$ such that problem (\ref {01as}) has a positive solution if $\lambda \in (\Lambda^*, \frac N2)$  and has no positive solution if $\lambda \in (-\infty, \frac N4)\cup [\frac N2, +\infty)$. Moreover, we estimate  that $\Lambda ^* \in (\frac N4, \frac {N-2}2)$ if $N\ge 5$  and $\Lambda ^* =\frac N4 =1$ if $N=4$ which gives the best lower bound of $\lambda$ for the existence of a positive solution of problem (\ref {01as}) if $N=4$.
\end{abstract}

\vskip 0.12truein


\vskip 0.02truein


\vskip 0.2truein

\section{Introduction}

In this paper, we  are concerned with the existence and nonexistence of positive solution for the
 following  Neumann problem
\begin{equation}\label{1.1}
	\left\{
	\begin{aligned}
		-\Delta {u}-\frac{1}{2}(x \cdot{\nabla u})&= \lambda u+a{|u|^{{2}^{*}-2}u}& \ \ \mbox{in} \ \ \ {{\mathbb{R}}^{N}_{+}}, \\
		 \frac{{\partial u}}{{\partial n}}&={ {\mu|u|^{q-2}u}}+|u|^{{2}_{*}-2}u \ & \mbox{on}\ {{\partial {{\mathbb{R}}^{N}_{+}}}}
	\end{aligned}
	\right.
\end{equation}
with the critical exponents,
where $ \mathbb{R}^{N}_{+}=\{(x{'}, x_{N}): x{'}\in {\mathbb{R}}^{N-1}, x_{N}>0\}$ is the upper half-space, $N\geq3$,  $2^{*}=\frac{2N}{N-2}$ is the usual critical exponent for the Sobolev embedding $D^{1,2}({\mathbb{R}}^{N}_{+})\hookrightarrow {L^{{2}^{*}}}({\mathbb{R}}^{N}_{+})$ and ${2}_{*}=\frac{2(N-1)}{N-2}$ is the critical exponent for the Sobolev trace embedding  $D^{1,2}({\mathbb{R}}^{N}_{+})\hookrightarrow {L^{{2}_{*}}}(\partial \mathbb{R}^{N}_{+})$, $\lambda, \mu \in\mathbb{R}$ are parameters, $a\in\{0, 1\}$,  $n$ is the outward normal vector at the boundary ${{\partial {{\mathbb{R}}^{N}_{+}}}}$, $2\leq q <{2}_{*}$.

The problem (\ref {1.1}) relates to the self-similar solutions for the heat equation
\begin{equation}\label{1.3}
		v_{t}-\Delta {v}=f(v) \ \ \mbox{in} \ \ \ {{\mathbb{R}}^{N}}{\times(0,\infty)},
\end{equation}
where the nonlinearity $f$ is a power function. It is know that equation \eqref{1.3} is invariant under the similarity transformation
\begin{equation*}
v(x,t)\mapsto v_\lambda(x,t)=\lambda^{\frac{2}{p-2}}v(\lambda x,\lambda^{2}t) \quad\mbox{for any}~\lambda>0,
\end{equation*}
provided that $f(v)=|v|^{p-2}v$ and $p>2$.
%Recall that a solution $v$ is said to be self-similar, if $v=v_\lambda$ for any $\lambda>0$, that is,
%\begin{equation*}
%v(x,t)= v_\lambda(x,t)=\lambda^{\frac{2}{p-2}}v(\lambda x,\lambda^{2}t) \quad\mbox{for any}~\lambda>0.
%\end{equation*}
A solution $v$ is said to be forward self-similar of equation \eqref{1.3} if and only if $v$ has the special form
\begin{equation*}
v(x,t)=t^{-\frac{1}{p-2}}u(\frac{x}{\sqrt{t}}),
%~~x \in {{\mathbb{R}}^{N}},~~t>0,
\end{equation*}
where
$u$ verifies
\begin{equation*}\label{002}
		-\Delta {u}-\frac{1}{2}(x \cdot{\nabla u})-\frac{1}{p-2} u={|u|^{p-2}u}\ \ \mbox{in} \ \ \ {{\mathbb{R}}^{N}}.
\end{equation*}
Such self-similar solutions have attracted widespread attention since they preserve the PDE scaling and carry simultaneously information about small and large scale behaviors (see \cite{17,37,29}).

There are numerous studies on self-similar solutions to problem (\ref{1.3}).
In \cite{21}, Haruax and Weissler investigated equation \eqref{1.3} with  $f(v)=|v|^{p-2}v$ and $2+2/N<p<2^*$. They obtained the existence of forward self-similar solutions by ODE technique. Still using ODE techniques, Brezis et al. \cite{9} constructed forward self-similar solutions to equation \eqref{1.3} with $f(v)=-|v|^{p-2}v$ and $2<p<2+2/N$. This result was later covered by M. Escobedo and O. Kavian (see \cite{10}). They were the first authors to propose variational approach to nonlinear heat problems.
By introducing a weighted Sobolev space, they established the existence of self-similar solutions to equation \eqref{1.3} with the same nonlinearity $f$ as \cite{21}.
When $f=\alpha \cdot\nabla ( |v|^{q-2}v)- |v|^{p-2}v$, $2<p<2+2/N, q=(p+2)/2$ and $\alpha \in\mathbb{R}^N$, thanks to the variational structure, Escobedo and Zuazua \cite{111} obtained self-similar solutions of equation \eqref{1.3} via a fixed point Theorem. Natio \cite{28} studied equation \eqref{1.3} with singular initial data $v(x,0)=l|x|^{-{2}/{(p-2)}}$ provided that  $f(v)=|v|^{p-2}v, p>2+2/N$ and $l>0$. They established the existence and multiplicity of positive self-similar solutions by ODE technique.
 For more results about the existence of solutions for equation \eqref{1.3} with different boundary conditions  on the bounded domains,  upper half-space and even in the whole space, we refer the readers to \cite{4, 5, 23, 33, 35, 12,19, 26,27,36} and references therein.

The mathematical literature concerning with the self-similar solutions to the  heat equations with Neumann boundary conditions in half space does not seem to be very extensive. Ferreira et al. \cite{2015} considered
 \begin{equation}\label{1.2}
	\left\{
	\begin{aligned}
		v_{t}-\Delta {v}&=\theta{|v|^{p-2}v}& \ \ \mbox{in} \ \ \ {{\mathbb{R}}^{N}_{+}}{\times(0,\infty)},\\
		\frac{{\partial v}}{{\partial n}}&={|v|^{q-2}v} \ & \mbox{on} \ {{\partial {{\mathbb{R}}^{N}_{+}}}}{\times(0,\infty)},
	\end{aligned}
	\right.
\end{equation}
where  $N\geq3$, $\theta\in\{0, 1\}, 2<p< 2^*$, $2<q<2_*$ and $q=(p+2)/2$.
If $\theta=0$, they showed that equation \eqref{1.2} has a positive  self-similar solution for $2+1/N<q<2_*$. If $\theta=1$, they obtained infinitely many self-similar solutions for equation \eqref{1.2} provided $1/(p-2)\in\mathbb{R}\setminus\Sigma$, where $\Sigma$ is a countable set associated with the eigenvalues of the linear problem. Later on, Ferreira et al. \cite{2021} established a nonexistence result of self-similar solutions for equation \eqref{1.2} when $N\geq3$, $ \theta=1, p=2^*$ and $q=2_*$.

Note that from the point of view of elliptic PDEs, the study of equation \eqref{1.2} has an interest of its own. In fact, the following elliptic equation with Neumann boundary condition
\begin{equation}\label{1.0}
	\left\{
	\begin{aligned}
		-\Delta {u}-\frac{1}{2}(x \cdot{\nabla u})&= \lambda u+ a{|u|^{p-2}u}& \ \ \mbox{in} \ \ \ {{\mathbb{R}}^{N}_{+}}, \\
		 \frac{{\partial u}}{{\partial n}}&={ {|u|^{q-2}u}} \ & \mbox{on}\ {{\partial {{\mathbb{R}}^{N}_{+}}}},
	\end{aligned}
	\right.
\end{equation}
naturally appears when one tries to find self-similar solutions for equation \eqref{1.2}.
In \cite{2015}, Ferreira et al. considered equation \eqref{1.0} provided that $N\geq3, \lambda\in \mathbb{R}, a\in\{0,1\}, 2<p<2^*, 2<q<2_* $. They showed the existence and multiplicity of solutions according to the range of $\lambda$ and $a$. Equation \eqref{1.0} with critical growth was recently
studied in \cite{2021}.  They established the existence of a positive solution for the Neumann problem
\begin{equation}\label{1.012}
	\left\{
	\begin{aligned}
		-\Delta {u}-\frac{1}{2}(x \cdot{\nabla u})&= \lambda u+ {|u|^{2^*-2}u}& \ \ \mbox{in} \ \ \ {{\mathbb{R}}^{N}_{+}}, \\
		 \frac{{\partial u}}{{\partial n}}&=|u|^{{2}_{*}-2}u \ & \mbox{on}\ {{\partial {{\mathbb{R}}^{N}_{+}}}},
	\end{aligned}
	\right.
\end{equation}
for $N\geq 3 $ if $\lambda\in(\frac{N}{2}-\delta,  \frac {N}{2})$, where $\delta>0$ is a small constant. They also obtained the existence of a positive solution for the Neumann problem
\begin{equation}\label{1.011}
	\left\{
	\begin{aligned}
		-\Delta {u}-\frac{1}{2}(x \cdot{\nabla u})&= \lambda u  & \ \ \mbox{in} \ \ \ {{\mathbb{R}}^{N}_{+}}, \\
		 \frac{{\partial u}}{{\partial n}}&=|u|^{{2}_{*}-2}u \ & \mbox{on}\ {{\partial {{\mathbb{R}}^{N}_{+}}}},
	\end{aligned}
	\right.
\end{equation}
 for $N\geq7$ if $\lambda\in(\frac {N}{4}+\frac {(N-4)}{8}, \frac {N}{2})$.

 Unfortunately, the existence of positive solutions for problem (\ref {1.011}) when $3\le N \le 6$ and the best lower bound of $\lambda$ for the existence of positive solutions for problem (\ref {1.012}) are remained open.

In this paper, we consider the existence of a positive solution for critical Neumann problem \eqref{1.1} with  $N\geq3, \lambda, \mu\in \mathbb{R}, a\in\{0,1\}, 2\leq q<2_*$. It is clear that problem (\ref {1.012}) and (\ref {1.011}) can be derived by taking $\mu =0$ in (\ref {1.1}). For $\mu =0$, we find a lower bound $\Lambda^*\in [\frac N4, \frac N2)$ depending on $N$ such that problem (\ref {1.1}) has a positive solution if $\lambda \in (\Lambda^*, \frac N2)$  and has no positive solution if $\lambda \in (-\infty, \frac N4)\cup [\frac N2, +\infty)$ for all $N\ge 3$. Moreover, we estimate  that $\Lambda ^* \in (\frac N4, \frac {N-2}2)$ if $N\ge 5$  and $\Lambda ^* =\frac N4$ if $N=4$ which gives the best lower bound of $\lambda$ for the existence of a positive solution of (\ref {1.1}) if $N=4$.

It should be mentioned that Furtado and da Silva \cite{116} investigated the existence
of a positive solution for equation \eqref{1.1} according to the value of $\mu >0$ and $q\in [2, 2_*)$ when $N\ge 4$, $\lambda=0$ and $a=0$. We highlight here the literatures \cite {BN1983}, \cite{2021}-\cite{116} and \cite {777}
which strongly motivate our main results.
%\cite{2015,2021,116}
\vskip 0.2cm

For convenience, we denote $ \mathbb{R}^{N-1}:=\partial\mathbb{R}^{N}_{+}$ and $\int_{\mathbb{R}^{N-1}}:=\int_{\partial\mathbb{R}^{N}_{+}}$.
Since the exponential-type weight $K(x):= e^{{|x|^{2}}/{4}}$
verifies $\text{div}(K(x)\nabla {u})=K(x)\big(\Delta {u}+\frac{1}{2}(x \cdot{\nabla u})\big)$ for any regular function $u$, problem (\ref{1.1}) can be rewritten as
\begin{equation}\label{1.5}
	\left\{
	\begin{split}
		-\mbox{div}(K(x)\nabla {u})&= {\lambda K(x) {u}}+a K(x){|u|^{{2}^{*}-2}u} \ \ \hspace{0.5cm} \mbox{in} \ \ {\mathbb{R}}^{N}_{+}, \ \  \\
		K(x{'},0)\frac{{\partial u}}{{\partial n}}&=\mu K(x{'},0){|u|^{q-2}u}+K(x{'},0){|u|^{{2}_{*}-2}u} \ \  \ \mbox{on} \ {\mathbb{R}}^{N-1}.
	\end{split}
	\right.
\end{equation}
Hence, it is nature to look for solutions in the space $X$ defined as the closure of ${{C_{c}^{\infty}}(\overline{{{\mathbb{R}}^{N}_{+}}})}$ with respect to the norm
\begin{equation*}
 \|{u}\|= \bigg({\int_{\mathbb{R}^{N}}}K(x)|\nabla{u}|^{2}dx \bigg)^{\frac{1}{2}},
\end{equation*}
which is induced by the inner product
%\begin{equation*}
$(u,v)={\int_{\mathbb{R}^{N}}}K(x)\nabla{u}\nabla{v}dx.$
%\end{equation*}
It follows from \cite{004} and \cite{2021} that the embedding $X \hookrightarrow L^{r}_{K}(\mathbb{R}^{N}_{+})$ is continuous for $r\in[2, 2^{*}]$ and compact for $r\in[2, 2^{*})$, the embedding $X\hookrightarrow {L^{s}_{K}}({\mathbb{R}}^{N-1})$ is continuous for $s\in[2, 2_{*}]$ and compact for $s\in[2, 2_{*})$, where
\begin{eqnarray}
&&{\ \ \ L^{r}_{K}({\mathbb{R}}^{N}_{+})}:= \bigg\{u \in {L^{r}({\mathbb{R}}^{N}_{+})}:{\int_{\mathbb{R}^{N}_{+}}K(x)|u|^{r}dx} < {\infty}\bigg\}, \notag\\
&&{L^{s}_{K}({\mathbb{R}}^{N-1})}:= \bigg\{u \in {L^{s}({\mathbb{R}}^{N-1})}:{\int_{\mathbb{R}^{N-1}}K(x{'},0)|u|^{s}dx{'}} < {\infty}\bigg\}. \notag
\end{eqnarray}
Then the energy functional $I^a_{\lambda, \mu}: X\rightarrow \mathbb{R}$ %associated to \eqref{1.5}
given by
\begin{equation}
	I^a_{\lambda, \mu}(u):={\frac{1}{2}}{\|{u}\|}^{2}-{\frac{\lambda }{2}}\|u_+\|_{L^{{2}}_{K}({\mathbb{R}}^{N}_{+})}^{2}
-{\frac{a }{2^{*}}}\|{u_{+}}\|_{L^{{2}^{*}}_{K}({\mathbb{R}}^{N}_{+})}^{2^{*}}
-{\frac{1 }{2_{*}}}\|{u_{+}}\|_{L^{{2}_{*}}_{K}({\mathbb{R}}^{N-1})}^{2_{*}}
-{\frac{\mu }{q}}{\|{u_{+}}\|}_{{L_{K}^{q}}({\mathbb{R}}^{N-1})}^{q}\notag
\end{equation}
is well defined and belongs to $C^{1}(X, \mathbb{R})$,
where $u_{+}= \max \{0,u\}$, $u_{-}= - \min \{0,u\}$.

We recall two basic results about the linear problem associated to \eqref{1.1} (see \cite{2015} and \cite{116}). The first eigenvalue of the linear problem
\begin{equation}\label{qw1}
\left\{
	\begin{aligned}
		-\Delta {u}-\frac{1}{2}(x \cdot{\nabla u})&={\lambda} u& \ \ \mbox{in} \ \ \ {\mathbb{R}}^{N}_{+},  \ \\
		 \frac{{\partial u}}{{\partial n}}&=0& \mbox{on}\ {\mathbb{R}}^{N-1},
	\end{aligned}
	\right.
\end{equation}
is characterized by
\begin{equation}\label{1}
{\lambda}_{1}:=\inf_{u\in X\setminus \{0\}} \frac{ \|{u}\|^{2}}
{{{\|{u}\|}}_{{L_K^{2}({\mathbb{R}}^{N}_{+})}}^{2}}=\frac{N}{2}.
\end{equation}
And the first positive eigenvalue of the linear problem
\begin{equation}\label{qw2}
\left\{
	\begin{aligned}
		-\Delta {u}-\frac{1}{2}(x \cdot{\nabla u})&=0& \ \ \mbox{in} \ \ \ {\mathbb{R}}^{N}_{+},  \ \\
		 \frac{{\partial u}}{{\partial n}}&=\mu u& \mbox{on}\ {\mathbb{R}}^{N-1},
	\end{aligned}
	\right.
\end{equation}
is defined by
\begin{equation}\label{2}
{\mu}_{1}:=\inf_{u\in X\setminus \{0\}} \frac{ \|{u}\|^{2}}
{{{\|{u}\|}}_{{L_K^{2}({\mathbb{R}}^{N-1})}}^{2}}>0.
\end{equation}
The functional $I^a_{\lambda, \mu}$ is introduced in order to obtain nonnegative weak solutions of \eqref{1.5}. Indeed, if $u\in X$ is a nonzero critical point of $I^a_{\lambda, \mu}(u)$, we follow from
(\ref{1}) that $u_{-}\equiv 0$, and in turn $u
\geq0$. Therefore, to obtain a nonnegative nontrivial weak solution of \eqref{1.5}, it suffices to find a nonzero critical point of $I^a_{\lambda, \mu}$.

%Existence of solutions to equation \eqref{1.1} strongly depends on the parameters $N, \lambda, a, \mu, q$. Patrial results about equation \eqref{1.1}
%were presented in \cite{2021,116}. The purpose of this paper is first to improve the existence result of problem $(P)$ in \cite{2021}. Secondly,
%we give an answer to the problem concerning the existence of solutions for equation $(\tilde{P})$ with lower dimensions $3\leq N \leq6$ proposed in
%\cite{2021}. Thirdly, we complement the existence of solutions to problem $(P_\lambda)$ in \cite{116} for dimension $N=3$ . Finally, the above
%problems with subcritical or critical perturbation also be investigated. More specifically, we prove the following results:

%By Pohozaev identity and a Hardy-type inequality, we establish our first nonexistence result.

Our first Theorem is concerned with the nonexistence of positive solutions for problem (\ref {1.1}).
\begin{Thm}\label{Th1.0}
	Let $N\geq3$, $a\in\{0, 1\}$ and $2\leq q<2_*$. Suppose that $u\in C^2(\mathbb{R}^N_+)\cap X$ is a nonnegative solution of equation \eqref{1.1}, then $u\equiv0$, if one of the following assumptions holds:
	\vskip 0.2cm
	
	$(1)$  $\lambda<{N}/{4}$ and $\mu=0$;
\vskip 0.2cm
	
	$(2)$  $\lambda\leq{N}/{4}$ and $\mu<0$;
\vskip 0.2cm
	
	$(3)$  $\lambda\geq{N}/{2}$ and $\mu\geq0$;
%\vskip 0.2cm
	
	%$(4)$  $a=0, \lambda>{N}/{2}$ and $\mu\geq0$;
\vskip 0.2cm
	
	$(4)$  $\lambda\geq0$ and $\mu\geq\mu_1$ if  $q=2$,
%\vskip 0.2cm
	
	%$(6)$  $q=2, a=0, \lambda>0$ and $\mu\geq\mu_1$,
\vskip 0.2cm
\noindent
where $\mu_1$ is defined by \eqref{2}.
\end{Thm}
%\begin{Rem}\label{re1.1}
As a meter of fact, we can deduce that problem (\ref {Th1.0}) has no nontrivial solution under the assumption  $(1)$ or $(2)$
from Pohozaev identity and a Hardy-type inequality (see the proof of Theorem \ref {Th1.0}).
%\end{Rem}

Our second Theorem is concerned with  the existence of a positive solution to equation \eqref{1.1} when $a=1$.
\begin{Thm}\label{Th1.1}
	Let $N\ge 3$ and $a=1$. Problem \eqref{1.1} has a positive solution if one of the following assumptions holds:
	
	$(1)$  $\mu=0$, $ \lambda \in (\bar \lambda, \frac {N}{2})$ ;
	
$(2)$  $\mu \in (0, +\infty)$, $\lambda \in [0, \frac {N} {2})$ if $2<q<2_*$;

$(3)$  $\mu \in (0, \mu_1 -\frac {2\mu_1}{N}\lambda )$, $\lambda \in [0, \frac {N} {2})$ if $q=2$,

\noindent
where
\begin{align} \label{barlambda}
\bar{\lambda}=
	\begin{cases}
\frac{3+\sqrt5}{4},~&~N=3,\\
1,~&~N=4,\\
 \lambda^*_N,~&~N\geq5 ,\\
 	\end{cases}
 \end{align}
and $\lambda_N^*\in (\frac N4, \frac {N-2}{2})$ is a constant dependent of $N$ and  $\mu_1$ is defined by \eqref{2}.
\end{Thm}

Our final Theorem is concerned with the existence of a positive solution to equation \eqref{1.1} when $a=0$.
\begin{Thm}\label{Th1.2}
	Let $N\ge 3$ and $a=0$. Equation \eqref{1.1} has a positive solution if one of the following assumptions holds:
	
	$(1)$  $\mu=0$,  $ \lambda \in (\hat \lambda, \frac {N}{2})$ ;
	
	$(2)$  $\mu \in (0, +\infty)$, $\lambda \in [0, \frac {N}{2})$ if $2<q<2_*$;

$(3)$  $\mu \in (0, \mu_1 -\frac {2\mu_1}{N}\lambda )$, $\lambda \in [0, \frac {N} {2})$ if $q=2$,

\noindent
where
\begin{align}\label{hatlambda}
\hat{\lambda}=
	\begin{cases}
\frac{3+\sqrt5}{4},~&~N=3,\\
\frac {N}{4}+\frac {N-4} {8},~&~N\geq4,
	\end{cases}
\end{align}
and  $\mu_1$ is defined by \eqref{2}.
\end{Thm}

\begin{Rem}\label{remark1.2}
In \cite{2021}, Ferreira et al. considered equation \eqref{1.1} with $N\geq3, \lambda\in \mathbb{R}, a=1$ and $\mu=0$. Since they used the first eigenfunction of the linear problem \eqref{qw1} as a test function to show that the Mountain Pass value belongs to the compactness range, the existence result for problem \eqref{1.1} is obtained only if $\lambda$ is near ${N}/{2}$ from left.
However, we  choose the test function related to the achieved function of the  infimum \eqref{2.1}  to get a local compactness result and then establish the natural range for the parameters $N$ and $\lambda$ from some delicate
 estimates of the test functions.
%$u_{{\varepsilon}}$ and $v_{{\varepsilon}}$.
%In \cite{2021}, Ferreira et al. used the first eigenfunction of the linear problem \eqref{qw1} to show that the Mountain Pass value belongs to the compactness range. Therefore, they obtained an existence result of \eqref{1.1} if $N\geq3, a=1, \mu=0$ and $\lambda$ is near ${N}/{2}$. In  the hypothesises $(1)$-$(3)$ of Theorem \ref {Th1.1}, we use the infimum \eqref{2.1} to get a local compactness result. Thanks to some detailed estimates of $u_{{\varepsilon}}$ and $v_{{\varepsilon}}$, we can establish the natural range for the parameters $\lambda$ and $N$ to obtain solution.
\end{Rem}

\begin{Rem}\label{remark1.2}
When $N\geq 3, \lambda\in \mathbb{R}, a=0$ and $\mu=0$, Ferreira et al. \cite{2021} showed that \eqref{1.1} has a positive solution for $N\geq 7$ and
${N}/{4}+{(N-4)}/{8}<\lambda<{N}/{2}$. Here, we complement the existence result for lower dimensions $3\leq N \leq6$.
%Particularly, it follows from the cases $(1)$ and $(3)$ of Theorem \ref {Th1.0} and the case $(2)$ of Theorem \ref {Th1.2} that, for $N=4$, the problem  (\ref {1.1}) has a positive solution if and only if $1<\mu<2$.
Moreover, we also extend the result in  \cite{116} when $\lambda=0, a=0$
and $\mu >0$.
\end{Rem}

The existence and nonexistence results given by Theorem \ref{Th1.0} - Theorem \ref{Th1.2} can be described on the $(\lambda, \mu)$ plane as Figure 1.
The pink regions and the red line stand for the regions where  problem (\ref {1.1}) exists a positive solution, while the blue regions correspond to the regions where problem (\ref {1.1}) has no positive solution. Here, the curve $\eta$ is given by $1-{2\lambda}/{N}-\mu/\mu_1=0$,
$\Lambda^*=
	\begin{cases}
\bar \lambda \ \   if   &a=1, \\
\hat \lambda \ \   if  &a=0,
	\end{cases}$
and the values of $\bar{\lambda}, \hat{\lambda}$ are given by (\ref {barlambda}) and (\ref {hatlambda}) respectively.
\begin{figure*}[http]
  \centering
  \includegraphics[width=0.8\textwidth,height=0.3\textheight]{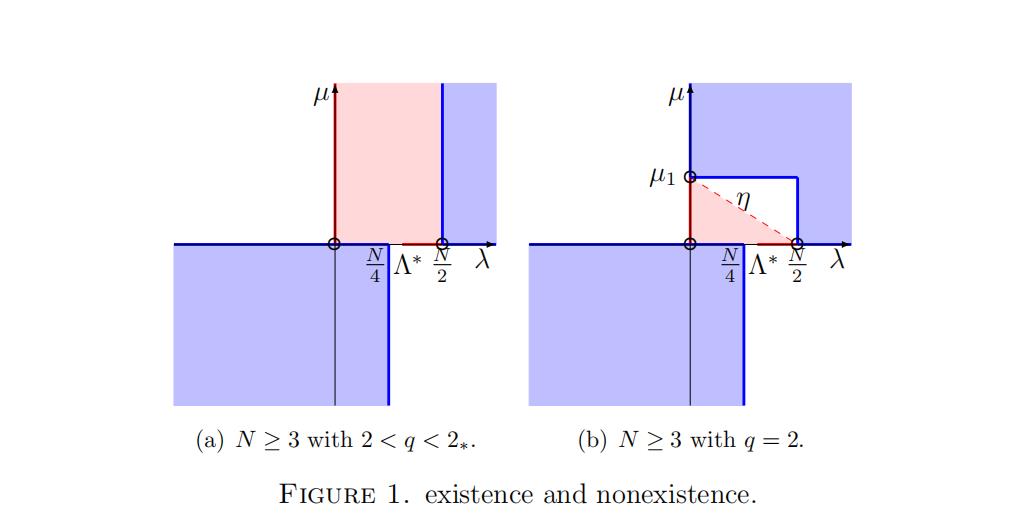}
\end{figure*}
%\begin{figure*}[http]
 % \centering
  %\includegraphics[width=0.8\textwidth,height=0.3\textheight]{a=1.pdf}
%\end{figure*}
\begin{Rem}\label{remark1.7}
From  Figure 1, we have the following observations:

 $(1)$  For $N\geq3$ and $\mu=0$, problem \eqref{1.1} has a positive solution if  $\lambda\in(\Lambda^*, \frac{N}{2})$
  and has no positive solution if $\lambda\in(-\infty, \frac{N}{4})\cup[\frac{N}{2},+\infty)$. Unfortunately,
  we do not know what happens if $\lambda\in[\frac{N}{4}, \Lambda^*]$.  Fortunately, we can estimate that $\Lambda ^* < \frac {N-2}{2}$ if $N\ge 5$ and
  $\Lambda^* =\bar \lambda =\hat \lambda =\frac N4$ if $N=4$  which provide  a largest interval $(1, 2)$ for parameter $\lambda$ where problem  \eqref{1.1} has a positive solution when $N=4$.

$(2)$ For $N\geq3$, $\lambda \ge 0$, $\mu>0$ and $q\in(2, 2^*)$, we obtain that problem \eqref{1.1} has a positive solution if and only if  $\lambda\in[0, \frac{N}{2})$.

\end{Rem}

In the proof of Theorem \ref{Th1.1} and Theorem \ref{Th1.2}, we follow from \cite{0} and \cite{2021}. After obtaining the local compactness result for the functional $I^a_{\lambda, \mu}$, the main step is to prove that the Mountain Pass level of the functional belongs to the compact range. At this point, different from the test function used in \cite{2021}, we construct the following test functions
\begin{align*}
u_{{\varepsilon}}(x)
&=K(x)^{-\frac{1}{2}}\phi(x)\frac{\big({\varepsilon^2} N(N-2)\big) ^{\frac{N-2}{4}}}{\big(\varepsilon^{2}+|x'|^{2}+|x_{N}+\varepsilon x_{N}^{0} |^{2}\big)^{\frac{N-2}{2}}},\quad N\geq4,\\
v_{\varepsilon}(x)&=K(x)^{-\frac{1}{2}}e^{-|x|^2/8\sqrt{5}}\frac{\big({\varepsilon^2} N(N-2)\big) ^{\frac{N-2}{4}}}{\big(\varepsilon^{2}+|x'|^{2}+|x_{N}+\varepsilon x_{N}^{0} |^{2}\big)^{\frac{N-2}{2}}}, \quad N=3,
\end{align*}
for Theorem \ref{Th1.1} and the following test functions
\begin{align*}
\hat{u}_\varepsilon(x)&=K(x)^{-\frac{1}{2}}\phi(x)
\bigg(\frac{\varepsilon}{|x{'}|^{2}+|x_{N}+{\varepsilon}|^{2}}\bigg)^{\frac{N-2}{2}},\quad N\geq4,\\
\hat{v}_{\varepsilon}(x)&=K(x)^{-\frac{1}{2}}e^{-|x|^2/8\sqrt{5}}\bigg(\frac{\varepsilon}{|x{'}|^{2}+|x_{N}+{\varepsilon}|^{2}}\bigg)^{\frac{N-2}{2}},\quad N=3,
\end{align*}
for Theorem \ref{Th1.2}, where $\phi(x) \in C_{0}^{\infty}({\mathbb{R}}^{N},[0,1])$ is a cut-off function and
$x_{N}^{0}:={\sqrt{{N}/{(N-2)}}}$. Then we succeed to localize correctly the Mountain Pass level in the range where (PS) condition holds by performing some delicate estimates on the asymptotic behavior of functions $u_{{\varepsilon}}, \hat{u}_{{\varepsilon}}, v_{{\varepsilon}}, \hat{v}_{{\varepsilon}}$ as $\varepsilon \to 0$.

The rest of this paper is structured as follows.
The nonexistence result is preformed in Section \ref{S1}.
% by Pohozaev identity and a Hardy-type inequality.
In Section \ref{S2}, after a brief introduction of the preliminary results, we verify the Mountain Pass geometric conditions for $I^a_{\lambda, \mu}$ and establish the local compactness for the functional under the assumption \eqref{2.11} or \eqref{2.12}.
In Section \ref{S3}, we are devoted to providing careful estimates of $u_{\varepsilon}$ and $v_\varepsilon$, and then establish the existence result for equation \eqref{1.1} when $a=1$ by verifying  condition \eqref{2.11}. In Section \ref{S4}, we investigate the existence of a positive solution to equation \eqref{1.1} when $a=0$ by verifying the condition \eqref{2.12}. Some fine estimates of $\hat{u}_{\varepsilon}$ and $\hat{v}_\varepsilon$ are also given.

\vs{3mm}
\section{The nonexistence}\label{S1}
\vs{2mm}

In this section, we deal with the nonexistence result for equation \eqref{1.1}. To this end, we introduce a Pohozaev identity for problem \eqref{1.1} which was first proved by L. C. F. Ferreira in \cite {2021} for $\mu =0$ and extended by Y. Deng in \cite {0} for $\mu \in \mathbb{R} $. For  convenience,  we always assume that $N\geq3$, $\lambda, \mu\in\mathbb{R}$, $a\in\{0, 1\}$ and $2\leq q<2_*$ in this section.

\begin{Lem}\label{lem1.1} (Pohozaev identity) If $u\in C^2(\mathbb{R}^N_+)\cap X$ is a solution of equation \eqref{1.1}, then
\begin{equation*}
{\|\nabla{u}\|_{{L^{2}}(\mathbb{R}^{N}_{+})}^{2}}-a{\|{u}\|}_{L^{{2}^{*}}({\mathbb{R}}^{N}_{+})}^{2^{*}}
-{\|{u}\|}_{L^{{2}_{*}}({\mathbb{R}}^{N-1})}^{2_{*}}
=\Big(\lambda-\frac{N}{4}\Big)\|u\|_{{L^{2}}(\mathbb{R}^{N}_{+})}^{2}+
\mu{\|{u}\|}_{L^q({\mathbb{R}}^{N-1})}^{q}
\end{equation*}
and
\begin{equation*}
	\begin{aligned}
&\frac{N-2}{2}\Big({\|\nabla{u}\|_{{L^{2}}(\mathbb{R}^{N}_{+})}^{2}}
-a{\|{u}\|}_{L^{{2}^{*}}({\mathbb{R}}^{N}_{+})}^{2^{*}}
-{\|{u}\|}_{L^{{2}_{*}}({\mathbb{R}}^{N-1})}^{2_{*}}\Big)\\
&=\frac{\lambda N}{2}\|u\|_{{L^{2}}(\mathbb{R}^{N}_{+})}^{2}+\frac{\mu {(N-1)}}{q}{\|{u}\|}_{L^q({\mathbb{R}}^{N-1})}^{q}-\frac{1}{2}{\int_{\mathbb{R}^{N}_{+}}(x \cdot{\nabla u})^2dx}.
	\end{aligned}
\end{equation*}
\end{Lem}
\begin{proof}
 The identities follow from  Lemma 2.1  in \cite{0} by taking $f(u):=\lambda u+a{|u|^{2^*-2}u}$ and  $g(u):=\mu{|u|^{q-2}u}+{|u|^{{2}_{*}-2}u}$.
\end{proof}
%Now, we present a Hardy-type inequality

The following Lemmas are essential to the proof of Theorem \ref{Th1.0}.
\begin{Lem}\label{lem1.2}(\cite{2021}, Proposition 3.3).
	 For any $u\in X$, there holds
\begin{equation*}
\frac{N^2}{4}\int_{\mathbb{R}^{N}_{+}} u^2dx\leq\int_{\mathbb{R}^{N}_{+}} (x \cdot{\nabla u})^2dx.
\end{equation*}
\end{Lem}
\begin{Lem}\label{lem1.3}  Problem  \eqref{1.1} has no positive solution if $\mu \ge 0$ and $\lambda \ge \frac{N}{2}$.
\end{Lem}
\begin{proof}
Note that the first equation in \eqref{qw1} can be written as
$$-\mbox{div}(K(x)\nabla {u})= \lambda K(x)u.$$
Hence, if $\varphi_1>0$
is the first eigenfunction corresponding to $\lambda_1$, it must be a critical point of the functional
\begin{equation*}
	J(v):={\frac{1}{2}}{\|{v}\|}^{2}-{\frac{\lambda_1 }{2}}\|v\|_{L^{{2}}_{K}({\mathbb{R}}^{N}_{+})}^{2},\quad v\in X.
\end{equation*}

Suppose by contradiction that $u$ is a positive solution of problem \eqref{1.1}, it must be the critical point of the functional $J_{\lambda, \mu}$,
where
\begin{equation*}
	J_{\lambda, \mu}(u):={\frac{1}{2}}{\|{u}\|}^{2}-{\frac{\lambda }{2}}\|u\|_{L^{{2}}_{K}({\mathbb{R}}^{N}_{+})}^{2}
-{\frac{a }{2^{*}}}\|{u}\|_{L^{{2}^{*}}_{K}({\mathbb{R}}^{N}_{+})}^{2^{*}}
-{\frac{1 }{2_{*}}}\|{u}\|_{L^{{2}_{*}}_{K}({\mathbb{R}}^{N-1})}^{2_{*}}
-{\frac{\mu }{q}}{\|{u}\|}_{{L_{K}^{q}}({\mathbb{R}}^{N-1})}^{q}.\notag
\end{equation*}
Thus we have $\langle {J'_{\lambda, \mu}(u)},\varphi_1\rangle=0$. It follows from the fact  $\langle {J{'}(\varphi_1)},{u}\rangle =0$ that
\begin{equation*}
	\begin{aligned}
	(\lambda_1-\lambda)\int_{\mathbb{R}^{N}_{+}}K(x)u\varphi_1dx
&=a\int_{\mathbb{R}^{N}_{+}}K(x)|u|^{2^{*}-2}u\varphi_1dx
		+\int_{\mathbb{R}^{N-1}} K(x{'},0)|u|^{2_{*}-2}u\varphi_1dx'\\
&\quad+\mu\int_{\mathbb{R}^{N-1}}  K(x{'},0)|u|^{q-2}u\varphi_1dx'.
	\end{aligned}
\end{equation*}
We conclude from $a\in\{0,1\}, \mu\geq0$ and $\varphi_1>0$ that $\lambda<\frac{N}{2}$ if $u$ is a positive solution of equation \eqref{1.1}. This contradicts with the assumption $\lambda\ge \frac{N}{2}$.
\end{proof}

The proof of following Lemma is similar to Lemma \ref{lem1.3}, we omit details  here.
\begin{Lem}\label{lem1.4}  Problem  \eqref{1.1} has no positive solution if  $\lambda \ge 0$,  $\mu\ge \mu_1$ and $q=2$.
\end{Lem}

\vs{2mm}

{\bf Proof of Theorem \ref{Th1.0}.}
From Lemma \ref{lem1.1}, there holds
\begin{equation*}
\mu \Big(\frac{N-1}{q}-\frac{N-2}{2}\Big){\|{u}\|}_{L^q({\mathbb{R}}^{N-1})}^{q}
+\Big(\lambda+\frac{N(N-2)}{8}\Big)\|u\|_{{L^{2}}(\mathbb{R}^{N}_{+})}^{2}
=\frac{1}{2}{\int_{\mathbb{R}^{N}_{+}}(x \cdot{\nabla u})^2dx}.
\end{equation*}
Combining the above with Lemma \ref{lem1.2}, we conclude
\begin{equation*}
\mu\Big(\frac{N-1}{q}-\frac{N-2}{2}\Big){\|{u}\|}_{L^p({\mathbb{R}}^{N-1})}^{p}
+\Big(\lambda-\frac{N}{4}\Big)\|u\|_{{L^{2}}(\mathbb{R}^{N}_{+})}^{2}\geq0.
\end{equation*}
Then $u\equiv0$ if either $\mu=0, \lambda<{N}/{4}$ or $\mu<0, \lambda\leq{N}/{4}$. This, together with Lemma \ref{lem1.3} and Lemma \ref{lem1.4}, we complete  the proof of Theorem \ref{Th1.0}.
\qed
\vs{3mm}

\section{ Preliminaries and local compactness}\label{S2}
\vs{2mm}

In this section, we present some preliminary lemmas including the local compactness for $I^a_{\lambda, \mu}$ under the assumption \eqref{2.11} or \eqref{2.12} and some estimates for the test functions which are used to verify the assumptions \eqref{2.11} and \eqref{2.12}. The proofs of those lemmas can be found in the corresponding references.

Recall that the best constant of the Sobolev trace embedding $D^{1,2}({\mathbb{R}}^{N}_{+})\hookrightarrow {L^{{2}_{*}}}(\mathbb{R}^{N-1})$ (see \cite{1112} and \cite{18}) given by
\begin{equation}\label{2.8}
S_{0}:=\inf_{{u\in D^{1,2}({\mathbb{R}}^{N}_{+})}\setminus \{0\}}{\frac{{\|\nabla{u}\|_{{L^{2}}(\mathbb{R}^{N}_{+})}^{2}}}{{{\|u\|_{L^{{2}_{*}}({\mathbb{R}}^{N-1})}^{2}}}}}
\end{equation}
is achieved by
\begin{equation}\label{2.9}
{\hat{U}_{\varepsilon}(x)}=\bigg(\frac{\varepsilon}{|x{'}|^{2}+|x_{N}+{\varepsilon}|^{2}}\bigg)^{\frac{N-2}{2}},
\end{equation}
where $\varepsilon>0$. Moreover, by the same argument as Theorem 3.3 in \cite{01}, we have the following Lemmas.
\begin{Lem}\label{lem2.1}
	For any $\theta \in (0,1]$, the infimum
\begin{equation}
S_{\theta}:=\inf_{{u\in D^{1,2}({\mathbb{R}}^{N}_{+})}\setminus \{0\}}\frac{{\|\nabla{u}\|_{{L^{2}}(\mathbb{R}^{N}_{+})}^{2}}}{\theta {{\|{u}\|_{L^{{2}^{*}}({\mathbb{R}}^{N}_{+})}^{2}}}+(1-\theta){{\|{u}\|_{L^{{2}_{*}}({\mathbb{R}}^{N-1})}^{2}}}}\\\label{2.1}
\end{equation}
is achieved by
%the function ${\varphi(x)}=(1+|x{'}|^{2}+|x_{N}+x_{N}^{0}|^{2})^{\frac{2-N}{2}}$, or after rescaling
 the function
\begin{equation}
{\varphi_{\varepsilon}(x)}=\bigg(\frac{\varepsilon}{{\varepsilon}^{2}+|x{'}|^{2}+|x_{N}+{\varepsilon}x_{N}^{0}|^{2}}\bigg)^{\frac{N-2}{2}},\notag
\end{equation}
where $\varepsilon>0$, $x'\in \mathbb{R}^{N-1}$ and $ x_{N}^{0}$ is a constant depending only on $\theta$ and $N$.
\end{Lem}
%\smallskip

For $\tau\geq0$, set
\begin{equation*}\label{2.2}
{\varphi}_{{\varepsilon},{\tau}}(x):=k_N\bigg(\frac{{\varepsilon}}{{\varepsilon}^{2}+|x{'}|^{2}+|x_{N}+{\varepsilon}{\tau}x_{N}^{0}|^{2}}\bigg)^{\frac{N-2}{2}},
\end{equation*}
where
\begin{equation} \label{2.02}
k_N=\big({\sqrt{N(N-2)}}\big)^{\frac{N-2}{2}}\quad\mbox{and}\quad x_{N}^{0}:={\sqrt{\frac{N}{N-2}}}.
\end{equation}
Simple computations indicate that ${\varphi}_{{\varepsilon},{\tau}}$ satisfies
\begin{equation}\label{2.3}
	\left\{
	\begin{aligned}
		-\Delta {u}&=u^{{2}^{*}-1}& \ \ \mbox{in} \ \ \ {{\mathbb{R}}^{N}_{+}}, \ \ \\
		\frac{{\partial u}}{{\partial n}}&=\tau u^{{2}^{*}-1} \ & \ \ \mbox{on}\ \mathbb{R}^{N-1}.
	\end{aligned}
	\right.
\end{equation}
Let
\begin{equation*}
 {\theta}:=\frac{{{\|{{\varphi}_{{\varepsilon},{\tau}}}\|_{L^{{2}^{*}}({\mathbb{R}}^{N}_{+})}^{2^{*}-2}}}}
{{{\|{{\varphi}_{{\varepsilon},{\tau}}}\|_{L^{{2}^{*}}({\mathbb{R}}^{N}_{+})}^{2^{*}-2}}}+{\tau}{{\|{{{\varphi}_{{\varepsilon},{\tau}}}}\|_{L^{{2}_{*}}({\mathbb{R}}^{N-1})}^{2_{*}-2}}}}, \end{equation*}
which is independent of $\varepsilon$. Then ${\varphi}_{{\varepsilon},{\tau}}(x)$ reaches the infimum $S_{\theta}$.

Denote
\begin{equation*}
\Phi(u):={\frac{1}{2}}{\|\nabla{u}\|_{{L^{2}}(\mathbb{R}^{N}_{+})}^{2}}-{\frac{1 }{2^{*}}}{\|{u_{+}}\|}_{L^{{2}^{*}}({\mathbb{R}}^{N}_{+})}^{2^{*}}-{\frac{1 }{2_{*}}}{\|{u_{+}}\|}_{L^{{2}_{*}}({\mathbb{R}}^{N-1})}^{2_{*}}  \label{2.4}
\end{equation*}
and set
\begin{equation}
A:={\inf\limits_{{u\in D^{1,2}({\mathbb{R}}^{N}_{+})}\setminus\{0\}}}\ {\sup\limits_{t>0}}\ \Phi(tu).\label{2.5}
\end{equation}
\begin{Lem}\label{lem2.2}(\cite{010}, Lemma 2.3).
	The infimum $A$ is achieved by
${\varphi}_{{\varepsilon},{1}}$.
%and
%\begin{equation*}
%A={\lambda^{-\frac{N-2}{2}}}\bigg({\frac{1}{2}}\|\nabla{{\varphi}_{{\varepsilon},{1}}}\|_{{L^{2}}(\mathbb{R}^{N}_{+})}^{2}-
%{\frac{1}{2^{*}}}{\|{\varphi}_{{\varepsilon},{1}}\|}_{L^{{2}^{*}}({\mathbb{R}}^{N}_{+})}^{2^{*}}
%-{\frac{1} {2_{*}}}{\|{\varphi}_{{\varepsilon},{1}}\|}_{L^{{2}_{*}}({\mathbb{R}}^{N-1})}^{2_{*}}  \bigg).
%\end{equation*}
\end{Lem}
Let %For simplicity, we define
\begin{equation}\label{2.6}
{U}_{\varepsilon}:={\varphi}_{{\varepsilon},{1}},~
K_{1}:={\|\nabla{{U}_{\varepsilon}}\|_{{L^{2}}(\mathbb{R}^{N}_{+})}^{2}},~
K_{2}:={\|{{U_{\varepsilon}}}\|}_{L^{{2}^{*}}({\mathbb{R}}^{N}_{+})}^{2^{*}},~
K_{3}:={\|{{U_{\varepsilon}}}\|}_{L^{{2}_{*}}({\mathbb{R}}^{N-1})}^{2_{*}}.
\end{equation}
It follows from Lemma \ref{lem2.2} and \eqref{2.3} that
\begin{equation}\label{2.7}
~A={\frac{K_{1}}{2}}-{\frac{K_{2} }{2^{*}}}-{\frac{K_{3} }{2_{*}}}, \quad\\%\label{2.6}
~K_{1}-K_{2}-K_{3}=0.
\end{equation}

On the other hand, the existence of the embeddings $X \hookrightarrow {L^{2^{*}}_{K}}({\mathbb{R}}^{N}_{+})$ and $X \hookrightarrow {L^{2_{*}}_{K}}({\mathbb{R}}^{N-1})$ allows us to define
\begin{equation}
S^{K}:=\inf_{u\in X\setminus \{0\}}\frac{{\|\nabla{u}\|_{{L_{K}^{2}}(\mathbb{R}^{N}_{+})}^{2}}} {\|{u}\|_{L^{{2}_{*}}_{K}(\mathbb{R}^{N-1})}^{2}}\notag
\end{equation}
and
\begin{equation}
S_{\theta}^{K}:=\inf_{u\in X\setminus \{0\}}\frac{{\|\nabla{u}\|_{{L_{K}^{2}}(\mathbb{R}^{N}_{+})}^{2}}}{\theta {{\|{u}\|_{L^{{2}^{*}}_{K}({\mathbb{R}}^{N}_{+})}^{2}}}+(1-\theta){{\|{u}\|_{L^{{2}_{*}}_{K}({\mathbb{R}}^{N-1})}^{2}}}},\notag
\end{equation}
where $\theta \in (0,1]$. It is worth mentioning that $S^{K}=S_{0}$ and  $S_{\theta}^{K}=S_{\theta}$ for any $\theta \in (0,1]$ (see  \cite{0} \cite{2021}).

Proceeding as done in  \cite{2021} or \cite{116}, we can easily verify that $I^a_{\lambda, \mu}$ has a Mountain Pass geometry structure.

\begin{Lem}\label{lem2.5}
	Let $N\geq3, a\in\{0,1\}$. The functional $I^a_{\lambda, \mu}$ has a Mountain Pass structure in each of the following cases:
\vskip 0.2cm
	
	$(1)$  $ \lambda\in(0, \frac{N}{2})$ and $\mu=0$;
\vskip 0.2cm
	
	$(2)$ $ \lambda\in[0, \frac{N}{2})$, $\mu>0$ and $q\in(2,2_*)$;
\vskip 0.2cm
	
$(3)$ $\mu \in (0, \mu_1 -\frac {2\mu_1}{N}\lambda )$, $\lambda \in [0, \frac {N} {2})$ if $q=2$,
\vskip 0.2cm

\noindent
where $\mu_1$ is defined by \eqref{2}.
\end{Lem}
As a consequence of Lemma \ref{lem2.5}, we get that
\begin{equation*}
c^a_{\lambda, \mu}:={\inf\limits_{\gamma \in \Gamma}}{\max\limits_{t\in[0,1]}}\ I^a_{\lambda, \mu}(\gamma(t))>0,
\end{equation*}
where
\begin{equation}
\Gamma:=\{\gamma \in C([0,1],X):\gamma(0)=0, I^a_{\lambda, \mu}(\gamma(1))<0\}.\notag
\end{equation}

The following Lemmas provide the intervals where the $(PS)$ condition holds under different specific situations for $I^a_{\lambda, \mu}(u)$.

\begin{Lem}\label{lem2.6}
Under the assumptions of Lemma \ref{lem2.5},
the functional $I^1_{\lambda, \mu}$ satisfies the $(PS)_{c}$ condition at the level $c^1_{\lambda, \mu}$ if
 \begin{equation}
c^1_{\lambda, \mu}<A,   \quad\\ \label{2.11}
\end{equation}
where $A$ is given by \eqref{2.5}.
\end{Lem}
The proof of this Lemma is similar to Lemma 3.5 in \cite{0}, we omit the details here. Naturally, we also have a local compactness result for
the functional $I^0_{\lambda, \mu}$ (see proposition 5.1 in  \cite{2021}). %with $a=0$.

\begin{Lem}\label{lem2.7}
Under the hypotheses of Lemma \ref{lem2.5},
the functional $I^0_{\lambda, \mu}(u)$ satisfies the $(PS)_{c}$ condition at the level $c^0_{\lambda, \mu}$ if
 \begin{equation}
c^0_{\lambda, \mu}<\frac{1}{2(N-1)}S_0^{N-1},   \quad\\ \label{2.12}
\end{equation}
where $S_0$ is given by \eqref{2.8}.
\end{Lem}

We devote the rest of the paper to verify that \eqref{2.11} or \eqref{2.12} holds under different conditions on $\lambda, \ \mu$ and $q$, and then complete the proof of Theorem \ref{Th1.1} and Theorem \ref{Th1.2}. To this end, we define the functions
\begin{eqnarray}
&& u_{\varepsilon}(x):=K(x)^{-\frac{1}{2}}\phi(x)U_{{\varepsilon}}(x),\label{b1}\\
&& \hat{u}_\varepsilon(x):=K(x)^{-\frac{1}{2}}\phi(x)\hat{U}_{{\varepsilon}}(x),\label{b2}
\end{eqnarray}
where the cut-off function $\phi \in C_0^{\infty}({\mathbb{R}}^{N}_{+},[0,1])$, $\phi \equiv 1$ in $B_{1}(0)\cap {\mathbb{R}}^{N}_{+}$, $\phi\equiv 0$ in $\overline{{{\mathbb{R}}^{N}_{+}}}\setminus B_{2}(0)$,
$\hat U_\varepsilon$ and ${U}_{{\varepsilon}}$ are defined in \eqref{2.9} and \eqref{2.6}. Then we have the following estimates for $u_\varepsilon $ (See Lemmas 4.1-4.3 in \cite{0} for details).
\begin{Lem}\label{lem2.8}
Suppose that $N\geq3$. As $\varepsilon\rightarrow0$, we have
\begin{equation*}
 \begin{aligned}
&{\|u_{{\varepsilon}}\|}^{2}=
\begin{cases}
K_{1}+\alpha_{N}{{\varepsilon}}^{2}+o({{\varepsilon}}^{2}),~&~~N\geq 5,\\[1mm]
K_{1}+\frac{k_{4}^{2}\omega_4}{2}{{\varepsilon}}^{2}|\ln{\varepsilon}|+O({\varepsilon}^{2}),~&~~N=4,\\[1mm]
K_1+O({\varepsilon}),~&~~N=3,\\[1mm]
\end{cases}\\
&{\|u_{{\varepsilon}}\|}_{L^{{2}^{*}}_{K}({\mathbb{R}}^{N}_{+})}^{2^{*}}=K_{2}-\beta_{N}{{\varepsilon}}^{2}+o({\varepsilon}^{2}),\\
&{\|u_{{\varepsilon}}\|}_{L^{{2}_{*}}_{K}({\mathbb{R}}^{N-1})}^{2_{*}}
=
\begin{cases}
K_{3}-\gamma_{N}{{\varepsilon}}^{2}+o({\varepsilon}^{2}),~&~~N\geq4,\\[1mm]
K_{3}+O({\varepsilon}^2|\ln\varepsilon|),~&~~N=3,\\[1mm]
\end{cases}
 \end{aligned}
\end{equation*}
where $k_N, K_1, K_2, K_3$ are given by \eqref{2.02}, \eqref{2.6}, $\omega_{4}$ is the area of unit sphere in ${\mathbb{R}^{4}}$ and
\begin{eqnarray}
&&\alpha_{N}=\frac{(N-2)k_{N}^{2}}{2}\int_{\mathbb{R}^{N}_{+}}\frac{|y{'}|^{2}+{y_{N}(y_{N}+x_{N}^{0})}}
{{{{(1+|y{'}|^{2}+|y_{N}+x_{N}^{0}|^{2})^{N-1}}}}}dy, \label{a1}\\
&&\beta_{N}=\frac{k_{N}^{2^{*}}}{2(N-2)}
\int_{\mathbb{R}^{N}_{+}}\frac{|y|^{2}}
{{{{(1+|y{'}|^{2}+|y_{N}+x_{N}^{0}|^{2})^{N}}}}}dy,\label{a2}\\
&&\gamma_N=\frac{k_{N}^{2_{*}}}{4(N-2)}\int_{\mathbb{R}^{N-1}}\frac{|y{'}|^{2}}
{{{{(1+|y{'}|^{2}+|x_{N}^{0}|^{2})^{N-1}}}}}dy{'}.\label{a3}
\end{eqnarray}
\end{Lem}

The next Lemma deals with the asymptotic behavior of the $X$-norm, $L^{{2}}_{K}(\mathbb{R}^{N}_{+})$-norm and $L^{{2}_*}_{K}(\mathbb{R}^{N-1})$-norm for $\hat{u}_{\varepsilon}$ as $\varepsilon \rightarrow0$. Define the constants
\begin{equation}\label{aa1}
A_N:={\|\nabla{{\hat{U}}_{\varepsilon}}\|_{{L^{2}}(\mathbb{R}^{N}_{+})}^{2}},~
B_N:={\|{{\hat{U}_{\varepsilon}}}\|}_{L^{{2}_{*}}({\mathbb{R}}^{N-1})}^{2}.
\end{equation}
Then $A_N/B_N=S_0$ (see \cite{18}). The following estimates for $\hat{u}_{\varepsilon}$ were established in  \cite{2021} (See Lemma 5.3 and Lemma 5.4 in \cite{2021} for details).
\begin{Lem}\label{lem2.9}
Suppose that $N\geq7$. As $\varepsilon\rightarrow0$, we have
\begin{equation*}
 \begin{aligned}
&{\|\hat{u}_{{\varepsilon}}\|}^{2}=
A_{N}+\hat{\alpha}_{N}{{\varepsilon}}^{2}+O({{\varepsilon}}^{4}),\\
&{\|\hat{u}_{{\varepsilon}}\|}_{L^{{2}}_{K}({\mathbb{R}}^{N}_{+})}^{2}=\hat{d}_{N}{{\varepsilon}}^{2}+O({\varepsilon}^{N-2}),\\
&{\|\hat{u}_{\varepsilon}\|}_{L^{{2}_*}_{K}({\mathbb{R}}^{N-1})}^{2_{*}}
={B_{N}}^{{2_*}/{2}}-\hat{\gamma}_{N}{{\varepsilon}}^{2}+o({\varepsilon}^{2}),
 \end{aligned}
\end{equation*}
where %$A_N={\|\nabla{{\hat{U}}_{\varepsilon}}\|_{{L^{2}}(\mathbb{R}^{N}_{+})}^{2}},
%B_N={\|{{\hat{U}_{\varepsilon}}}\|}_{L^{{2}_{*}}({\mathbb{R}}^{N-1})}^{2}$,
\begin{eqnarray}
&&\hat{\alpha}_{N}=\frac{\omega_{N-1}(N-2)}{4(N-4)}\Bigg(B\Big(\frac{N+1}{2}, \frac{N-3}{2}\Big)+\frac{1}{N-3}B\Big(\frac{N-1}{2}, \frac{N-1}{2}\Big)\Bigg), \label{a5}\\
&&\hat{d}_{N}=\frac{\omega_{N-1}}{2(N-4)}B\Big(\frac{N-1}{2}, \frac{N-3}{2}\Big), \label{a66}\\
&&\hat{\gamma}_N=\frac{\omega_{N-1}}{8(N-2)}B\Big(\frac{N+1}{2}, \frac{N-3}{2}\Big), \label{a7}
\end{eqnarray}
and $\omega_{N-1}$ is the area of unit sphere in ${\mathbb{R}^{N-1}}$.
\end{Lem}

\vs{2mm}
\section{The proof of Theorem \ref{Th1.1}}\label{S3}
\vs{2mm}
%Proofs of Theorems \ref{Th1.1}-\ref{Th1.3}

 In this section, we are going to verify  \eqref{2.11} and then complete the proof of Theorem \ref{Th1.1}. We always assume that $N\geq3$, $\lambda\geq0, a=1$,
$\mu\geq0, \lambda+\mu>0$ and $2\leq q<2_*$ in this section.

Set
\begin{equation*}\label{030}
c_{\lambda, \mu}^{*}:=\inf\limits_{u\in X}\big\{\sup\limits_{t>0}I^1_{\lambda, \mu}(tu):u\geq0 ~\text{and}~ u\not \equiv 0\big\}.
\end{equation*}
It is clear that $c^1_{\lambda, \mu}\leq c_{\lambda, \mu}^{*}$. Hence, to verify  \eqref{2.11} in Lemma \ref{lem2.6}, it is sufficient to verify that \begin{equation}\label{3.0}
c_{\lambda, \mu}^{*}<A,
\end{equation}
where $A$ is given by \eqref{2.5}.

\subsection{Verification of the condition (4.1) when $\lambda>0$ and $\mu=0$ }\label{S3.1}

This subsection is devote to verifying condition \eqref{3.0} under the hypothesis  $(1)$ in Theorem  \ref{Th1.1}. Notice that the energy function associated to \eqref{1.5} with $ a=1$ and $\mu=0$ is
\begin{equation}
	I^1_{\lambda, 0}(u):={\frac{1}{2}}{\|{u}\|}^{2}-{\frac{\lambda }{2}}\|u_+\|_{L^{{2}}_{K}({\mathbb{R}}^{N}_{+})}^{2}
-{\frac{1 }{2^{*}}}\|{u_{+}}\|_{L^{{2}^{*}}_{K}({\mathbb{R}}^{N}_{+})}^{2^{*}}
-{\frac{1 }{2_{*}}}\|{u_{+}}\|_{L^{{2}_{*}}_{K}({\mathbb{R}}^{N-1})}^{2_{*}}, \quad u\in X.\notag
\end{equation}
%Firstly, we provide the following delicate estimates.
We start with the asymptotic estimate of the ${L^{2}_{K}({\mathbb{R}}^{N}_{+})}$-norm of $u_{\varepsilon}$ as  $\varepsilon\rightarrow0$.
\begin{Lem}\label{lem3.1}
 As $\varepsilon\rightarrow0$, one has
%There holds
\begin{equation*}
{\|u_{{\varepsilon}}\|}_{L^{2}_{K}({\mathbb{R}}^{N}_{+})}^{2}=
\begin{cases}
d_{N}{{\varepsilon}}^{2}+o({{\varepsilon}}^{2}),~&~~N\geq5,\\[1mm]
\frac{{k_{4}^{2}\omega_4}}{2}\varepsilon^2|\ln\varepsilon|+O({{\varepsilon}}^{2}),~&~~N=4,\\[1mm]
O({\varepsilon}),~&~~N=3,\\[1mm]
\end{cases}
\end{equation*}
where %$\varepsilon>0$ is sufficiently small,
\begin{equation}\label{a4}
d_{N}=k_{N}^{2}\int_{\mathbb{R}^{N}_{+}}
\frac{1}{{{(1+|y{'}|^{2}+|y_{N}+x_{N}^{0}|^{2})^{N-2}}}}dy,
\end{equation}
$k_{N}$ is given by \eqref{2.02} and $\omega_{4}$ is the area of unit sphere in ${\mathbb{R}^{4}}$.
\end{Lem}

\begin{proof}
For $N\geq5$, by using the change of variables $y={x}/{{\varepsilon}}$, we have
\begin{equation}\label{z4}
  \begin{aligned}
{\|u_{{\varepsilon}}\|}_{L^{2}_{K}({\mathbb{R}}^{N}_{+})}^{2}
&=\int_{\mathbb{R}^{N}_{+}}{\phi}^{2}U_{\varepsilon}^{2}dx
=\int_{\mathbb{R}^{N}_{+}}U_{\varepsilon}^{2}dx+O(\varepsilon^{N-2})\\
&={\varepsilon}^{N-2}k_{N}^{2}\int_{\mathbb{R}^{N}_{+}}
\frac{1}{{{({\varepsilon}^{2}+|x{'}|^{2}+|x_{N}+{\varepsilon}x_{N}^{0}|^{2})^{N-2}}}}dx+O({\varepsilon}^{N-2})\\
&={\varepsilon}^{2}k_{N}^{2}\int_{\mathbb{R}^{N}_{+}}
\frac{1}{{{(1+|y{'}|^{2}+|y_{N}+x_{N}^{0}|^{2})^{N-2}}}}dy+O(\varepsilon^{N-2}).
 \end{aligned}
\end{equation}

For $N=3, 4$, we can similarly calculate
\begin{equation}\label{zz5}
  \begin{aligned}
{\|u_{{\varepsilon}}\|}_{L^{2}_{K}({\mathbb{R}}^{N}_{+})}^{2}
&=\int_{\mathbb{R}^{N}_{+}}{\phi}^{2}U_{\varepsilon}^{2}dx=\int_{B_{2}^{+}}U_\varepsilon^{2}dx+\int_{ B_{2}^{+}\setminus{B_{1}^{+}}}({\phi}^{2}-1)U_\varepsilon^{2}dx\\
&=\int_{B_{2}^{+}}U_\varepsilon^{2}dx+O({\varepsilon}^{N-2})\\
&={\varepsilon}^{2}k_{N}^{2}\int_{{{B^{+}_{{2}/{\varepsilon}}}}}
\frac{1}{{({1+|y{'}|^{2}+|y_{N}+x_{N}^{0}|^{2})^{N-2}}}}dy+O({\varepsilon}^{N-2})\\
&={\varepsilon}^{2}k_{N}^{2}\int_{B^{+}_{{2}/{\varepsilon}}\setminus{B_{1}^{+}}}
\frac{1}{{({1+|y{'}|^{2}+|y_{N}+x_{N}^{0}|^{2})^{N-2}}}}dy+O({\varepsilon}^{2})+O({\varepsilon}^{N-2})
  \end{aligned}
\end{equation}
as $\varepsilon  \rightarrow 0$, where $B_{r}^{+}:=B_{r}(0)\cap {\mathbb{R}}^{N}_{+}$ for any $r>0$. If $N=4$, we have
\begin{equation*}
\begin{aligned}
0&<\int_{B_{2/\varepsilon}^{+}\setminus{B_{1}^{+}}}\frac{1}{|y|^{4}}dy-\int_{B^{+}_{{2}/{\varepsilon}}\setminus{B_{1}^{+}}}
\frac{1}{{({1+|y{'}|^{2}+|y_{4}+x_{4}^{0}|^{2})^{2}}}}dy\\
&=\int_{B_{2/\varepsilon}^{+}\setminus{B_{1}^{+}}}\frac{{(1+|y{'}|^{2}+|y_{4}+x_{4}^{0}|^{2})^{2}}-|y|^{4}}{|y|^{4}{(1+|y{'}|^{2}+|y_{4}+x_{4}^{0}|^{2})^{2}}}dy=O_\varepsilon(1),
\end{aligned}
\end{equation*}
and thus
\begin{equation}\label{zzz2}
\begin{aligned}
\int_{B^{+}_{{2}/{\varepsilon}}\setminus{B_{1}^{+}}}
\frac{1}{{({1+|y{'}|^{2}+|y_{4}+x_{4}^{0}|^{2})^{2}}}}dy
&=\int_{B_{2/\varepsilon}^{+}\setminus{B_{1}^{+}}}\frac{1}{|y|^{4}}dy+O_\varepsilon(1)\\
&=\frac{\omega_4}{2}\int_{1}^{{{2}/\varepsilon}}r^{-1}dr+O_\varepsilon(1)\\
&=\frac{\omega_4}{2}(|\ln \varepsilon|+\ln 2)+O_\varepsilon(1),
\end{aligned}
\end{equation}
where $\omega_4$ is the area of unit sphere in $\mathbb{R}^{4}$ and  $O_\varepsilon(1)$ is a constant associated with $\varepsilon$. If $N=3$, we deduce that
\begin{equation}\label{zzz3}
\begin{aligned}
\int_{B^{+}_{{2}/{\varepsilon}}\setminus{B_{1}^{+}}}
\frac{1}{{(1+|y{'}|^{2}+|y_{3}+x_{3}^{0}|^{2})}}dy
=O\Big(\int_{B_{2/\varepsilon}^{+}\setminus{B_{1}^{+}}}\frac{1}{|y|^{2}}dy\Big)
=O(\varepsilon^{-1}).
\end{aligned}
\end{equation}
It follows from \eqref{zz5}-\eqref{zzz3} that
\begin{equation}\label{z5}
{\|u_{{\varepsilon}}\|}_{L^{2}_{K}({\mathbb{R}}^{N}_{+})}^{2}=
\begin{cases}
\frac{{k_{4}^{2}\omega_4}}{2}\varepsilon^2|\ln\varepsilon|+O({\varepsilon}^2),~&~~N=4,\\[1mm]
O({\varepsilon}),~&~~N=3.\\[1mm]
\end{cases}
\end{equation}

Hence, the proof is completed from \eqref{z4} and \eqref{z5}.
\end{proof}

To estimate the Mountain Pass level for $N=3$, we need more delicate estimates. To this end, we define
\begin{equation}\label{xxx1}
\psi(x):=e^{-|x|^2/8\sqrt{5}}
\end{equation}
and
\begin{equation*}\label{z6} v_{\varepsilon}(x):=K(x)^{-\frac{1}{2}}\psi(x)U_{{\varepsilon}}(x),
\end{equation*}
where $\varepsilon>0$ and $U_\varepsilon$ is given  by \eqref{2.6}. Then we have the following lemma.
\begin{Lem}\label{lem3.2}
 For small $\varepsilon > 0$, we have
\begin{align*}
&{\|v_{{\varepsilon}}\|}^{2}<K_1+\frac{(3+\sqrt{5})\sqrt{3}\varepsilon}{4} \int_{\mathbb{R}^{3}_{+}}
\frac{\psi^2}{|x|^{2}}dx,\\
&{\|v_{{\varepsilon}}\|}_{L^{2}_{K}({\mathbb{R}}^{3}_{+})}^{2}> \sqrt{3}\varepsilon\int_{\mathbb{R}^{3}_{+}}
\frac{\psi^2}{|x|^{2}}dx-d_1\varepsilon^2|\ln\varepsilon|-d_2\varepsilon^2,
\end{align*}
where $K_1, \psi$ are given by \eqref{2.6}, \eqref{xxx1} and $d_1$, $d_2$ are positive constants.
\end{Lem}
\begin{proof}
By the definition of $v_{{\varepsilon}} $, we directly compute that
  \begin{align*}
{\|v_{{\varepsilon}} \|}^{2}
&=\sqrt{3}\varepsilon\int_{\mathbb{R}^{3}_{+}}\frac{|\nabla{\psi}|^{2}}{(\varepsilon^2+|x{'}|^{2}+|x_{3}+\sqrt{3}{\varepsilon}|^{2})}dx\\
&\quad-2\sqrt{3}\varepsilon\int_{\mathbb{R}^{3}_{+}}\frac{\psi{\nabla{\psi}\cdot{(x{'},x_3+\sqrt{3}\varepsilon)}}}{(\varepsilon^2+|x{'}|^{2}+|x_{3}+\sqrt{3}{\varepsilon}|^{2})^2}dx\\
&\quad-\frac{\sqrt{3}\varepsilon}{2}\int_{\mathbb{R}^{3}_{+}}\frac{\psi(x\cdot{\nabla{\psi}})}{(\varepsilon^2+|x{'}|^{2}+|x_{3}+\sqrt{3}{\varepsilon}|^2)}dx+\int_{\mathbb{R}^{3}_{+}}{\psi}^{2}|\nabla{{{U}_{\varepsilon}}}|^{2}dx\\
&\quad+\frac{\sqrt{3}\varepsilon}{2}\int_{\mathbb{R}^{3}_{+}}\frac{{\psi}^2\big(|x{'}|^2+x_3(x_3+\sqrt{3}\varepsilon)\big)}{(\varepsilon^2+|x{'}|^{2}+|x_{3}+\sqrt{3}{\varepsilon}|^2)^2}dx\\
&\quad+\frac{\sqrt{3}\varepsilon}{16}\int_{\mathbb{R}^{3}_{+}}\frac{{\psi}^2|x|^2}{(\varepsilon^2+|x{'}|^{2}+|x_{3}+\sqrt{3}{\varepsilon}|^2)}dx\\
&<K_1+\sqrt{3}\varepsilon\int_{\mathbb{R}^{3}_{+}}\frac{|\nabla{\psi}|^{2}}{(\varepsilon^2+|x{'}|^{2}+|x_{3}+\sqrt{3}{\varepsilon}|^{2})}dx\\
&\quad-2\sqrt{3}\varepsilon\int_{\mathbb{R}^{3}_{+}}\frac{\psi{\nabla{\psi}\cdot{(x{'},x_3+\sqrt{3}\varepsilon)}}}{(\varepsilon^2+|x{'}|^{2}+|x_{3}+\sqrt{3}{\varepsilon}|^{2})^2}dx\\
&\quad-\frac{\sqrt{3}\varepsilon}{2}\int_{\mathbb{R}^{3}_{+}}\frac{\psi(x\cdot{\nabla{\psi}})}{(\varepsilon^2+|x{'}|^{2}+|x_{3}+\sqrt{3}{\varepsilon}|^2)}dx\\
&\quad+\frac{\sqrt{3}\varepsilon}{2}\int_{\mathbb{R}^{3}_{+}}\frac{{\psi}^2}{(\varepsilon^2+|x{'}|^{2}+|x_{3}+\sqrt{3}{\varepsilon}|^2)}dx+\frac{\sqrt{3}\varepsilon}{16}\int_{\mathbb{R}^{3}_{+}}{\psi}^2dx.\\
  \end{align*}
Since
\begin{equation*}
\nabla\psi(x)=-\frac{1}{4\sqrt{5}}\psi x,
\end{equation*}
one has
\begin{equation}\label{w1}
  \begin{aligned}
{\|v_{{\varepsilon}} \|}^{2}
%&<K_1+\frac{\sqrt{3}\varepsilon}{80}\int_{\mathbb{R}^{3}_{+}}\frac{{\psi}^{2}|x|^2}{(\varepsilon^2+|x{'}|^{2}+|x_{3}+\sqrt{3}{\varepsilon}|^{2})}dx\\
%&\quad+\frac{\sqrt{3}\varepsilon}{2\sqrt{5}}\int_{\mathbb{R}^{3}_{+}}\frac{\psi^2(|x{'}|^2+x_3(x_3+\sqrt{3}\varepsilon))}{(\varepsilon^2+|x{'}|^{2}+|x_{3}+\sqrt{3}{\varepsilon}|^{2})^2}dx\\
%&\quad+\frac{\sqrt{3}\varepsilon}{8\sqrt{5}}\int_{\mathbb{R}^{3}_{+}}\frac{\psi^2|x|^2}{(\varepsilon^2+|x{'}|^{2}+|x_{3}+\sqrt{3}{\varepsilon}|^2)}dx\\
%&\quad+\frac{\sqrt{3}\varepsilon}{2}\int_{\mathbb{R}^{3}_{+}}\frac{{\psi}^2}{(\varepsilon^2+|x{'}|^{2}+|x_{3}+\sqrt{3}{\varepsilon}|^2)}dx\\
%&\quad+\frac{\sqrt{3}\varepsilon}{16}\int_{\mathbb{R}^{3}_{+}}{\psi}^2dx\\
%&<K_1+\frac{\sqrt{3}\varepsilon}{80}\int_{\mathbb{R}^{3}_{+}}{\psi}^{2}dx
%+\frac{\sqrt{3}\varepsilon}{2\sqrt{5}}\int_{\mathbb{R}^{3}_{+}}\frac{\psi^2}{(\varepsilon^2+|x{'}|^{2}+|x_{3}+\sqrt{3}{\varepsilon}|^{2})}dx\\
%&\quad+\frac{\sqrt{3}\varepsilon}{8\sqrt{5}}\int_{\mathbb{R}^{3}_{+}}\psi^2dx+\frac{\sqrt{3}\varepsilon}{2}\int_{\mathbb{R}^{3}_{+}}\frac{{\psi}^2}{(\varepsilon^2+|x{'}|^{2}+|x_{3}+\sqrt{3}{\varepsilon}|^2)}dx\\
%&\quad+\frac{\sqrt{3}\varepsilon}{16}\int_{\mathbb{R}^{3}_{+}}{\psi}^2dx\\
&<K_1+\frac{(3+\sqrt{5})\sqrt{3}\varepsilon}{40}\int_{\mathbb{R}^{3}_{+}}{\psi}^{2}dx\\
&\quad+\frac{(5+\sqrt{5})\sqrt{3}\varepsilon}{10}\int_{\mathbb{R}^{3}_{+}}\frac{\psi^2}{(\varepsilon^2+|x{'}|^{2}+|x_{3}+\sqrt{3}{\varepsilon}|^{2})}dx.
  \end{aligned}
\end{equation}
Note that
\begin{equation}\label{w2}
  \begin{aligned}
\int_{\mathbb{R}^{3}_{+}}{\psi}^2dx
&=\int_{\mathbb{R}^{3}_{+}}e^{-\frac{|x|^2}{4\sqrt{5}}}dx
=\frac{\omega_3}{2}\int_{0}^{\infty}e^{-\frac{r^2}{4\sqrt{5}}}r^{2}dr
=\sqrt{5}\omega_3\int_{0}^{\infty}e^{-\frac{r^2}{4\sqrt{5}}}dr\\
&=2\sqrt{5}\int_{\mathbb{R}^{3}_{+}}\frac{e^{-\frac{|x|^2}{4\sqrt{5}}}}{|x|^2}dx
=2\sqrt{5}\int_{\mathbb{R}^{3}_{+}}\frac{\psi^2}{|x|^2}dx,
  \end{aligned}
\end{equation}
where $\omega_3$ is the area of unit sphere in $\mathbb{R}^{3}$.
It follows from  \eqref{w1} and \eqref{w2} that
\begin{equation*}
{\|v_{{\varepsilon}}\|}^{2}<K_1+\frac{(3+\sqrt{5})\sqrt{3}\varepsilon}{4} \int_{\mathbb{R}^{3}_{+}}
\frac{e^{-\frac{|x|^2}{4\sqrt{5}}}}{|x|^{2}}dx.
\end{equation*}

On the other hand, we have
\begin{equation}\label{A1}
  \begin{aligned}
{\|v_{{\varepsilon}}\|}_{L^{2}_{K}({\mathbb{R}}^{3}_{+})}^{2}
&=\int_{\mathbb{R}^{3}_{+}}{\psi}^{2}U_{\varepsilon}^{2}dx
=\sqrt{3}\varepsilon\int_{\mathbb{R}^{3}_{+}}\frac{e^{-\frac{|x|^2}{4\sqrt{5}}}}{(\varepsilon^2+|x{'}|^{2}+|x_{3}+\sqrt{3}{\varepsilon}|^2)}dx.
 \end{aligned}
\end{equation}
It is clear that
\begin{align*}
0&<\notag
\int_{\mathbb{R}^{3}_{+}}\frac{e^{-\frac{|x|^2}{4\sqrt{5}}}}{|x|^2}dx-\notag
\int_{\mathbb{R}^{3}_{+}}\frac{e^{-\frac{|x|^2}{4\sqrt{5}}}}
{(\varepsilon^2+|x{'}|^{2}+|x_{3}+\sqrt{3}{\varepsilon}|^2)}dx\\ \notag
&=\int_{\mathbb{R}^{3}_{+}}\frac{e^{-\frac{|x|^2}{4\sqrt{5}}}(\varepsilon^2+|x{'}|^{2}+|x_{3}+\sqrt{3}{\varepsilon}|^2-|x|^2)}{|x|^2(\varepsilon^2+|x{'}|^{2}+|x_{3}+\sqrt{3}{\varepsilon}|^2)}dx\\ \label{A2}
&=\int_{\mathbb{R}^{3}_{+}}\frac{e^{-\frac{|x|^2}{4\sqrt{5}}}(2\sqrt{3}{\varepsilon}x_3+4\varepsilon^2)}{|x|^2(\varepsilon^2+|x{'}|^{2}+|x_{3}+\sqrt{3}{\varepsilon}|^2)}dx\\ \notag
&<2\sqrt{3}{\varepsilon}\int_{\mathbb{R}^{3}_{+}}\frac{e^{-\frac{|x|^2}{4\sqrt{5}}}}{|x|(\varepsilon^2+|x|^{2})}dx
+4\varepsilon^2\int_{\mathbb{R}^{3}_{+}}\frac{e^{-\frac{|x|^2}{4\sqrt{5}}}}{|x|^2(\varepsilon^2+|x|^{2})}dx. \notag
%&<I_1+I_2,
 \end{align*}
Let
\begin{equation*}
I_1=2\sqrt{3}{\varepsilon}\int_{\mathbb{R}^{3}_{+}}\frac{e^{-\frac{|x|^2}{4\sqrt{5}}}}{|x|(\varepsilon^2+|x|^{2})}dx,
\quad
I_2=4\varepsilon^2\int_{\mathbb{R}^{3}_{+}}\frac{e^{-\frac{|x|^2}{4\sqrt{5}}}}{|x|^2(\varepsilon^2+|x|^{2})}dx.
\end{equation*}
%Now, we proceed with the computation of the two integrals above.
We can use the polar coordinates to obtain
  \begin{align}
I_1\notag
&=\sqrt{3}\omega_3{\varepsilon}\int_{0}^{\infty}\frac{e^{-\frac{r^2}{4\sqrt{5}}}r}{\varepsilon^2+r^{2}}dr\\ \notag
&=\sqrt{3}\omega_3{\varepsilon}\Big(\int_{0}^{\varepsilon}\frac{e^{-\frac{r^2}{4\sqrt{5}}}r}{\varepsilon^2+r^{2}}dr
+\int_{\varepsilon}^{1}\frac{e^{-\frac{r^2}{4\sqrt{5}}}r}{\varepsilon^2+r^{2}}dr
+\int_{1}^{\infty}\frac{e^{-\frac{r^2}{4\sqrt{5}}}r}{\varepsilon^2+r^{2}}dr\Big)\\ \label{A3}
&<\sqrt{3}\omega_3{\varepsilon}\Big(\int_{0}^{\varepsilon}\frac{1}{\varepsilon}dr
+\int_{\varepsilon}^{1}\frac{1}{r}dr
+\int_{1}^{\infty}e^{-\frac{r^2}{4\sqrt{5}}}rdr\Big)\\ \notag
&=\sqrt{3}\omega_3{\varepsilon}+\sqrt{3}\omega_3{\varepsilon}|\ln\varepsilon|+\bar{d}_1\varepsilon,\notag
  \end{align}
and
\begin{equation}\label{A4}
  \begin{aligned}
I_2&=2\omega_3{\varepsilon}^2\int_{0}^{\infty}\frac{e^{-\frac{r^2}{4\sqrt{5}}}}{\varepsilon^2+r^{2}}dr\\
&=2\omega_3{\varepsilon}^2\Big(\int_{0}^{\varepsilon}\frac{e^{-\frac{r^2}{4\sqrt{5}}}}{\varepsilon^2+r^{2}}dr
+\int_{\varepsilon}^{\infty}\frac{e^{-\frac{r^2}{4\sqrt{5}}}}{\varepsilon^2+r^{2}}dr
\Big)\\
&<2\omega_3{\varepsilon}^2\Big(\int_{0}^{\varepsilon}\frac{1}{\varepsilon^2}dr
+\int_{\varepsilon}^{\infty}\frac{1}{r^2}dr\Big)=4\omega_3{\varepsilon},
  \end{aligned}
\end{equation}
where $\omega_3$ is the area of unit sphere in $\mathbb{R}^{3}$ and $\bar{d}_1>0$. It follows from \eqref{A2}-\eqref{A4} that
%\begin{equation*}
%\int_{\mathbb{R}^{3}_{+}}\frac{e^{-\frac{|x|^2}{4\sqrt{5}}}}{|x|^2}dx-
%\int_{\mathbb{R}^{3}_{+}}\frac{e^{-\frac{|x|^2}{4\sqrt{5}}}}{(\varepsilon^2+|x{'}|^{2}+|x_{3}+\sqrt{3}{\varepsilon}|^2)}dx<\bar{d}_2{\varepsilon}|\ln\varepsilon|+\bar{d}_3\varepsilon,
%\end{equation*}
%that is,
\begin{equation}\label{A5}
\int_{\mathbb{R}^{3}_{+}}\frac{e^{-\frac{|x|^2}{4\sqrt{5}}}}{(\varepsilon^2+|x{'}|^{2}+|x_{3}+\sqrt{3}{\varepsilon}|^2)}dx>\int_{\mathbb{R}^{3}_{+}}\frac{e^{-\frac{|x|^2}{4\sqrt{5}}}}{|x|^2}dx-\bar{d}_2{\varepsilon}|\ln\varepsilon|-\bar{d}_3\varepsilon,
\end{equation}
where $\bar{d}_2, \bar{d}_3>0$.
We conclude from \eqref{A1} and \eqref{A5} that
\begin{equation*}
{\|v_{{\varepsilon}}\|}_{L^{2}_{K}({\mathbb{R}}^{3}_{+})}^{2}> \sqrt{3}\varepsilon \int_{\mathbb{R}^{3}_{+}}
\frac{e^{-\frac{|x|^2}{4\sqrt{5}}}}{|x|^{2}}dx-\sqrt{3}\bar{d}_2\varepsilon^2|\ln\varepsilon|-\sqrt{3}\bar{d}_3\varepsilon^2.
\end{equation*}
We complete the proof of this Lemma.
\end{proof}

The next Lemma provides the $L^{6}_{K}(\mathbb{R}^3_{+})$-norm and $L^{4}_{K}(\mathbb{R}^2)$-norm of $v_{{\varepsilon}}$ as $\rightarrow0$.
\begin{Lem}\label{lem3.3}
Suppose that $N=3$. As $\varepsilon\rightarrow0$, one has
\begin{align*}
&{\|v_{{\varepsilon}}\|}_{L^{{2}^{*}}_{K}({\mathbb{R}}^{N}_{+})}^{2^{*}}=K_{2}+O({\varepsilon}^{2}),\\
&{\|v_{{\varepsilon}}\|}_{L^{{2}_{*}}_{K}({\mathbb{R}}^{N-1})}^{2_{*}}=K_{3}+O({\varepsilon}^{2}|\ln\varepsilon|),
\end{align*}
where $K_2$ and $K_3$ are defined in \eqref{2.6}.
\end{Lem}
\begin{proof}
By the definitions of $v_\varepsilon$ and $\psi$, we conclude
	\begin{align}
{\|v_{{\varepsilon}}\|}_{L^6_{K}({\mathbb{R}}^{3}_{+})}^{6} \notag
&=\int_{\mathbb{R}^{3}_{+}}K(x)v_{{\varepsilon}}^{6}dx \notag
=\int_{\mathbb{R}^{3}_{+}}K(x)^{-2}{\psi}^{6}U_{\varepsilon}^{6}dx \notag
=\int_{\mathbb{R}^{3}_{+}}e^{-\alpha_1|x|^2}U_{\varepsilon}^{6}dx\\ \notag
&=\int_{\mathbb{R}^{3}_{+}}U_{\varepsilon}^{6}dx+\int_{\mathbb{R}^{3}_{+}}\big(e^{-\alpha_1|x|^2}-1\big)U_{\varepsilon}^{6}dx\\ \label{p11}
&=K_{2}+\sqrt{27}\int_{\mathbb{R}^{3}_{+}}\frac {e^{-\alpha_1\varepsilon^2|y|^2}-1}
{(1+|y{'}|^{2}+|y_{3}+\sqrt{3}|^{2})^{3}}dy\\ \notag
&=K_2+\sqrt{27}\int_{B_{1/\varepsilon}^{+}}\frac{ e^{-\alpha_1{{\varepsilon}^{2}|y|^{2}}}-1}
{(1+|y{'}|^{2}+|y_{3}+\sqrt{3}|^{2})^{3}}dy\\  \notag
&\quad+\sqrt{27}\int_{\mathbb{R}^{3}_{+}\setminus B_{1/\varepsilon}^{+}}\frac{ e^{-\alpha_1{{\varepsilon}^{2}|y|^{2}}}-1}
{(1+|y{'}|^{2}+|y_{3}+\sqrt{3}|^{2})^{3}}dy, \notag
	\end{align}
where $\alpha_1:=\frac{1}{2}+\frac{3}{4\sqrt{5}}$.
From Taylor's formula, we have
\begin{equation*}\label{p0}
e^{-\alpha_1{\varepsilon}^{2}|y|^{2}}-1=-\alpha_1{\varepsilon}^{2}|y|^{2}+O({\varepsilon}^{4}|y|^4),\quad\quad
y\in B_{1/\varepsilon}^{+},
\end{equation*}
which means
\begin{equation}\label{p22}
  \begin{aligned}
\int_{B_{1/\varepsilon}^{+}}\frac{ e^{-\alpha_1{{\varepsilon}^{2}|y|^{2}}-1}}
{(1+|y{'}|^{2}+|y_{3}+\sqrt{3}|^{2})^{3}}dy
&=-\alpha_1{\varepsilon}^{2}\int_{B_{1/\varepsilon}^{+}}\frac{|y|^2}
{(1+|y{'}|^{2}+|y_{3}+\sqrt{3}|^{2})^{3}}dy\\
&\quad+O\Big(\varepsilon^4\int_{B_{1/\varepsilon}^{+}}\frac{ |y|^{4}}
{(1+|y{'}|^{2}+|y_{3}+\sqrt{3}|^{2})^{3}}dy\Big)\\
&=-\bar{d}_4{\varepsilon}^{2}+O\Big(\varepsilon^4\int_{B_{1/\varepsilon}^{+}\setminus B_{1}^+}\frac{ |y|^{4}}
{(1+|y{'}|^{2}+|y_{3}+\sqrt{3}|^{2})^{3}}dy\Big)\\
&\quad+O(\varepsilon^4),
  \end{aligned}
  \end{equation}
where $\bar{d}_4>0$. Note that
\begin{equation*}
\int_{B_{1/\varepsilon}^{+}\setminus B_{1}^+}\frac{ |y|^{4}}
{(1+|y{'}|^{2}+|y_{3}+\sqrt{3}|^{2})^{3}}dy
=O\Big(\int_{B_{1/\varepsilon}^{+}\setminus B_{1}^+}\frac{ 1}
{|y|^{2}}dy\Big)
=O\Big(\int_{1}^{{{1}/\varepsilon}}dr\Big)=O(\varepsilon^{-1}).
\end{equation*}
Combining the above with \eqref{p22}, there holds
\begin{equation}\label{p33}
\int_{B_{1/\varepsilon}^{+}}\frac{ e^{-\alpha_1{{\varepsilon}^{2}|y|^{2}}-1}}
{(1+|y{'}|^{2}+|y_{3}+\sqrt{3}|^{2})^{3}}dy
=-\bar{d}_4{\varepsilon}^{2}+o(\varepsilon^2).
  \end{equation}
Moreover,
\begin{equation}\label{p44}
  \begin{aligned}
\int_{\mathbb{R}^{3}_{+}\setminus B_{1/\varepsilon}^{+}}\frac{ e^{-\alpha_1{{\varepsilon}^{2}|y|^{2}}}-1}
{(1+|y{'}|^{2}+|y_{3}+\sqrt{3}|^{2})^{3}}dy
=O\Big(\int_{\mathbb{R}^{3}_{+}\setminus B_{1/\varepsilon}^{+}}\frac{ 1}
{|y|^6}dy\Big)
=O\Big(\int_{{{1}/\varepsilon}}^{\infty}r^{-4}dr\Big)
=O(\varepsilon^3).
  \end{aligned}
\end{equation}
It follows from \eqref{p11}, \eqref{p33} and \eqref{p44} that
\begin{equation*}
{\|v_{{\varepsilon}}\|}_{{L}_{K}^{6}({\mathbb{R}}^{3}_{+})}^{6}
=K_{2}+O( {\varepsilon}^{2}).
\end{equation*}

On the other hand, we can easily compute that
\begin{equation}\label{w10}
  \begin{aligned}
{\|v_{{\varepsilon}}\|}_{{L}_{K}^{4}({\mathbb{R}}^{2})}^{4}
&=\int_{\mathbb{R}^{2}}K(x{'},0)^{{-1}}{\psi}^{4}U^{4}_{\varepsilon}dx{'}
=\int_{\mathbb{R}^{2}}e^{-\alpha_2|x'|^2}U^{4}_{\varepsilon}dx{'}\\
&=\int_{\mathbb{R}^{2}}U^{4}_{\varepsilon}dx{'}+\int_{\mathbb{R}^{2}}(e^{-\alpha_2|x'|^2}-1)U^{4}_{\varepsilon}dx{'}\\
&=K_{3}+3\int_{\mathbb{R}^{2}}\frac{e^{-\alpha_2{{\varepsilon}^{2}|y'|^{2}}}-1}
{{{{(|y{'}|^{2}+4)^{2}}}}}dy{'}\\
&=K_{3}+3\int_{\hat{B}_{1/\varepsilon}}\frac{e^{-\alpha_2{\varepsilon}^{2}|y'|^{2}}-1}
{{{{(|y{'}|^{2}+4)^{2}}}}}dy{'}
+3\int_{\mathbb{R}^{2}\setminus \hat{B}_{1/\varepsilon}}\frac{e^{-\alpha_2{\varepsilon}^{2}|y'|^{2}}-1}
{{{{(|y{'}|^{2}+4)^{2}}}}}dy{'},
  \end{aligned}
\end{equation}
where $\alpha_2:=\frac{1}{4}+\frac{1}{2\sqrt{5}}$ and $\hat{B}_{r}:=\hat{B}_{r}(0)\subset {\mathbb{R}}^{2}$ for any $r>0$.
We derive from Taylor's formula that
\begin{equation}\label{w11}
  \begin{aligned}
\int_{\hat{B}_{1/\varepsilon}}\frac{ e^{-\alpha_2{\varepsilon}^{2}|y{'}|^{2}}-1}
{(|y{'}|^{2}+4)^{2}}dy{'}
&=-\alpha_2{\varepsilon^{2}}\int_{\hat{B}_{1/\varepsilon}}\frac{|y{'}|^{2}}
{(|y{'}|^{2}+4)^{2}}dy{'}
+O\Big(\varepsilon^4\int_{\hat{B}_{1/\varepsilon}}\frac{ |y{'}|^{4}}
{(|y{'}|^{2}+4)^{2}}dy{'}\Big)\\
&=-\alpha_2\varepsilon^{2}\int_{\hat{B}_{1/\varepsilon}\setminus \hat{B}_{1}}\frac{|y{'}|^{2}}
{(|y{'}|^{2}+4)^{2}}dy{'}
+O\Big(\varepsilon^{4}\int_{\hat{B}_{1/\varepsilon}\setminus \hat{B}_{1}}\frac{ |y{'}|^{4}}
{(|y{'}|^{2}+4)^{2}}dy{'}\Big)\\
&\quad+O(\varepsilon^{2})+O(\varepsilon^4).
  \end{aligned}
  \end{equation}
Since
\begin{equation*}
\int_{\hat{B}_{1/\varepsilon}\setminus \hat{B}_{1}}\frac{ |y{'}|^{2}}
{(|y{'}|^{2}+4)^{2}}dy{'}
=O\Big(\int_{\hat{B}_{2/\varepsilon}\setminus \hat{B}_{1}}\frac{ 1}
{|y{'}|^{2}}dy{'}\Big)
=O\Big(\int_{1}^{{{1}/\varepsilon}}r^{-1}dr\Big)
=O(|\ln\varepsilon|)
\end{equation*}
and
\begin{equation*}
\int_{\hat{B}_{1/\varepsilon}\setminus \hat{B}_{1}}\frac{ |y{'}|^{4}}
{(|y{'}|^{2}+4)^{2}}dy{'}
=O\Big(\int_{\hat{B}_{1/\varepsilon}\setminus \hat{B}_{1}}dy{'}\Big)
=O(\varepsilon^{-2}),
\end{equation*}
we conclude from \eqref{w11} that
\begin{equation}\label{w12}
\int_{\hat{B}_{1/\varepsilon}}\frac{ e^{-\alpha_2{\varepsilon}^{2}|y{'}|^{2}}-1}
{(|y{'}|^{2}+4)^{2}}dy{'}=O(\varepsilon^2|\ln\varepsilon|).
\end{equation}
Moreover,
\begin{equation}\label{w13}
\int_{\mathbb{R}^{2}\setminus \hat{B}_{1/\varepsilon}}\frac{e^{-\alpha_2{\varepsilon}^{2}|y'|^{2}}-1}
{{{{(|y{'}|^{2}+4)^{2}}}}}dy{'}=O\Big(\int_{\mathbb{R}^{2}\setminus \hat{B}_{1/\varepsilon}}\frac{ 1}
{|y{'}|^{4}}dy{'}\Big)
=O\Big(\int_{{{1}/\varepsilon}}^{\infty}r^{-3}dr\Big)
=O(\varepsilon^2).
\end{equation}
In view of \eqref{w10}, \eqref{w12} and \eqref{w13}, we have
\begin{equation*}
{\|v_{{\varepsilon}}\|}_{L^{4}_{K}({\mathbb{R}}^{2})}^{4}=
K_{3}+O({\varepsilon}^2|\ln\varepsilon|).
\end{equation*}
Therefore, we complete the proof.
\end{proof}

For convenience, we denote
\begin{align}
&K_{1}({\varepsilon}):={\|u_{{\varepsilon}}\|}^2,\quad\quad\quad\,\,\,\, \ \
K_{2}({\varepsilon}):={\|u_{{\varepsilon}}\|}_{L^{{2}^{*}}_{K}({\mathbb{R}}^{N}_{+})}^{2^{*}},\label{c7}\\
&K_{3}({\varepsilon}):={\|u_{{\varepsilon}}\|}_{L^{{2}_{*}}_{K}({\mathbb{R}}^{N-1})}^{2_{*}},
 \ \ K_{4}({\varepsilon}):={\|u_{{\varepsilon}}\|}_{L^{2}_{K}({\mathbb{R}}^{N}_{+})}^{2},\label{c8}\\
&\bar{K}_{1}({\varepsilon}):={\|v_{{\varepsilon}}\|}^2,\quad\quad\quad\,\,\,\, \ \
\bar{K}_{2}({\varepsilon}):={\|v_{{\varepsilon}}\|}_{L^{6}_{K}({\mathbb{R}}^{3}_{+})}^{6},\label{c9}\\
&\bar{K}_{3}({\varepsilon}):={\|v_{{\varepsilon}}\|}_{L^{4}_{K}({\mathbb{R}}^{2})}^{4},\,\,\,\,\,\,\, \ \
\bar{K}_{4}({\varepsilon}):={\|v_{{\varepsilon}}\|}_{L^{2}_{K}({\mathbb{R}}^{3}_{+})}^{2}.\label{c10}
\end{align}
 Now, we are ready to verify \eqref{3.0}.
\begin{Lem}\label{lem3.4}  The inequality (\ref {3.0}) holds %and (\ref {2.11}) is naturally obtained,
if one of the following assumptions holds:
\vskip 0.2cm
	
	$( \romannumeral 1)$  $~N=3\ \text{and}~ \lambda>\frac{3+\sqrt{5}}{4}$;
%a=1 \\text{and}~ \mu=0$.
\vskip 0.2cm
	
	$( \romannumeral 2)$  $~N=4 \ \text{and}~\lambda>1$;
%a=1 \\text{and}~ \mu=0$.
	\vskip 0.2cm
	
	$( \romannumeral 3)$   $~N\geq5 \ \text{and}~\lambda>\lambda_N^*$,
%a=1 \ \text{and}~\mu=0$,
~\text{where}
\begin{equation}\label{x9}
\lambda_N^*=\frac{\alpha_{N}}{d_{N}}+\frac{2}{{2^{*}}}\frac{\beta_{N}}{d_{N}}+\frac{2}{{2_{*}}}\frac{\gamma_{N}}{d_{N}},
\end{equation}
and $\alpha_{N}, \beta_N, \gamma_N, d_N$ are given by \eqref{a1}, \eqref{a2}, \eqref{a3}, \eqref{a4}.
\end{Lem}
	
\begin{proof}
For $N=3$, we
%From Lemma \ref{lem3.1}-\ref{lem3.5}, we get $\tilde{I}_{\lambda, \mu}(t\tilde{U}_{{\varepsilon}})=f_{{\varepsilon}}(t)$, where
define the function
\begin{equation*}
\bar{f}_{{\varepsilon}}(t):=I^1_{\lambda, 0}(tv_{{\varepsilon}})=\frac{1}{2}\big(\bar{K}_{1}({\varepsilon})-\lambda \bar{K}_{4}({\varepsilon})\big)t^{2}
-\frac{\bar{K}_{2}({\varepsilon})}{6}t^{6}-\frac{\bar{K}_{3}({\varepsilon})}{4}t^{4},~t>0,
\end{equation*}
and call $\bar{t}_\varepsilon>0$ the point where it attains maximum value. The inequality (\ref {3.0}) holds if we verify that
\begin{equation}\label{zx0}
\sup\limits_{t>0}\bar{f}_{{\varepsilon}}(t)<A.
\end{equation}
Notice that $\bar{t}_\varepsilon$ satisfies
\begin{equation*}
\bar{K}_{1}({\varepsilon})-\lambda \bar{K}_{4}({\varepsilon})-
\bar{K}_{2}({\varepsilon})t^{4}-\bar{K}_{3}({\varepsilon})t^{2}=0.\notag
\end{equation*}
It follows from Lemma \ref{lem3.2}, Lemma \ref{lem3.3} and $K_{1}=K_{2}+K_{3}$ that
\begin{equation*}
\bar{t}_{{\varepsilon}}^{2}=\frac{-\bar{K}_{3}({\varepsilon})+\sqrt{\bar{K}_{3}^{2}({\varepsilon})
+4\bar{K}_{2}({\varepsilon})\big(\bar{K}_{1}({\varepsilon})-\lambda \bar{K}_{4}({\varepsilon})\big)}}{2\bar{K}_{2}({\varepsilon})}
=\frac{2\bar{K}_{2}+O(\varepsilon)}{2\bar{K}_{2}+O(\varepsilon^2)}=1+O({\varepsilon}),
\end{equation*}
which implies $\Delta{\bar{t}_{{\varepsilon}}}:={\bar{t}_{{\varepsilon}}}-1=O({\varepsilon})$.
From Taylor's formula, we conclude that for any $\tau>1$,
\begin{equation*}\label{03}
\bar{t}_{{\varepsilon}}^{\tau}=1+\tau\Delta{\bar{t}_{{\varepsilon}}}+O({{\varepsilon}}^{2}).
\end{equation*}
We derive from the above and \eqref{2.7} that
\begin{equation*}\label{zx1}
  \begin{aligned}
\bar{f}_{{\varepsilon}}(\bar{t}_{{\varepsilon}})
&=\frac{1}{2}\big(\bar{K}_{1}({\varepsilon})-\lambda \bar{K}_{4}({\varepsilon})\big)\bar{t}_{{\varepsilon}}^{2}
-\frac{\bar{K}_{2}({\varepsilon})}{6}\bar{t}_{{\varepsilon}}^{6}-
\frac{\bar{K}_{3}({\varepsilon})}{4}\bar{t}_{{\varepsilon}}^{4}\\
&<\frac{1}{2}\Big(K_1+\sqrt{3}\big(\frac{3+\sqrt{5}}{4}-\lambda\big)\varepsilon \int_{\mathbb{R}^{3}_{+}}
\frac{\psi^2}{|x|^{2}}dx\Big)\bar{t}_{{\varepsilon}}^{2}-\frac{K_{2}}{6}\bar{t}_{{\varepsilon}}^{6}-
\frac{K_{3}}{4}\bar{t}_{{\varepsilon}}^{4}+o(\varepsilon)\\
&=\frac{K_{1}}{2}+\frac{\sqrt{3}}{2}\Big(\frac{3+\sqrt{5}}{4}-\lambda\Big)\varepsilon \int_{\mathbb{R}^{3}_{+}}
\frac{\psi^2}{|x|^{2}}dx
-\frac{K_{2}}{6}
-\frac{K_{3}}{4}\\
& \ \ \ \ +(K_{1}-K_{2}-K_{3})\Delta{\bar{t}_{{\varepsilon}}}+o({{\varepsilon}})\\
&=A+\frac{\sqrt{3}}{2}\Big(\frac{3+\sqrt{5}}{4}-\lambda\Big)\varepsilon \int_{\mathbb{R}^{3}_{+}}
\frac{\psi^2}{|x|^{2}}dx+o({{\varepsilon}})<A,
  \end{aligned}
\end{equation*}
since $\lambda>\frac{3+\sqrt{5}}{4}$ and $\varepsilon$ is small. Thus, (\ref {zx0}) holds.

For $N\geq4$, we define
\begin{equation*}
f_{{\varepsilon}}(t):=I^1_{\lambda, 0}(tu_{{\varepsilon}})=\frac{1}{2}\big(K_{1}({\varepsilon})-\lambda K_{4}({\varepsilon})\big)t^{2}
-\frac{K_{2}({\varepsilon})}{2^{*}}t^{2^{*}}-\frac{K_{3}({\varepsilon})}{2_{*}}t^{2_{*}},~t>0.\notag
\end{equation*}
This means that we only need to show
\begin{equation}\label{zx10}
\sup\limits_{t>0}f_{{\varepsilon}}(t)<A.
\end{equation}
Let $t_{{\varepsilon}}$ be a positive constant such that
\begin{equation*}
f_{{\varepsilon}}(t_{{\varepsilon}})
=\sup\limits_{t>0}f_{{\varepsilon}}(t).
\end{equation*}
Then $t_{{\varepsilon}}$ satisfies
\begin{equation*}
K_{1}({\varepsilon})-\lambda K_{4}({\varepsilon})-
K_{2}({\varepsilon})t^{{2^{*}-2}}-K_{3}({\varepsilon})t^{{2_{*}}-2}=0.\notag
\end{equation*}
Since $2^{*}-2=2(2_{*}-2)$, we have
\begin{equation*}
t_{{\varepsilon}}^{2_{*}-2}=\frac{-K_{3}({\varepsilon})+\sqrt{K_{3}^{2}({\varepsilon})
+4K_{2}({\varepsilon})\big(K_{1}({\varepsilon})-\lambda K_{4}({\varepsilon})\big)}}{2K_{2}({\varepsilon})}.
\end{equation*}
In view of Lemma \ref{lem2.8}, Lemma \ref{lem3.1} and $K_{1}=K_{2}+K_{3}$, we deduce that
\begin{equation*}	
t_{{\varepsilon}}^{2_{*}-2}=
	\begin{cases}
1+O({{\varepsilon}}^{2}|{\ln}\varepsilon|),~&~N=4,\\
1+O({{\varepsilon}}^{2}),~&~N\geq5,
%O({{\varepsilon}}),~&~N=3.
	\end{cases}
\end{equation*}
which means
\begin{equation*}	
\Delta{t_{{\varepsilon}}}:={t_{{\varepsilon}}}-1=
	\begin{cases}
O({{\varepsilon}}^{2}|{\ln}\varepsilon|),~&~N=4,\\
O({{\varepsilon}}^{2}),~&~N\geq5.
%O({{\varepsilon}}),~&~N=3.
	\end{cases}
\end{equation*}
For $N=4$, we follow from Taylor's formula and \eqref{2.7} that for $\varepsilon$ sufficiently small,
\begin{equation}\label{zx2}
  \begin{aligned}
f_{{\varepsilon}}(t_{{\varepsilon}})
&=\frac{1}{2}\big(K_{1}({\varepsilon})-\lambda K_{4}({\varepsilon})\big)t_{{\varepsilon}}^{2}
-\frac{K_{2}({\varepsilon})}{4}t_{{\varepsilon}}^{4}-
\frac{K_{3}({\varepsilon})}{3}t_{{\varepsilon}}^{3}\\
&=\frac{1}{2}\Big(K_{1}+\frac{k_{4}^{2}\omega_4}{2}{{\varepsilon}}^{2}|\ln{\varepsilon}|\Big)t_{{\varepsilon}}^{2}-\frac{\lambda{k_{4}^{2}\omega_4}}{4}\varepsilon^2|\ln\varepsilon|t_{{\varepsilon}}^{2}
-\frac{K_{2}}{4}t_{{\varepsilon}}^{4}-\frac {K_{3}}{3}t_{{\varepsilon}}^{3}+
O({{\varepsilon}}^{2})\\
&=\frac{1}{2}\Big(K_{1}+\frac{k_{4}^{2}\omega_4}{2}{{\varepsilon}}^{2}|\ln{\varepsilon}|\Big)-\frac{\lambda{k_{4}^{2}\omega_4}}{4}\varepsilon^2|\ln\varepsilon|-\frac{K_{2}}{4}
-\frac{K_{3}}{3}\\
& \ \ \ \ +(K_{1}-K_{2}-K_{3})\Delta{t_{{\varepsilon}}}+O({{\varepsilon}}^{2})\\
&=A+\frac{k_{4}^{2}\omega_4}{4}(1-\lambda){{\varepsilon}}^{2}|\ln{\varepsilon}|+O({{\varepsilon}}^{2})
<A,
  \end{aligned}
\end{equation}
because $\lambda>1$. For $N\ge5$, we conclude that
  \begin{align}
f_{{\varepsilon}}(t_{{\varepsilon}})\nonumber
&=\frac{1}{2}\big(K_{1}({\varepsilon})-\lambda K_{4}({\varepsilon})\big)t_{{\varepsilon}}^{2}
-\frac{K_{2}({\varepsilon})}{2^{*}}t_{{\varepsilon}}^{2^{*}}-
\frac{K_{3}({\varepsilon})}{2_{*}}t_{{\varepsilon}}^{2_{*}}\\ \nonumber
&=\frac{1}{2}(K_{1}+\alpha_{N}{{\varepsilon}}^{2})t_{{\varepsilon}}^{2}-\frac{\lambda}{2}d_{N}{{\varepsilon}}^{2}t_{{\varepsilon}}^{2}
-\frac{1}{2^{*}}(K_{2}-\beta_{N}{{\varepsilon}}^{2})t_{{\varepsilon}}^{2^{*}}-\frac{1}{2_{*}}(K_{3}-\gamma_{N}{{\varepsilon}}^{2})t_{{\varepsilon}}^{2_{*}}+
o({{\varepsilon}}^{2})\\ \label{zx3}
&=\frac{1}{2}(K_{1}+\alpha_{N}{{\varepsilon}}^{2})-\frac{\lambda}{2}d_{N}{{\varepsilon}}^{2}-\frac{1}{2^{*}}(K_{2}-\beta_{N}{{\varepsilon}}^{2})
-\frac{1}{2_{*}}(K_{3}-\gamma_{N}{{\varepsilon}}^{2})\\ \nonumber
& \ \ \ \ +(K_{1}-K_{2}-K_{3})\Delta{{t}_{{\varepsilon}}}+o({{\varepsilon}}^{2})\\ \nonumber
&=A+\Big(\frac{\alpha_{N}}{2}-\frac{\lambda}{2}d_{N}+\frac{\beta_{N}}{2^{*}}+\frac{\gamma_{N}}{2_{*}}\Big){{\varepsilon}}^{2}+o({{\varepsilon}}^{2})<A,\nonumber
  \end{align}
for small $\varepsilon$ since
\begin{equation*}
\lambda>\frac{\alpha_{N}}{d_{N}}+\frac{2}{{2^{*}}}\frac{\beta_{N}}{d_{N}}+\frac{2}{{2_{*}}}\frac{\gamma_{N}}{d_{N}}.
\end{equation*}
Hence, \eqref{zx10} holds from \eqref{zx2} and \eqref{zx3}.
\end{proof}

Next, we give an upper and lower bounds of $\lambda_{N}^*$. To this end, we first recall that for any $\hat{a},\hat{b}>0$, the Beta function is defined by
\begin{equation*}\label{2.13}
B(\hat{a},\hat{b}):=\int_{0}^{\infty}\frac{r^{\hat{a}-1}}{(r+1)^{\hat{a}+\hat{b}}}dr,
\end{equation*}
and it also can be written as
\begin{equation*}\label{2.14}
B(\hat{a},\hat{b})=\frac{\Gamma(\hat{a})\Gamma(\hat{b})}{\Gamma(\hat{a}+\hat{b})},
\end{equation*}
where
$\Gamma(\hat{a}):=\int_{0}^{\infty}r^{\hat{a}-1}e^{-r}dr $
is the Gamma function. Moreover, $\Gamma(\hat{a})=(\hat{a}-1)\Gamma(\hat{a}-1)$. For simplicity, we denote
$\Gamma_{0}:=\Gamma\big(\frac{N-1}{2}\big)$.
\begin{Lem}\label{lem3.5}
For $N\ge 5$, we have
 $\lambda_{N}^*\in(\frac{N}{4}, \frac{N-2}{2})$, where $\lambda_{N}^*$ is given by \eqref{x9}.
\end{Lem}
\begin{proof}
Let
\begin{align*}
&C_{1,N}=\int_{\mathbb{R}^{N}_{+}}\frac{|y{'}|^{2}+{y_{N}(y_{N}+x_{N}^{0})}}
{{{{(1+|y{'}|^{2}+|y_{N}+x_{N}^{0}|^{2})^{N-1}}}}}dy,\quad C_{2,N}=\int_{\mathbb{R}^{N}_{+}}\frac{|y|^{2}}
{{{{(1+|y{'}|^{2}+|y_{N}+x_{N}^{0}|^{2})^{N}}}}}dy,\\
&C_{3,N}=\int_{\mathbb{R}^{N-1}}\frac{|y{'}|^{2}}
{{{{(1+|y{'}|^{2}+|x_{N}^{0}|^{2})^{N-1}}}}}dy{'},\quad\quad  C_{4,N}=\int_{\mathbb{R}^{N}_{+}}
\frac{1}{{{(1+|y{'}|^{2}+|y_{N}+x_{N}^{0}|^{2})^{N-2}}}}dy,\\
&C_{5,N}=\int_{\mathbb{R}^{N}_{+}}\frac{1}
	{{{{(1+|y{'}|^{2}+|y_{N}+x_{N}^{0}|^{2})^{N-1}}}}}dy,\,\,\,\, C_{6,N}=\int_{\mathbb{R}^{N}_{+}}\frac{{x_{N}^{0}(y_{N}+x_{N}^{0})}} {{{{(1+|y{'}|^{2}+|y_{N}+x_{N}^{0}|^{2})^{N-1}}}}}dy.
\end{align*}

Firstly, we give an upper bound of $\lambda_{N}^*$.
Note that
\begin{equation*}
C_{1,N}=C_{4,N}-C_{5,N}-C_{6,N}.\label{x1}\\
%&&C_{2,N}<C_{5,N}.\label{x2}
\end{equation*}
Combining \eqref{a1}, \eqref{a2}, \eqref{a4} with the above, there holds
\begin{equation}\label{ww1}
\frac{\alpha_N}{d_N}=\frac{N-2}{2}\Big(1-\frac{C_{5,N}}{C_{4,N}}-\frac{C_{6,N}}{C_{4,N}}\Big)
\end{equation}
and
\begin{equation}\label{ww2}
\frac{2}{{2^{*}}}\frac{\beta_N}{d_N}=\frac{N-2}{2}\frac{C_{2,N}}{C_{4,N}}
<\frac{N-2}{2}\frac{C_{5,N}}{C_{4,N}}.
\end{equation}
On the other hand, in view of \eqref{a3} and \eqref{a4}, we have
\begin{equation}\label{ww3}
\frac{2}{{2_{*}}}\frac{\gamma_{N}}{d_{N}}
=\frac{\sqrt{N(N-2)}}{4(N-1)}\frac{C_{3,N}}{C_{4,N}}.
\end{equation}
Thus,
\begin{equation}\label{a555}
\frac{\alpha_{N}}{d_{N}}+\frac{2}{{2^{*}}}\frac{\beta_{N}}{d_{N}}+\frac{2}{{2_{*}}}\frac{\gamma_{N}}{d_{N}}
<\frac{N-2}{2}\Big(1-\frac{C_{6,N}}{C_{4,N}}\Big)+\frac{\sqrt{N(N-2)}}{4(N-1)}\frac{C_{3,N}}{C_{4,N}}.
\end{equation}
We now proceed with the computation of the integrals $C_{3,N}$ and $C_{6,N}$.
For the first one,
by changing of variables $z{'}={y{'}}/b$ and using the polar coordinates, we obtain
\begin{equation}\label{a666}
  \begin{aligned}
C_{3,N}&=\frac{1}{b^{N-3}}\int_{\mathbb{R}^{N-1}}\frac{|z{'}|^{2}}{(1+|z{'}|^{2})^{N-1}}dz{'}=\frac{{\omega_{N-1}}}{b^{N-3}}\int_{0}^{\infty}\frac{{\rho}^{N}}{(1+{\rho}^{2})^{N-1}}d{\rho}\\
&=\frac{{\omega_{N-1}}}{2{b^{N-3}}}\int_{0}^{\infty}\frac{{r}^{\frac{N-1}{2}}}{(1+r)^{N-1}}dr=\frac{{\omega_{N-1}}}{2{b^{N-3}}}B\Big(\frac{N+1}{2},\frac{N-3}{2}\Big)\\
&=\frac{(N-1)\omega_{N-1}}{2(N-3)b^{N-3}}{\frac{\Gamma_{0}^2}{\Gamma({N-1})}},
  \end{aligned}
\end{equation}
where
\begin{equation}\label{xxx2}
b=\sqrt{1+|x_{N}^{0}|^{2}}=\sqrt{\frac{2N-2}{N-2}}
\end{equation}
and $\omega_{N-1}$
is the area of unit sphere in $\mathbb{R}^{N-1}$.
For the second one, we use the Fubini Theorem and Cylindrical coordinates to obtain
  \begin{align}
C_{6,N}
&=\int_{0}^{\infty}
\int_{\mathbb{R}^{N-1}}\notag
\frac{{x_{N}^{0}(y_{N}+x_{N}^{0})}}{{{{(1+|y{'}|^{2}+|y_{N}+x_{N}^{0}|^{2})^{N-1}}}}}dy{'}dy_{N}\\ \notag
&=x_{N}^{0}\omega_{N-1}\int_{0}^{\infty}\int_{0}^{\infty}\frac{{r^{N-2}(y_{N}+x_{N}^{0})}} {{{{(1+r^{2}+|y_{N}+x_{N}^{0}|^{2})^{N-1}}}}}drdy_{N}\\ \notag
&=x_{N}^{0}\omega_{N-1}\int_{b}^{\infty}\int_{0}^{\infty}\frac{{r^{N-2}\alpha}} {{{{(\alpha^2+r^{2})^{N-1}}}}}drd\alpha\quad
\quad\bigg(\alpha=\sqrt{1+|y_{N}+x_{N}^{0}|^2}\bigg)\\ \notag
&=x_{N}^{0}\omega_{N-1}\Big(\int_{b}^{\infty}\frac{1} {\alpha^{N-2}}d\alpha\Big)\Big(\int_{0}^{\infty}\frac{{\rho^{N-2}}} {(1+\rho^2)^{N-1}}d\rho\Big)\quad\quad(\rho={r}/{\alpha})\\ \label{a777}
&=\frac{x_{N}^{0}\omega_{N-1}}{(N-3)b^{N-3}}\int_{0}^{\infty}\frac{{\rho^{N-2}}} {(1+\rho^2)^{N-1}}d\rho\\ \notag
&=\frac{x_{N}^{0}\omega_{N-1}}{2(N-3)b^{N-3}}\int_{0}^{\infty}\frac{{\sigma^{\frac{N-3}{2}}}} {(1+\sigma)^{N-1}}d\sigma\quad\quad(\sigma=\rho^2)\\ \notag
&=\frac{x_{N}^{0}\omega_{N-1}}{2(N-3)b^{N-3}}
B\Big(\frac{N-1}{2},\frac{N-1}{2}\Big)\\ \notag
&=\frac{x_{N}^{0}\omega_{N-1}}{2(N-3)b^{N-3}}{\frac{\Gamma_{0}^2}{\Gamma({N-1})}},\notag
\end{align}
where $b$ is given by \eqref{xxx2}.
It follows from \eqref{2.02}, \eqref{a555}, \eqref{a666} and \eqref{a777} that
\begin{equation}\label{K00}
\lambda_{N}^*
%=\frac{\alpha_{N}}{d_{N}}+\frac{2}{{2^{*}}}\frac{\beta_{N}}{d_{N}}+\frac{2}{{2_{*}}}\frac{\gamma_{N}}{d_{N}}
<\frac{N-2}{2}-\frac{\sqrt{N(N-2)}\omega_{N-1}}{8(N-3)b^{N-3}C_{4,N}}{\frac{\Gamma_{0}^2}{\Gamma({N-1})}}
<\frac{N-2}{2}.
\end{equation}

Next, we give a lower bound of $\lambda_{N}^*$. In view of  \eqref{ww1} and \eqref{ww2}, we conclude
\begin{equation*}
\frac{\alpha_{N}}{d_{N}}+\frac{2}{{2^{*}}}\frac{\beta_{N}}{d_{N}}
=\frac{N-2}{2}\Big(1-\frac{C_{5,N}}{C_{4,N}}-\frac{C_{6,N}}{C_{4,N}}\Big)
+\frac{N-2}{2}\frac{C_{2,N}}{C_{4,N}}.
\end{equation*}
Since $|x_{N}^{0}|^2>1$, we have $C_{5,N}<C_{6,N}$, and then
\begin{equation}\label{K1}
\frac{\alpha_{N}}{d_{N}}+\frac{2}{{2^{*}}}\frac{\beta_{N}}{d_{N}}
>\frac{N-2}{2}-(N-2)\frac{C_{6,N}}{C_{4,N}}
+\frac{N-2}{2}\frac{C_{2,N}}{C_{4,N}}.
\end{equation}
From \eqref{2.02} and \eqref{a777}, we obtain
\begin{equation}\label{K2}
\frac{C_{6,N}}{C_{4,N}}
=\frac{\sqrt{N}\omega_{N-1}}{2\sqrt{N-2}(N-3)b^{N-3}C_{4,N}}{\frac{\Gamma_{0}^2}{\Gamma({N-1})}}.
\end{equation}
On the other hand, from the definition of $C_{2,N}$, one has
  \begin{align}
C_{2,N}\notag
&>\int_{0}^{\infty}
\int_{\mathbb{R}^{N-1}}\frac{|y'|^2} {{{{(1+|y{'}|^{2}+|y_{N}+x_{N}^{0}|^{2})^{N}}}}}dy{'}dy_{N}\\ \notag
&=\omega_{N-1}\int_{0}^{\infty}\int_{0}^{\infty}\frac{r^{N}} {{{{(1+r^{2}+|y_{N}+x_{N}^{0}|^{2})^{N}}}}}drdy_{N}\\ \notag
&=\omega_{N-1}\int_{b}^{\infty}\int_{0}^{\infty}\frac{{r^{N}}} {{{{(\alpha^2+r^{2})^{N}}}}}\frac {\alpha}{\sqrt{\alpha^2-1}}drd\alpha\quad\quad\bigg(\alpha=\sqrt{1+|y_{N}+x_{N}^{0}|^2}\bigg)\\ \notag
&>\omega_{N-1}\int_{b}^{\infty}\int_{0}^{\infty}\frac{{r^{N}}} {{{{(\alpha^2+r^{2})^{N}}}}}drd\alpha\\  \label{K3}
&=\omega_{N-1}\Big(\int_{b}^{\infty}\frac{1} {\alpha^{N-1}}d\alpha\Big)\Big(\int_{0}^{\infty}\frac{{\rho^{N}}} {(1+\rho^2)^{N}}d\rho\Big)\quad\quad(\rho={r}/{\alpha})\\ \notag
%&=\frac{\omega_{N-1}}{(N-2)b^{N-2}}\int_{0}^{\infty}\frac{{\rho^{N}}} {(1+\rho^2)^{N}}d\rho\\
&=\frac{\omega_{N-1}}{2(N-2)b^{N-2}}\int_{0}^{\infty}\frac{{\sigma^{\frac{N-1}{2}}}} {(1+\sigma)^N}d\sigma\quad\quad(\sigma=\rho^2)\\ \notag
&=\frac{\omega_{N-1}}{2(N-2)b^{N-2}}B\Big(\frac{N+1}{2},\frac{N-1}{2}\Big)\\ \notag
&=\frac{\omega_{N-1}}{4(N-2)b^{N-2}}{\frac{\Gamma_{0}^2}{\Gamma({N-1})}}, \notag
\end{align}
where $b$ is given by \eqref{xxx2}.
It follows from \eqref{ww3}, \eqref{a666} and \eqref{K1}-\eqref{K3} that
\begin{equation}\label{K4}
  \begin{aligned}
\frac{\alpha_{N}}{d_{N}}+\frac{2}{{2^{*}}}\frac{\beta_{N}}{d_{N}}+\frac{2}{{2_{*}}}\frac{\gamma_{N}}{d_{N}}
&>\frac{N-2}{2}-\frac{3\sqrt{N(N-2)}\omega_{N-1}}{8(N-3)b^{N-3}C_{4,N}}{\frac{\Gamma_{0}^2}{\Gamma({N-1})}}
+\frac{\omega_{N-1}}{8b^{N-2}C_{4,N}}{\frac{\Gamma_{0}^2}{\Gamma({N-1})}}\\
&=\frac{N-2}{2}-\Big(\frac{3\sqrt{N(N-2)}}{8(N-3)b^{N-3}}-\frac{1}{{8b^{N-2}}}\Big)
\frac{\omega_{N-1}}{C_{4,N}}
{\frac{\Gamma_{0}^2}{\Gamma({N-1})}}.
\end{aligned}
\end{equation}
Note that
  \begin{align}
C_{4,N}\notag
%&=\int_{0}^{\infty}\int_{\mathbb{R}^{N-1}}\frac{1} {{{{(1+|y{'}|^{2}+|y_{N}+x_{N}^{0}|^{2})^{N-2}}}}}dy{'}dy_{N}\\
&=\omega_{N-1}\int_{0}^{\infty}
\int_{0}^{\infty}\frac{r^{N-2}}{{{{(1+r^{2}+|y_{N}+x_{N}^{0}|^{2})^{N-2}}}}}drdy_{N}\\ \notag
&=\omega_{N-1}\int_{b}^{\infty}\int_{0}^{\infty}\frac{{r^{N-2}}} {{{{(\alpha^2+r^{2})^{N-2}}}}}\frac {\alpha}{\sqrt{\alpha^2-1}}drd\alpha\quad\quad\bigg(\alpha=\sqrt{1+|y_{N}+x_{N}^{0}|^2}\bigg)\\ \notag
&>\omega_{N-1}\int_{b}^{\infty}\int_{0}^{\infty}\frac{{r^{N-2}}} {{{{(\alpha^2+r^{2})^{N-2}}}}}drd\alpha\\ \notag
&=\omega_{N-1}\Big(\int_{b}^{\infty}\frac{1} {\alpha^{N-3}}d\alpha\Big)\Big(\int_{0}^{\infty}\frac{{\rho^{N-2}}} {(1+\rho^2)^{N-2}}d\rho\Big)\quad\quad(\rho={r}/{\alpha})\\ \notag
%&=\frac{\omega_{N-1}}{(N-2)b^{N-2}}\int_{0}^{\infty}\frac{{\rho^{N}}} {(1+\rho^2)^{N}}d\rho\\
&=\frac{\omega_{N-1}}{2(N-4)b^{N-4}}\int_{0}^{\infty}\frac{{\sigma^{\frac{N-3}{2}}}} {(1+\sigma)^{N-2}}d\sigma\quad\quad(\sigma=\rho^2)\\ \label{K5}
&=\frac{\omega_{N-1}}{2(N-4)b^{N-4}}B\Big(\frac{N-1}{2},\frac{N-3}{2}\Big)\\ \notag
&=\frac{(N-2)\omega_{N-1}}{(N-3)(N-4)b^{N-4}}{\frac{\Gamma_{0}^2}{\Gamma({N-1})}},\notag
\end{align}
where $b$ is given by \eqref{xxx2}.
Combining \eqref{K4} with \eqref{K5}, we deduce that
\begin{equation}\label{K44}
  \begin{aligned}
\lambda_{N}^*
%=\frac{\alpha_{N}}{d_{N}}+\frac{2}{{2^{*}}}\frac{\beta_{N}}{d_{N}}+\frac{2}{{2_{*}}}\frac{\gamma_{N}}{d_{N}}
&>\frac{N-2}{2}-\frac{3\sqrt{N}(N-4)}{8\sqrt{2N-2}}+\frac{(N-3)(N-4)}{8(2N-2)}>\frac{N}{4}
\end{aligned}
\end{equation}
for $N\geq7$. For $N=5$, we have $C_{5,5}<\frac{3}{5}C_{6,5}$ because $|x_{5}^{0}|^2=\frac{5}{3}$. Then
\begin{equation}\label{K7}
\frac{\alpha_{5}}{d_{5}}+\frac{3}{5}\frac{\beta_{5}}{d_{5}}
>\frac{3}{2}-\frac{12}{5}\frac{C_{6,5}}{C_{4,5}}
+\frac{3}{2}\frac{C_{2,5}}{C_{4,5}}.
\end{equation}
From \eqref{ww3}, \eqref{a666}, \eqref{K2}, \eqref{K3} and \eqref{K7}, we conclude
\begin{equation}\label{K8}
\frac{\alpha_{5}}{d_{5}}+\frac{3}{5}\frac{\beta_{5}}{d_{5}}+\frac{3}{4}\frac{\gamma_{5}}{d_{5}}
>\frac{3}{2}-\Big(\frac{11\sqrt{15}}{80b^{2}}-\frac{1}{{8b^{3}}}\Big)
\frac{\omega_{4}}{C_{4,5}}
{\frac{\Gamma_{0}^2}{\Gamma({4})}}.
\end{equation}
It follows from \eqref{K5} and \eqref{K8} that
\begin{equation}\label{K19}
\lambda_{5}^*
%=\frac{\alpha_{5}}{d_{5}}+\frac{3}{5}\frac{\beta_{5}}{d_{5}}+\frac{3}{4}\frac{\gamma_{5}}{d_{5}}
>\frac{3}{2}-\frac{11\sqrt{5}}{80\sqrt{2}}+\frac{1}{32}>\frac{5}{4}.
\end{equation}
For $N=6$, since $|x_{6}^{0}|^2=\frac{3}{2}$, we conclude $C_{5,6}<\frac{2}{3}C_{6,6}$. Thus,
\begin{equation}\label{K11}
\frac{\alpha_{6}}{d_{6}}+\frac{2}{3}\frac{\beta_{6}}{d_{6}}
>2-\frac{10}{3}\frac{C_{6,6}}{C_{4,6}}
+2\frac{C_{2,6}}{C_{4,6}}.
\end{equation}
From \eqref{ww3}, \eqref{a666}, \eqref{K2}, \eqref{K3} and \eqref{K11}, we get
\begin{equation}\label{K12}
\frac{\alpha_{6}}{d_{6}}+\frac{2}{3}\frac{\beta_{6}}{d_{6}}
+\frac{4}{5}\frac{\gamma_{6}}{d_{6}}
>2-\Big(\frac{7\sqrt{6}}{36b^{3}}-\frac{1}{{8b^{4}}}\Big)
\frac{\omega_{5}}{C_{4,6}}
{\frac{\Gamma_{0}^2}{\Gamma({5})}}.
\end{equation}
Then we obtain
\begin{equation}\label{K13}
\lambda_{6}^*
%=\frac{\alpha_{6}}{d_{6}}+\frac{2}{3}\frac{\beta_{6}}{d_{6}}
%+\frac{4}{5}\frac{\gamma_{6}}{d_{6}}
>2-\frac{7\sqrt{3}}{12\sqrt{5}}+\frac{3}{40}>\frac{3}{2}
\end{equation}
from \eqref{K5} and \eqref{K12}. Therefore, it follows from \eqref{K00}, \eqref{K44}, \eqref{K19} and \eqref{K13} that for $N\geq 5$, $\frac{N}{4}<\lambda_{N}^*<\frac{N-2}{2}$.
\end{proof}

\subsection{Verification of the condition (4.1) when $\lambda\geq0$ and $\mu>0$ }\label{S3.2}

In this subsection, we are in a position to check condition \eqref{3.0} under the assumptions $(2)$-$(3)$ in Theorem \ref{Th1.1}. It is clear that the energy function associated to equation \eqref{1.5} with $ \lambda\geq0, a=1$, $\mu>0$ and $2\leq q<2_*$ is
\begin{equation}
	I^1_{\lambda, \mu}(u):={\frac{1}{2}}{\|{u}\|}^{2}-{\frac{\lambda }{2}}\|u_+\|_{L^{{2}}_{K}({\mathbb{R}}^{N}_{+})}^{2}
-{\frac{1 }{2^{*}}}\|{u_{+}}\|_{L^{{2}^{*}}_{K}({\mathbb{R}}^{N}_{+})}^{2^{*}}
-{\frac{1 }{2_{*}}}\|{u_{+}}\|_{L^{{2}_{*}}_{K}({\mathbb{R}}^{N-1})}^{2_{*}}-
{\frac{\mu}{q}}\|{u_{+}}\|_{L^{q}_{K}({\mathbb{R}}^{N-1})}^{q}.\notag
\end{equation}

In the following, we give a fine estimate of the ${L^{q}_{K}({\mathbb{R}}^{N-1})}$-norm of $u_{\varepsilon}$ as  $\varepsilon\rightarrow0$. %where $q\in [2,2_{*})$.
\begin{Lem}\label{lem3.6}
Suppose that  $q\in [2,2_{*})$ and $\theta _{N}=N-1-\frac{(N-2)q}{2}$. As $\varepsilon\rightarrow0$, there holds
\begin{equation}\label{c0}
{\|u_{{\varepsilon}}\|}_{L^{q}_{K}({\mathbb{R}}^{N-1})}^{q}
\geq
\begin{cases}
b_1{\varepsilon}^{\theta_N}+o({\varepsilon}^{\theta_N}),~&~~N\geq3,~2< q<2_*,\\[1mm]
b_1{\varepsilon}+o({\varepsilon}),~&~~N\geq4,~q=2,\\[1mm]
b_2{\varepsilon}|\ln{\varepsilon}|+b_3{\varepsilon},~&~~N=3,~q=2,\\[1mm]
\end{cases}
\end{equation}
where $b_{i}, i=1, 2, 3$ are positive constants independent of $\varepsilon$.
\end{Lem}

\begin{proof}
For any $|x'|\leq 2$, we have $K(x',0)^{1-\frac{q}{2}}\geq e^{-\frac{q-2}{2}}>0$ since $q\in [2, 2_{*})$.
By the definitions of $u_{{\varepsilon}}$ and ${\phi}$, we get
\begin{align}
{\|u_{{\varepsilon}}\|}_{L^{q}_{K}({\mathbb{R}}^{N-1})}^q \notag
&=\int_{\mathbb{R}^{N-1}}K(x',0)u_{\varepsilon}^qdx'
=\int_{\mathbb{R}^{N-1}}\frac{K(x',0)^{1-\frac{q}{2}}\phi(x',0)^{q}k_{N}^{q}\varepsilon^{\frac{(N-2)q}{2}}}{(\varepsilon^{2}+|x'|^2+|\varepsilon x_{N}^{0}|^{2})^{\frac{(N-2)q}{2}}}
dx'\\ \notag
&\geq e^{-\frac{q-2}{2}}k_{N}^{q}\varepsilon^{\frac{(N-2)q}{2} } \int_{\hat{B}_{2}} \frac{\phi(x',0)^{q}}{(\varepsilon^{2}+|x'|^{2}+|\varepsilon x_{N}^{0}|^{2})^{\frac{(N-2)q}{2}}} dx' \\ \notag
&\geq e^{-\frac{q-2}{2}} k_{N}^{q}\varepsilon^{\frac{(N-2)q}{2}}\int_{\hat{B}_{1}} \frac{1}{(\varepsilon^{2}+|x'|^{2}+|\varepsilon x_{N}^{0}|^{2})^{\frac{(N-2)q}{2}}} dx'\\ \label{c1}
&=e^{-\frac{q-2}{2}} k_{N}^{q}\varepsilon^{\theta _{N}}\int_{\hat{B}_{1/\varepsilon}} \frac{1}{(1+|y'|^{2}+| x_{N}^{0}|^{2})^{\frac{(N-2)q}{2}}} dy'\\  \notag
&=e^{-\frac{q-2}{2}} k_{N}^{q}\varepsilon^{\theta _{N}}\int_{\hat{B}_{1}} \frac{1}{(1+|y'|^{2}+|x_{N}^{0}|^{2})^{\frac{(N-2)q}{2}}} dy'\\ \notag
&\quad +e^{-\frac{q-2}{2}} k_{N}^{q}\varepsilon^{\theta _{N}}\int_{\hat{B}_{1/\varepsilon} \setminus \hat{B}_{1}} \frac{1}{(1+|y'|^{2}+|x_{N}^{0}|^{2})^{\frac{(N-2)q}{2}}} dy'\\ \notag
&=\bar{b}_1\varepsilon^{\theta _{N}}+e^{-\frac{q-2}{2}} k_{N}^{q}\varepsilon^{\theta _{N}}\int_{\hat{B}_{1/\varepsilon} \setminus \hat{B}_{1}} \frac{1}{(1+|y'|^{2}+|x_{N}^{0}|^{2})^{\frac{(N-2)q}{2}}} dy', \notag
\end{align}
where $\hat{B}_{r}=\hat{B}_{r}(0)\subset {\mathbb{R}}^{N-1}$ for any $r>0$, $\theta _{N}=N-1-\frac{(N-2)q}{2}$ and $\bar{b}_1>0$. Since $|x_{N}^{0}|^{2}\leq3$ for $N\geq3$, we have
\begin{equation*}%\label{cc6}
\frac{1}{(1+|y'|^{2}+|x_{N}^{0}|^{2})^{\frac{(N-2)q}{2}}}\geq
\frac{1}{(5
|y'|^2)^{\frac{(N-2)q}{2}}}, \ \ y'\in \hat{B}_{1/\varepsilon}\setminus \hat{B}_{1}.
\end{equation*}
Combining the above with \eqref{c1}, there holds
\begin{equation*}	
\begin{aligned}
{\|u_{{\varepsilon}}\|}_{L^{q}_{K}({\mathbb{R}}^{N-1})}^q
&\geq \bar{b}_1\varepsilon^{\theta _{N}}+\bar{b}_2\varepsilon^{\theta _{N}}\int_{\hat{B}_{1/\varepsilon} \setminus \hat{B}_{1}} \frac{1}{|y'|^{(N-2)q}} dy'\\
&=\bar{b}_1\varepsilon^{\theta _{N}}+\bar{b}_{2}{\omega_{N-1}}\varepsilon^{\theta_{N}} \int_{1}^{{{1}/\varepsilon}}r^{N-2-(N-2)q}dr,
	\end{aligned}
\end{equation*}
where $\bar{b}_2>0$ and $\omega_{N-1}$
is the area of unit sphere in $\mathbb{R}^{N-1}$.
It follows that
\begin{equation}\label{c2}
	\begin{aligned}
	{\|u_{{\varepsilon}}\|}_{L^{q}_{K}({\mathbb{R}}^{N-1})}^{q}
&\geq \bar{b}_{1}\varepsilon^{\theta_{N}}+
	\bar{b}_{2}{\omega_{N-1}}\varepsilon^{\theta_{N}} \int_{1}^{{{1}/\varepsilon}}r^{-1}dr
	=\bar{ b}_{1}\varepsilon^{\frac{N-1}{2}}+ \bar{b}_{3} \varepsilon^{\frac{N-1}{2}}|\ln \varepsilon|
	\end{aligned}
\end{equation}
if $q=\frac{N-1}{N-2}$;
\begin{equation}\label{c3}
	\begin{aligned}
		{\|u_{{\varepsilon}}\|}_{L^{q}_{K}({\mathbb{R}}^{N-1})}^{q}
		&\geq \bar{b}_{1}\varepsilon^{\theta_{N}}+ \frac{\bar{b}_{2} \omega_{N-1}}{\big(N-1-(N-2)q\big)} \varepsilon^{\theta_{N}}(\varepsilon^{(N-2)q-(N-1)} -1)\\
		&= \bar{b}_{4}\varepsilon^{\frac{(N-2)q}{2}}+o(\varepsilon^{\frac{(N-2)q}{2}})
	\end{aligned}
\end{equation}
if $q<\frac{N-1}{N-2}$
and
\begin{equation}\label{c4}
	\begin{aligned}
		{\|u_{{\varepsilon}}\|}_{L^{q}_{K}({\mathbb{R}}^{N-1})}^{q}
		&\geq \bar{b}_{1}\varepsilon^{\theta_{N}}+ \frac{\bar{b}_{2} \omega_{N-1}}{\big((N-2)q-(N-1)\big)} \varepsilon^{\theta_{N}}(1-\varepsilon^{(N-2)q-(N-1)})\\
		&= \bar{b}_{5}\varepsilon^{\theta_{N}}+o(\varepsilon^{\theta_{N}})
	\end{aligned}
\end{equation}
if $q>\frac{N-1}{N-2}$,
where $\bar{b}_{3}, \bar{b}_{4}, \bar{b}_{5}$ are positive constants independent of $\varepsilon$.
From \eqref{c2}-\eqref{c4}, we conclude that
\begin{equation*}%\label{c3}
{\|u_{{\varepsilon}}\|}_{L^{q}_{K}({\mathbb{R}}^{N-1})}^{q}
\geq
\begin{cases}
\bar{b}_{4}\varepsilon^{\frac{(N-2)q}{2}}+o(\varepsilon^{\frac{(N-2)q}{2}}),~&~~q<\frac{N-1}{N-2},\\[1mm]
\bar{b}_{3} \varepsilon^{\frac{N-1}{2}}|\ln \varepsilon| +\bar{b}_{1}\varepsilon^{\frac{N-1}{2}} ,~&~~q=\frac{N-1}{N-2},\\[1mm]
\bar{b}_{5}\varepsilon^{\theta_{N}}+o(\varepsilon^{\theta_{N}}), ~&~~q>\frac{N-1}{N-2},
\end{cases}
\end{equation*}
which implies that \eqref{c0} holds.
\end{proof}

In the following, we are going to prove \eqref{3.0}.

\begin{Lem}\label{lem3.7}
The inequality \eqref{3.0} holds if $N\geq3, \lambda\geq0$, $\mu>0$ and $2\leq q<2_*$.
\end{Lem}
\begin{proof}
%Define the function
Let
\begin{equation*}
{g}_{{\varepsilon}}(t):=I^1_{\lambda, \mu}(tu_\varepsilon)=\frac{1}{2}\big(K_{1}({\varepsilon})-\lambda K_{4}({\varepsilon})\big)t^{2}
-\frac{K_{2}({\varepsilon})}{2^{*}}t^{2^{*}}-\frac{K_{3}({\varepsilon})}{2_{*}}t^{2_{*}}
-\frac{\mu{K_{5}({\varepsilon})}}{q}t^{q},
\end{equation*}
where $t>0, K_{i}({\varepsilon}), i=1, 2, 3, 4$ are given by \eqref{c7}, \eqref{c8} and  $K_{5}({\varepsilon}):={\|u_{{\varepsilon}}\|}_{L^{q}_{K}({\mathbb{R}}^{N-1})}^{q}.$
Then the inequality \eqref{3.0} holds if we can prove
\begin{equation}\label{zx4}
\sup\limits_{t>0}g_{{\varepsilon}}(t)<A.
\end{equation}
Let $\tilde{t_{\varepsilon}}>0$ be a constant such that ${g}_{{\varepsilon}}(t)$ attains its maximum value. There holds
\begin{equation*}
 K_{1}({\varepsilon})-\lambda
  K_{4}({\varepsilon})-
K_{2}({\varepsilon})\tilde{t_{\varepsilon}}^{{2^{*}-2}}
-K_{3}({\varepsilon})\tilde{t_{\varepsilon}}^{{2_{*}}-2}-\mu K_{5}({\varepsilon})\tilde{t_{\varepsilon}}^{q-2}=0.\notag
\end{equation*}
For any fixed $\lambda\geq0$, in view of Lemma \ref{lem2.8}, Lemma \ref{lem3.1}, Lemma \ref{lem3.6} and $K_{1}=K_{2}+K_{3}$, one has $\tilde{t_{\varepsilon}}\rightarrow 1$ as ${\varepsilon}\rightarrow 0$, which implies that there exists $a_{1}>0$, independent of ${\varepsilon}$, such that $\tilde{t_{\varepsilon}}\geq a_{1}$ for any ${\varepsilon}>0$ small enough. Thus,
\begin{equation}\label{zx5}
  \begin{aligned}
g_{{\varepsilon}}(\tilde{t_{\varepsilon}})
&\leq\sup\limits_{t>0}\Big( \frac{1}{2}\big(K_{1}({\varepsilon})-\lambda K_{4}({\varepsilon})\big)t^{2}-
\frac{K_{2}({\varepsilon})}{2^{*}}t^{2^{*}}-\frac{K_{3}({\varepsilon})}{2_{*}}t^{2_{*}}\Big)-\frac{\mu K_5(\varepsilon)}{q}\tilde{t_{\varepsilon}}^{p}\\
&\leq\sup\limits_{t>0}\Big(\frac{1}{2}\big(K_{1}({\varepsilon})-\lambda K_{4}({\varepsilon})\big)t^{2}-\frac{K_{2}({\varepsilon})}{2^{*}}t^{2^{*}}-\frac{K_{3}({\varepsilon})}{2_{*}}t^{2_{*}}\Big)-\frac{\mu      a_{1}^{q}}{q} K_{5}({\varepsilon}).
  \end{aligned}
\end{equation}
By modifying the proof of Lemma \ref{lem3.4}, we deduce
\begin{equation}\label{zx6}
\sup\limits_{t>0}\Big(\frac{1}{2}\big(K_{1}({\varepsilon})-\lambda K_{4}({\varepsilon})\big)t^{2}-\frac{K_{2}({\varepsilon})}{2^{*}}t^{2^{*}}-\frac{K_{3}({\varepsilon})}{2_{*}}t^{2_{*}}\Big)=
A+B_\varepsilon,
\end{equation}
where
\begin{equation}\label{00}	
B_\varepsilon=
	\begin{cases}
		O({{\varepsilon}}^{2}),~&~N\geq5 ,\\
O({{\varepsilon}}^{2}|{\ln}\varepsilon|),~&~N=4,\\
O({{\varepsilon}}),~&~N=3.
	\end{cases}
\end{equation}
Now, we are ready to verify \eqref{zx4}.
For $N\geq3$ and $2<q<2_*$, it follows from Lemma \ref{lem3.6}, \eqref{zx5} and \eqref{zx6} that for $\varepsilon$ sufficiently small,
\begin{equation}\label{zss1}	
g_{{\varepsilon}}(\tilde{t_{\varepsilon}})\leq
A+B_\varepsilon-\frac{\mu a_{1}^q b_{1}}{q}{{\varepsilon}}^{\theta _{N}}+o({{\varepsilon}}^{\theta _{N}})<A,
\end{equation}
because $\theta_{N}\in (0,1)$ and $\mu>0$. For $q=2$, from Lemma \ref{lem3.6}, \eqref{zx5} and \eqref{zx6}, we proceed as follows:
\vskip 1mm
(1) If $N\geq4$, one has that for any $\mu>0$, $\varepsilon$ sufficiently small,
\begin{equation}\label{zss2}
  \begin{aligned}
g_{{\varepsilon}}(\tilde{t_{\varepsilon}})\leq
A-\frac{\mu a_{1}^q b_{1}}{q}{{\varepsilon}}+o({{\varepsilon}})<A.
  \end{aligned}
\end{equation}

(2) If $N=3$, there holds
\begin{equation}\label{zss3}
g_{{\varepsilon}}(\tilde{t_{\varepsilon}})\leq  %A+O({\varepsilon})-\frac{\mu{\lambda^{\frac{(N-2)(2-p)}{4}}      } c_{1}^{p}Q_{3}}{p}{{\varepsilon}}^{3-\frac{p}{2}}=
A-\frac{\mu a_{1}^q b_{2}}{q}{\varepsilon}|\ln{\varepsilon}|+O(\varepsilon)<A,
\end{equation}
for small $\varepsilon$ since $\mu>0$.
%\begin{enumerate}
 %    \item If $N\geq4$, one has that for any $\mu>0$, $\varepsilon$ sufficiently small,
%\begin{equation}\label{zss2}
 % \begin{aligned}
%g_{{\varepsilon}}(\tilde{t_{\varepsilon}})\leq
%A-\frac{\mu a_{1}^q b_{1}}{q}{{\varepsilon}}+o({{\varepsilon}})<A.
%  \end{aligned}
%\end{equation}
 %    \item If $N=3$, there holds
%\begin{equation}\label{zss3}
%g_{{\varepsilon}}(\tilde{t_{\varepsilon}})\leq
%A-\frac{\mu a_{1}^q b_{2}}{q}{\varepsilon}|\ln{\varepsilon}|+O(\varepsilon)<A,
%\end{equation}
%for small $\varepsilon$ since $\mu>0$.
%\end{enumerate}
Therefore, we conclude that \eqref{zx4} holds for $\varepsilon>0$ sufficiently small from \eqref{zss1}-\eqref{zss3}.
\end{proof}

Now we are ready to prove our Theorem \ref{Th1.1}.

{\bf Proof of Theorem \ref{Th1.1}.}
It follows from Lemma \ref{lem2.5}, Lemma \ref{lem2.6}, Lemma \ref{lem3.4}, Lemma \ref{lem3.5}, Lemma \ref{lem3.7} and Mountain Pass Lemma that there exists a nonnegative weak solution $u$ of \eqref{1.5} under the assumptions of Theorem \ref{Th1.1}. Moreover, using the Brezis-Kato Theorem, we deduce that $u\in C^2(\overline{{{\mathbb{R}}^{N}_{+}}})$, and then $u$ is a positive solution of \eqref{1.5} from maximum principle. That is, $u$ is a positive solution of \eqref{1.1}. The proof of Theorem \ref{Th1.1} is completed.
\qed

\vs{2mm}
\section{The proof of Theorem \ref{Th1.2}}\label{S4}
\vs{2mm}
In this section, we prove the existence results stated in Theorem \ref{Th1.2} %The existence result will be prove
by checking  the condition \eqref{2.12}. In what follows, we always assume that  $N\geq3$, $\lambda\geq0$, $a=0$,
$\mu\geq0, \lambda+\mu>0$ and $2\leq q<2_*$.

Let
\begin{equation*}\label{030}
\hat{c}_{\lambda, \mu}:=\inf\limits_{u\in X}\big\{\sup\limits_{s>0}I^0_{\lambda, \mu}(su):u\geq0 ~\text{and}~ u\not \equiv 0\big\}.
\end{equation*}
It is clear  that $c^0_{\lambda, \mu}\leq \hat{c}_{\lambda, \mu}$ and the condition \eqref{2.12} in Lemma \ref{lem2.7} holds if
\begin{equation}\label{4.0}
\hat{c}_{\lambda, \mu}<\frac{1}{2(N-1)}S_0^{N-1},
\end{equation}
where $S_0$ is given by \eqref{2.8}.

\subsection{Verification of the condition (5.1) when $\lambda>0$ and $\mu=0$ }\label{S3.1}

In this subsection, we are going to verify condition \eqref{4.0} under the assumption $(1)$ in Theorem \ref{Th1.2}. We start this subsection by defining the energy function associated to equation \eqref{1.5} with $ a=0$ and $\mu=0$, namely
\begin{equation}
	I^0_{\lambda, 0}(u):={\frac{1}{2}}{\|{u}\|}^{2}-{\frac{\lambda }{2}}\|u_+\|_{L^{{2}}_{K}({\mathbb{R}}^{N}_{+})}^{2}
-{\frac{1 }{2_{*}}}\|{u_{+}}\|_{L^{{2}_{*}}_{K}({\mathbb{R}}^{N-1})}^{2_{*}}.\notag
\end{equation}
In the sequel, we perform the finer estimates than Lemma \ref{lem2.9}  on the asymptotic behavior
 of $\hat{u}_{{\varepsilon}}$ as $\varepsilon\rightarrow0$, where $\hat{u}_{{\varepsilon}}$ is given by \eqref{b2}. Particularly, the estimates  on the asymptotic behavior
 of $\hat{u}_{{\varepsilon}}$ as $\varepsilon\rightarrow0$ for $3\le N\le 6$ are obtained.

 Firstly, we present the
 asymptotic estimate of the $X$-norm of $\hat{u}_{{\varepsilon}}$ as $\varepsilon \rightarrow0$.
%The next Lemmas are more delicate estimates than Lemma \ref{Th2.8},
\begin{Lem}\label{lem4.1}
%Suppose that $N\geq3$.
There holds
\begin{equation*}\label{3.09}
{\|\hat{u}_{{\varepsilon}}\|}^{2}=
\begin{cases}
A_N+\hat{\alpha}_{N}{{\varepsilon}}^{2}+o({{\varepsilon}}^{2}),~&~~N\geq 5,\\[1mm]
A_4+\frac{\omega_4}{2}{{\varepsilon}}^{2}|\ln{\varepsilon}|+O({\varepsilon}^{2}),~&~~N=4,\\[1mm]
A_3+O({\varepsilon}),~&~~N=3,\\[1mm]
\end{cases}
\end{equation*}
where $\varepsilon>0$ is sufficiently small, $A_N, \hat{\alpha}_{N}$ are given by \eqref{aa1}, \eqref{a5} and $\omega_{4}$ is the area of unit sphere in ${\mathbb{R}^{4}}$.
\end{Lem}

\begin{proof}
A straightforward computation yields that
\begin{equation*}
  \begin{aligned}
{\|\hat{u}_{{\varepsilon}}    \|}^{2}
&=\int_{\mathbb{R}^{N}_{+}}\Big(|\nabla{\phi}|^{2}\hat{U}_\varepsilon^{2}+\notag
2{\phi}{{{\hat{U}}_{\varepsilon}}}({\nabla{\phi}}\cdot{\nabla{{{{\hat{U}}_{\varepsilon}}}}})\notag
-\frac{1}{2}{\phi}\hat{U}_\varepsilon^{2}(x\cdot{\nabla{\phi}})\Big)dx\\ &\quad+\int_{\mathbb{R}^{N}_{+}}{\phi}^{2}|\nabla{{{\hat{U}}_{\varepsilon}}}|^{2}dx-\notag
\frac{1}{2}\int_{\mathbb{R}^{N}_{+}}{\phi}^{2}{{\hat{U}}_{\varepsilon}}(x\cdot{\nabla{{{\hat{U}}_{\varepsilon}}}})dx+\frac{1}{16}\int_{\mathbb{R}^{N}_{+}}{\phi}^{2}|x|^{2}\hat{U}_\varepsilon^{2}dx.\notag
  \end{aligned}
\end{equation*}
%For any $r>0$, let $B_{r}^{+}:=B_{r}(0)\cap {\mathbb{R}}^{N}_{+}$. Then we have
Note that
\begin{equation*}
\int_{\mathbb{R}^{N}_{+}}{|\nabla{\phi}|^{2}\hat{U}_\varepsilon^{2}}dx={{\varepsilon}^{{N-2}}}\int_{ B_{2}^{+}\setminus{B_{1}^{+}}}{\frac{|\nabla{\phi}|^{2}}{(|x{'}|^{2}+|x_{N}+{\varepsilon}|^{2})^{N-2}}}dx=O({\varepsilon}^{N-2}) \notag
\end{equation*}
as  $\varepsilon \rightarrow 0$. Similarly, one has
\begin{equation*}
\int_{\mathbb{R}^{N}_{+}}\Big(|\nabla{\phi}|^{2}\hat{U}_\varepsilon^{2}+\notag
2{\phi}{{{\hat{U}}_{\varepsilon}}}({\nabla{\phi}}\cdot{\nabla{{{{\hat{U}}_{\varepsilon}}}}})\notag
-\frac{1}{2}{\phi}\hat{U}_\varepsilon^{2}(x\cdot{\nabla{\phi}})\Big)dx=O({\varepsilon}^{N-2}),
\end{equation*}
and thus
\begin{equation}\label{3.05}
{\|\hat{u}_{\varepsilon}\|}^{2}
= \int_{\mathbb{R}^{N}_{+}}{\phi}^{2}|\nabla{\hat{U}_{\varepsilon}}|^{2}dx-
\frac{1}{2}\int_{\mathbb{R}^{N}_{+}}\phi^{2}{{\hat{U}}_{\varepsilon}}(x\cdot{\nabla{{{\hat{U}}_{\varepsilon}}}})dx+\frac{1}{16}\int_{\mathbb{R}^{N}_{+}}{\phi}^{2}|x|^{2}\hat{U}_\varepsilon^{2}dx+O(\varepsilon^{N-2}).
\end{equation}

Now, we estimate each of the integrals on the right-hand side of (\ref{3.05}). To this end, we first calculate
\begin{equation*}
{\nabla{{{\hat{U}}_{\varepsilon}}}}=-\frac{{(N-2){\varepsilon}^{\frac {N-2}{2}}}}{{(|x{'}|^{2}+|x_{N}+{\varepsilon}|^{2})^{\frac{N}{2}}}}(x', x_{N}+{\varepsilon}).\notag
\end{equation*}
From the definition of $A_N$, we have
\begin{equation}\label{3.06}
	\begin{aligned}
		\int_{\mathbb{R}^{N}_{+}}{\phi}^{2}|\nabla{{{\hat{U}}_{\varepsilon}}}|^{2}dx
		&=
		A_N+{(N-2)^{2}}{\varepsilon}^{N-2}\int_{ {\mathbb{R}^{N}_{+}}\setminus{B_{1}^{+}}} \frac{({\phi}^{2}-1)}
		{{{(|x{'}|^{2}+|x_{N}+{\varepsilon}|^{2})^{N-1}}}}dx\\
		&=A_N+O({\varepsilon}^{N-2}).
	\end{aligned}
\end{equation}
%where we used the fact that the integral
%\begin{equation*}
%	\int_{ {\mathbb{R}^{N}_{+}}\setminus{B_{1}^{+}}}\frac{|x{'}|^{2}+|x_{N}+{\varepsilon}x_{N }^{0}|^{2}}{{{ (\varepsilon^{2}+|x'|^{2}+|x_{N}+{\varepsilon}x_{N}^{0}|^{2})^{N}}}}dx
%\end{equation*}
%is finite as $N\geq 3$.

Next, we compute the second integral on the right-hand side of (\ref {3.05}).
Arguing as \eqref{3.06}, %Using the same arguments, %by changing of variables $y={x}/{{\varepsilon}}$,
one has that for $N\geq 5$,
\begin{equation}\label{3.07}
 \begin{aligned}
& \ \ \ \ \int_{\mathbb{R}^{N}_{+}}{\phi}^{2}{{\hat{U}}_{\varepsilon}}(x\cdot{\nabla{{{\hat{U}}_{\varepsilon}}}})dx\\
&=\int_{\mathbb{R}^{N}_{+}}{{\hat{U}}_{\varepsilon}}(x\cdot{\nabla{{{\hat{U}}_{\varepsilon}}}})dx+O({\varepsilon}^{N-2})\\
&=-(N-2){\varepsilon}^{N-2}\int_{\mathbb{R}^{N}_{+}}\frac{|x{'}|^{2}+x_{N}(x_{N}+{\varepsilon})}{{{{(|x{'}|^{2} +|x_{N}+{\varepsilon}|^{2})^{N-1}}}}}dx+O({\varepsilon}^{N-2})\\
&=-(N-2){\varepsilon}^{2}\int_{\mathbb{R}^{N}_{+}}\frac{|y{'}|^{2}+y_{N}(y_{N}+1)}{{{{(|y{'}|^{2}+|y_{N}+1|^{2})^{N-1}}}}}dy+O({\varepsilon}^{N-2})\\
&=-(N-2)(D_{1, N}+D_{2, N}){\varepsilon}^{2}+O({\varepsilon}^{N-2}),
 \end{aligned}
\end{equation}
 where
\begin{equation*}\label{adc} D_{1,N}=\int_{\mathbb{R}^{N}_{+}}\frac{|y{'}|^{2}}{{{{(|y{'}|^{2}+|y_{N}+1|^{2})^{N-1}}}}}dy,\quad
D_{2,N}=\int_{\mathbb{R}^{N}_{+}}\frac{y_{N}(y_{N}+1)}{{{{(|y{'}|^{2}+|y_{N}+1|^{2})^{N-1}}}}}dy.\quad
\end{equation*}
For $ N=3,4$, there holds
\begin{equation}\label{z1}
\begin{aligned}
\int_{\mathbb{R}^{N}_{+}}{\phi}^{2}{{\hat{U}}_{\varepsilon}}(x\cdot{\nabla{{{\hat{U}}_{\varepsilon}}}})dx
&=\int_{B_{2}^{+}}{{\hat{U}}_{\varepsilon}}(x\cdot{\nabla{{{\hat{U}}_{\varepsilon}}}})dx
+\int_{B_{2}^{+}\setminus{B_{1}^{+}}}({\phi}^{2}-1){{\hat{U}}_{\varepsilon}}(x\cdot{\nabla{{{\hat{U}}_{\varepsilon}}}})dx\\
&=-(N-2){\varepsilon}^{2}\int_{B_{2/\varepsilon}^{+}}\frac{|y{'}|^{2}+y_{N}(y_{N}+1)}
{{(|y{'}|^{2}+|y_{N}+1|^{2})^{N-1}}}dy+O({\varepsilon}^{N-2})\\
&=-(N-2){\varepsilon}^{2}\int_{B_{2/\varepsilon}^{+}\setminus{B_{1}^{+}}}\frac{|y{'}|^{2}+y_{N}(y_{N}+1)}
{{{{(|y{'}|^{2}+|y_{N}+1|^{2})^{N-1}}}}}dy+O(\varepsilon^2)\\
&\quad+O({\varepsilon}^{N-2}).
\end{aligned}
\end{equation}
Since for $N=4$,
\begin{equation*}
\begin{aligned}
\int_{B_{2/\varepsilon}^{+}\setminus B_{1}^{+}}\frac{|y{'}|^{2}+y_{4}(y_{4}+1)}
{{{{(|y{'}|^{2}+|y_{4}+1|^{2})^{3}}}}}dy
&=\int_{B_{2/\varepsilon}^{+}\setminus{B_{1}^{+}}}\frac{|y|^{2}}
{{{{(|y{'}|^{2}+|y_{4}+1|^{2})^{3}}}}}dy\\
&\quad+
\int_{B_{2/\varepsilon}^{+}\setminus{B_{1}^{+}}}\frac{y_{4}}
{{{{(|y'|^{2}+|y_{4}+1)^{3}}}}}dy\\
&=\int_{B_{2/\varepsilon}^{+}\setminus{B_{1}^{+}}}\frac{|y|^{2}}
{{{{(|y{'}|^{2}+|y_{4}+1|^{2})^{3}}}}}dy+O_\varepsilon(1)
\end{aligned}
\end{equation*}
and
\begin{equation*}
\begin{aligned}
0&<\int_{B_{2/\varepsilon}^{+}\setminus{B_{1}^{+}}}\frac{1}{|y|^{4}}dy-\int_{B_{2/\varepsilon}^{+}\setminus{B_{1}^{+}}}\frac{|y|^{2}}
{{{{(|y{'}|^{2}+|y_{4}+1|^{2})^{3}}}}}dy\\
&=\int_{B_{2/\varepsilon}^{+}\setminus{B_{1}^{+}}}\frac{{(|y{'}|^{2}+|y_{4}+1|^{2})^{3}}-|y|^{6}}{|y|^{4}{(|y{'}|^{2}+|y_{4}+1|^{2})^{3}}}dy=O_\varepsilon(1),
\end{aligned}
\end{equation*}
 we obtain that
\begin{equation}\label{zz2}
\begin{aligned}
\ \ \ \ \int_{B_{2/\varepsilon}^{+}\setminus B_{1}^{+}}\frac{|y{'}|^{2}+y_{4}(y_{4}+1)}
{{{{(|y{'}|^{2}+|y_{4}+1|^{2})^{3}}}}}dy
&=\int_{B_{2/\varepsilon}^{+}\setminus{B_{1}^{+}}}\frac{1}{|y|^{4}}dy+O_\varepsilon(1)\\
&=\frac{\omega_4}{2}\int_{1}^{{{2}/\varepsilon}}r^{-1}dr+O_\varepsilon(1)\\
&=\frac{\omega_4}{2}(|\ln \varepsilon|+\ln 2)+O_\varepsilon(1),
\end{aligned}
\end{equation}
 where $\omega_{4}$ is the area of unit sphere in ${\mathbb{R}^{4}}$ and $O_\varepsilon(1)$ is a constant associated with $\varepsilon$.
For $N=3$, we get that
\begin{equation}\label{zz3}
\begin{aligned}
\int_{B_{2/\varepsilon}^{+}\setminus{B_{1}^{+}}}\frac{|y{'}|^{2}+y_{3}(y_{3}+1)}
{{{{(|y{'}|^{2}+|y_{3}+1|^{2})^{2}}}}}dy
&=O\Big(\int_{B_{2/\varepsilon}^{+}\setminus{B_{1}^{+}}}\frac{1}{|y|^{2}}dy\Big)
=O(\varepsilon^{-1}).
\end{aligned}
\end{equation}
It follows from \eqref{3.07}--\eqref{zz3} that
\begin{equation}\label{cz1}
\int_{\mathbb{R}^{N}_{+}}{\phi}^{2}{{\hat{U}}_{\varepsilon}}(x\cdot{\nabla{{{\hat{U}}_{\varepsilon}}}})dx
=
\begin{cases}
-(N-2)(D_{1, N}+D_{2, N}){\varepsilon}^{2}+O({\varepsilon}^{N-2}),~&~N\geq5,\\[1mm]
{-\omega_4}{{\varepsilon}}^{2}|{\ln}\varepsilon|+O({{\varepsilon}}^{2}),~&~N=4,\\[1mm]
O({{\varepsilon}}),~&~N=3.
\end{cases}
\end{equation}

By the same trick computation, we calculate the last integral on the right-hand side of (\ref {3.05}) as follows.
\begin{align*}
\int_{\mathbb{R}^{N}_{+}}{\phi}^{2}|x|^{2}\hat{U}_\varepsilon^{2}dx
&=\int_{B_{2}^{+}}|x|^{2}\hat{U}_\varepsilon^{2}dx+\int_{ B_{2}^{+}\setminus{B_{1}^{+}}}({\phi}^{2}-1)|x|^{2}\hat{U}_\varepsilon^{2}dx\\
&=\int_{B_{2}^{+}}|x|^{2}\hat{U}_\varepsilon^{2}dx+O({\varepsilon}^{N-2})\\
&={\varepsilon}^{4}\int_{{{B^{+}_{{2}/{\varepsilon}}}}}
\frac{|y|^{2}}{{({|y{'}|^{2}+|y_{N}+1|^{2})^{N-2}}}}dy+O({\varepsilon}^{N-2})\notag\\
&={\varepsilon}^{4}\int_{B^{+}_{{2}/{\varepsilon}}\setminus{B_{1}^{+}}}
\frac{|y|^{2}}{{({|y{'}|^{2}+|y_{N}+1|^{2})^{N-2}}}}dy+O({\varepsilon}^{4})+O({\varepsilon}^{N-2}).\notag
\end{align*}
Due to
\begin{equation*}
  \begin{aligned}
\int_{{B^{+}_{{2}/{\varepsilon}}}\setminus{B_{1}^{+}}}
\frac{|y|^{2}}{{({|y{'}|^{2}+|y_{N}+1|^{2})^{N-2}}}}dy
&=O\Big(\int_{{B^{+}_{{2}/{\varepsilon}}}\setminus B_{1}^{+}}\frac{1}{|y|^{2N-6}}dy\Big)\notag\\
&=O\Big(\int_{1}^{{2}/{\varepsilon}}r^{5-N}dr\Big),\notag
  \end{aligned}
\end{equation*}
there holds
\begin{equation}\label{cz3}
\int_{\mathbb{R}^{N}_{+}}{\phi}^{2}|x|^{2}\hat{U}_\varepsilon^{2}dx=
\begin{cases}
O({\varepsilon}^{4}),~&~~N\geq 7,\\[1mm]
O({\varepsilon}^{4}|\ln{\varepsilon}|),~&~~N=6,\\[1mm]
O({\varepsilon}^{N-2}),~&~~3\leq N\leq5.\\[1mm]
\end{cases}
\end{equation}

We derive
\begin{equation*}\label{13.09}
{\|\hat{u}_{{\varepsilon}}\|}^{2}=
\begin{cases}
A_N+\frac{N-2}{2}(D_{1, N}+D_{2, N}){\varepsilon}^{2}+o({{\varepsilon}}^{2}),~&~~N\geq 5,\\[1mm]
A_4+\frac{\omega_4}{2}{{\varepsilon}}^{2}|\ln{\varepsilon}|+O({\varepsilon}^{2}),~&~~N=4,\\[1mm]
A_3+O({\varepsilon}),~&~~N=3,\\[1mm]
\end{cases}
\end{equation*}
from \eqref{3.05}, \eqref{3.06}, \eqref{cz1}, \eqref{cz3}. Our lemma follows from the fact (see Lemma 5.3 in \cite{2021}) that
\begin{equation*}
D_{1, N}+D_{2, N}=\frac{\omega_{N-1}}{2(N-4)}\Bigg(B\Big(\frac{N+1}{2}, \frac{N-3}{2}\Big)+\frac{1}{N-3}B\Big(\frac{N-1}{2}, \frac{N-1}{2}\Big)\Bigg).
\end{equation*}
%We complete the proof of this Lemma.
\end{proof}

Next, we state the ${L^{{2}_{*}}_{K}({\mathbb{R}}^{N-1})}$-norm of $\hat{u}_{{\varepsilon}}$ as $\varepsilon\rightarrow0$
\begin{Lem}\label{lem4.2}
As  $\varepsilon\rightarrow 0$, we have
\begin{equation*}
{\|\hat{u}_{{\varepsilon}}\|}_{L^{{2}_{*}}_{K}({\mathbb{R}}^{N-1})}^{2_{*}}
=
\begin{cases}
B_N^{{2_*}/{2}}-\hat{\gamma}_{N}{{\varepsilon}}^{2}+o({\varepsilon}^{2}),~&~~N\geq4,\\[1mm]
B_3^2+O({\varepsilon}^2|\ln\varepsilon|),~&~~N=3, \\[1mm]
\end{cases}
\end{equation*}
 where $B_N, \hat{\gamma}_{N}$ are given in \eqref{aa1}, \eqref{a7}.
\end{Lem}

\begin{proof}
For $N\geq4$, there holds
\begin{equation}\label{xx1}
  \begin{aligned}
{\|\hat{u}_{{\varepsilon}}\|}_{{L}_{K}^{2_{*}}({\mathbb{R}}^{N-1})}^{2_{*}}
%&=\int_{\mathbb{R}^{N-1}}K(x{'},0)^{\frac{1}{2-N}}{\phi}^{2_{*}}U_{\varepsilon}^{2_{*}}dx{'}\\
&=\int_{\mathbb{R}^{N-1}}K(x{'},0)^{\frac{1}{2-N}}\hat{U}^{2_{*}}_{\varepsilon}dx{'}+O({\varepsilon}^{N-1})\\
%&=K_{3}+k_{N}^{2_{*}}\int_{\mathbb{R}^{N-1}}\frac{K({\varepsilon}y{'},0)^{\frac{1}{2-N}}-1}
%{{{{(1+|y{'}|^{2}+|x_{N}^{0}|^{2})^{N-1}}}}}dy{'}+O({\varepsilon}^{N-1})\\
&=B_N^{{2_*}/{2}}+\int_{\mathbb{R}^{N-1}}\frac{e^{-\frac{{\varepsilon}^{2}|y'|^{2}}{4(N-2)}}-1}
{{{{(|y{'}|^{2}+1)^{N-1}}}}}dy{'}+O({\varepsilon}^{N-1}).
  \end{aligned}
\end{equation}
Notice that
\begin{equation}\label{xx2}
  \begin{aligned}
& \ \ \ \ \int_{\mathbb{R}^{N-1}}\frac{ e^{-\frac{{\varepsilon}^{2}|y{'}|^{2}}{4(N-2)}}-1 }
{{{{(|y{'}|^{2}+1)^{N-1}}}}}dy{'}
-\int_{\mathbb{R}^{N-1}}\frac{-\frac{\varepsilon^{2}|y{'}|^{2}}{4(N-2)}}
{{{{(|y{'}|^{2}+1)^{N-1}}}}}dy{'}\\
&=\int_{\hat{B}_{1/\varepsilon}}\frac{ e^{-\frac{{\varepsilon}^{2}|y{'}|^{2}}{4(N-2)}}-1+\frac{\varepsilon^{2}|y{'}|^{2}}{4(N-2)} }
{{{{(|y{'}|^{2}+1)^{N-1}}}}}dy{'}+\int_{\mathbb{R}^{N-1}\setminus \hat{B}_{1/\varepsilon}}\frac{ e^{-\frac{{\varepsilon}^{2}|y{'}|^{2}}{4(N-2)}}-1+\frac{\varepsilon^{2}|y{'}|^{2}}{4(N-2)} }
{{{{(|y{'}|^{2}+1)^{N-1}}}}}dy{'}.
  \end{aligned}
\end{equation}
%where $\hat{B}_{r}:=\hat{B}_{r}(0)\subset {\mathbb{R}}^{N-1}$ for any $r>0$ is a ball.
We conclude from Taylor's formula that
\begin{equation}\label{t1}
e^{-\frac{{\varepsilon}^{2}|y{'}|^{2}}{4(N-2)}}-1=-\frac{{\varepsilon}^{2}|y{'}|^{2}}{4(N-2)}+O({\varepsilon}^{4}|y{'}|^4),\quad\quad
y{'}\in \hat{B}_{1/\varepsilon},
\end{equation}
%Hence,
and then
\begin{equation*}%\label{xx3}
  \begin{aligned}
\int_{\hat{B}_{1/\varepsilon}}\frac{ e^{-\frac{{\varepsilon}^{2}|y{'}|^{2}}{4(N-2)}}-1+\frac{\varepsilon^{2}|y{'}|^{2}}{4(N-2)} }
{{{{(|y{'}|^{2}+1)^{N-1}}}}}dy{'}
&=O\Big(\varepsilon^4\int_{\hat{B}_{1/\varepsilon}}\frac{ |y{'}|^{4}}
{{{{(|y{'}|^{2}+1)^{N-1}}}}}dy{'}\Big)\\
&=O\Big(\varepsilon^4\int_{\hat{B}_{1/\varepsilon}\setminus \hat{B}_{1}}\frac{ |y{'}|^{4}}
{{{{(|y{'}|^{2}+1)^{N-1}}}}}dy{'}\Big)+O(\varepsilon^4).
  \end{aligned}
  \end{equation*}
Due to
\begin{equation*}
  \begin{aligned}
\int_{\hat{B}_{1/\varepsilon}\setminus \hat{B}_{1}}\frac{ |y{'}|^{4}}
{{{{(|y{'}|^{2}+1)^{N-1}}}}}dy{'}
=O\Big(\int_{\hat{B}_{1/\varepsilon}\setminus \hat{B}_{1}}\frac{ 1}
{|y{'}|^{2N-6}}dy{'}\Big)
=O\Big(\int_{1}^{{{1}/\varepsilon}}r^{-N+4}dr\Big),
  \end{aligned}
\end{equation*}
one has
\begin{equation}\label{xx4}
\int_{\hat{B}_{1/\varepsilon}}\frac{ e^{-\frac{{\varepsilon}^{2}|y{'}|^{2}}{4(N-2)}}-1+\frac{\varepsilon^{2}|y{'}|^{2}}{4(N-2)} }
{{{{(|y{'}|^{2}+1)^{N-1}}}}}dy{'}=o(\varepsilon^2).
\end{equation}
On the other hand,
  \begin{align}
& \ \ \ \ \int_{\mathbb{R}^{N-1}\setminus \hat{B}_{1/\varepsilon}}\frac{  \notag e^{-\frac{{\varepsilon}^{2}|y{'}|^{2}}{4(N-2)}}-1+\frac{\varepsilon^{2}|y{'}|^{2}}{4(N-2)} }
{{{{(|y{'}|^{2}+1)^{N-1}}}}}dy{'}\\ \label{xx5}
&=O\Big(\int_{\mathbb{R}^{N-1}\setminus \hat{B}_{1/\varepsilon}}\frac{ 1}
{|y{'}|^{2N-2}}dy{'}\Big)+
O\Big(\varepsilon^2\int_{\mathbb{R}^{N-1}\setminus \hat{B}_{1/\varepsilon}}\frac{ 1}
{|y{'}|^{2N-4}}dy{'}\Big)\\ \notag
&=O\Big(\int_{{{1}/\varepsilon}}^{\infty}r^{-N}dr\Big)+O\Big(\varepsilon^2\int_{{{1}/\varepsilon}}^{\infty}r^{-N+2}dr\Big)=O(\varepsilon^{N-1}).\notag
  \end{align}
In view of \eqref{xx2}, \eqref{xx4} and \eqref{xx5}, we get that
\begin{equation}\label{xx7} \int_{\mathbb{R}^{N-1}}\frac{e^{-\frac{{\varepsilon}^{2}|y'|^{2}}{4(N-2)}}-1}
{{{{(|y{'}|^{2}+1)^{N-1}}}}}dy{'}
=-\frac{\varepsilon^{2}}{4(N-2)}\int_{\mathbb{R}^{N-1}}\frac{|y{'}|^{2}}
{{{{(|y{'}|^{2}+1)^{N-1}}}}}dy{'}+o({\varepsilon}^{2})
\end{equation}
for $N\geq4$. It follows from \eqref{xx1} and \eqref{xx7} that for $N\geq4$,
\begin{equation}\label{xx8}
{\|\hat{u}_{{\varepsilon}}\|}_{{L}_{K}^{2_{*}}({\mathbb{R}}^{N-1})}^{2_{*}}
=B_{N}^{2_*/2}-\frac{D_{3,N}}{4(N-2)} {\varepsilon}^{2}+o({\varepsilon}^{2})
\end{equation}
with
\begin{equation*}
D_{3,N}=\int_{\mathbb{R}^{N-1}}\frac{|y{'}|^{2}}
{{{{(|y{'}|^{2}+1)^{N-1}}}}}dy{'}.
\end{equation*}
%Substituting the above into \eqref{3.017}, then
%\begin{equation}\label{3.018}
%{\|\tilde{U}_{{\varepsilon}}\|}_{{L}_{K}^{2_{*}}({\mathbb{R}}^{N-1})}^{2_{*}}
%=K_{3}-{\varepsilon}^{2}\frac{k_{N}^{2_{*}}}{4(N-2)}C_{3,N}+o({\varepsilon}^{3})+O({\varepsilon}^{N-1}),
%\end{equation}
%where
%\begin{equation*}\label{a5}
%C_{3,N}=\int_{\mathbb{R}^{N-1}}\frac{|y{'}|^{2}}
%{{{{(1+|y{'}|^{2}+|x_{N}^{0}|^{2})^{N-1}}}}}dy{'}.
%\end{equation*}

For $N=3$, one has
  \begin{align}
{\|\hat{u}_{{\varepsilon}}\|}_{{L}_{K}^{4}({\mathbb{R}}^{2})}^{4}\notag
&=\int_{\mathbb{R}^{2}}K(x{'},0)^{{-1}}{\phi}^{4}\hat{U}^{4}_{\varepsilon}dx{'}\\ \notag
&=\int_{\hat{B}_{2}}K(x{'},0)^{{-1}}\hat{U}^{4}_{\varepsilon}dx{'}+
\int_{\hat{B}_{2}\setminus{\hat{B}_{1}}}K(x{'},0)^{{-1}}\big({\phi}^{4}-1\big)\hat{U}^{4}_{\varepsilon}dx{'}\\ \label{q1}
&=\int_{\hat{B}_{2}}K(x{'},0)^{{-1}}\hat{U}^{4}_{\varepsilon}dx{'}+O({\varepsilon}^2)\\ \notag
&=\int_{\hat{B}_{2}}\hat{U}^{4}_{\varepsilon}dx{'}
+\int_{\hat{B}_{2}}\big(K(x{'},0)^{{-1}}-1\big)\hat{U}^{4}_{\varepsilon}dx{'}+O({\varepsilon}^2)\\ \notag
&=B_{3}^{2}+\int_{\hat{B}_{{2}/{\varepsilon}}}\frac{e^{-\frac{{\varepsilon}^{2}|y'|^{2}}{4}}-1}
{{{{(|y{'}|^{2}+1)^{2}}}}}dy{'}+O({\varepsilon}^{2}).\notag
  \end{align}
Similar as (\ref{t1}), we deduce that
%\begin{equation*}
%e^{-\frac{{\varepsilon}^{2}|y{'}|^{2}}{4}}-1=-\frac{{\varepsilon}^{2}|y{'}|^{2}}{4}+O({\varepsilon}^{4}|y{'}|^4),\quad\quad
%y{'}\in \hat{B}_{2/\varepsilon},\notag
%\end{equation*}
%and then
\begin{equation*}\label{q02}
  \begin{aligned}
\int_{\hat{B}_{2/\varepsilon}}\frac{ e^{-\frac{{\varepsilon}^{2}|y{'}|^{2}}{4}}-1}
{{{{(|y{'}|^{2}+1)^{2}}}}}dy{'}
&=-\frac{\varepsilon^{2}}{4}\int_{\hat{B}_{2/\varepsilon}}\frac{|y{'}|^{2}}
{(|y{'}|^{2}+1)^{2}}dy{'}
+O\Big(\varepsilon^4\int_{\hat{B}_{2/\varepsilon}}\frac{ |y{'}|^{4}}
{(|y{'}|^{2}+1)^{2}}dy{'}\Big)\\
&=-c_1\varepsilon^{2}-\frac{\varepsilon^{2}}{4}\int_{\hat{B}_{2/\varepsilon}\setminus \hat{B}_{1}}\frac{|y{'}|^{2}}
{(|y{'}|^{2}+1)^{2}}dy{'}+O(\varepsilon^4)
\\
&\quad+O\Big(\varepsilon^4\int_{\hat{B}_{2/\varepsilon}\setminus \hat{B}_{1}}\frac{ |y{'}|^{4}}
{(|y{'}|^{2}+1)^{2}}dy{'}\Big),
  \end{aligned}
  \end{equation*}
where $c_1$ is a positive constant. Since
\begin{equation*}
\int_{\hat{B}_{2/\varepsilon}\setminus \hat{B}_{1}}\frac{ |y{'}|^{2}}
{(|y{'}|^{2}+1)^{2}}dy{'}
=O\Big(\int_{\hat{B}_{2/\varepsilon}\setminus \hat{B}_{1}}\frac{ 1}
{|y{'}|^{2}}dy{'}\Big)
=O\Big(\int_{1}^{{{2}/\varepsilon}}r^{-1}dr\Big)
=O(|\ln\varepsilon|)
\end{equation*}
and
\begin{equation*}
\int_{\hat{B}_{2/\varepsilon}\setminus \hat{B}_{1}}\frac{ |y{'}|^{4}}
{(|y{'}|^{2}+1)^{2}}dy{'}
=O\Big(\int_{\hat{B}_{2/\varepsilon}\setminus \hat{B}_{1}}dy{'}\Big)
=O(\varepsilon^{-2}),
\end{equation*}
there holds
\begin{equation*}\label{q5}
\int_{\hat{B}_{2/\varepsilon}}\frac{ e^{-\frac{{\varepsilon}^{2}|y{'}|^{2}}{4}}-1}
{{{{(|y{'}|^{2}+1)^{2}}}}}dy{'}=O(\varepsilon^2|\ln\varepsilon|).
\end{equation*}
This identity and \eqref{q1} imply that for $N=3$,
\begin{equation}\label{q4}
{\|\hat{u}_{{\varepsilon}}\|}_{L^{{2}_{*}}_{K}({\mathbb{R}}^{N-1})}^{2_{*}}
=B_{3}^2+O({\varepsilon}^2|\ln\varepsilon|).
\end{equation}
Hence, combining \eqref{xx8} with \eqref{q4}, there holds
\begin{equation*}
{\|\hat{u}_{{\varepsilon}}\|}_{L^{{2}_{*}}_{K}({\mathbb{R}}^{N-1})}^{2_{*}}
=
\begin{cases}
B_{N}^{2_*/2}- \frac{D_{3,N}}{4(N-2)}{\varepsilon}^{2}+o({\varepsilon}^{2}),~&~~N\geq4,\\[1mm]
B_3^2+O({\varepsilon}^2|\ln\varepsilon|),~&~~N=3.\\[1mm]
\end{cases}
\end{equation*}
Moreover, we obtain from Lemma 5.4 in \cite{2021} that
\begin{equation*}
D_{3,N}=\frac{\omega_{N-1}}{2}B\Big(\frac{N+1}{2}, \frac{N-3}{2}\Big).
\end{equation*}
We complete the proof of this Lemma.
\end{proof}

Now, we are going to compute the ${L^{2}_{K}({\mathbb{R}}^{N}_{+})}$-norm of $\hat{u}_{{\varepsilon}}$ as $\varepsilon\rightarrow0$.
\begin{Lem}\label{lem4.3}
%Suppose that $N\geq3$.
There holds
\begin{equation*}
{\|\hat{u}_{{\varepsilon}}\|}_{L^{2}_{K}({\mathbb{R}}^{N}_{+})}^{2}=
\begin{cases}
\hat{d}_{N}{{\varepsilon}}^{2}+o({{\varepsilon}}^{2}),~&~~N\geq5,\\[1mm]
\frac{\omega_4}{2}\varepsilon^2|\ln\varepsilon|+O({{\varepsilon}}^{2}),~&~~N=4,\\[1mm]
O({\varepsilon}),~&~~N=3 \\[1mm]
\end{cases}
\end{equation*}
as $\varepsilon\rightarrow 0$, $\hat{d}_{N}$ is given by \eqref{a66} and $\omega_{4}$ is the area of unit sphere in ${\mathbb{R}^{4}}$.
\end{Lem}

\begin{proof}
From ${\phi}^{2}\hat{U}_{\varepsilon}^{2}
=\hat{U}_{\varepsilon}^{2}+({\phi}^{2}-1)\hat{U}_{\varepsilon}^{2}$, we have
\begin{equation}\label{wz4}
  \begin{aligned}
{\|\hat{u}_{{\varepsilon}}\|}_{L^{2}_{K}({\mathbb{R}}^{N}_{+})}^{2}
&=\int_{\mathbb{R}^{N}_{+}}{\phi}^{2}\hat{U}_{\varepsilon}^{2}dx
=\int_{\mathbb{R}^{N}_{+}}\hat{U}_{\varepsilon}^{2}dx+O(\varepsilon^{N-2})\\
&={\varepsilon}^{N-2}\int_{\mathbb{R}^{N}_{+}}
\frac{1}{{{(|x{'}|^{2}+|x_{N}+{\varepsilon}|^{2})^{N-2}}}}dx+O({\varepsilon}^{N-2})\\
&=D_{4,N}{\varepsilon}^{2}+O(\varepsilon^{N-2})
 \end{aligned}
\end{equation}
for $N\geq5$, where
\begin{equation*}
D_{4,N}=\int_{\mathbb{R}^{N}_{+}}
\frac{1}{{{(|y{'}|^{2}+|y_{N}+1|^{2})^{N-2}}}}dy.
\end{equation*}

For $N=3, 4$, arguing as above, there holds
\begin{equation}\label{wzz5}
  \begin{aligned}
{\|\hat{u}_{{\varepsilon}}\|}_{L^{2}_{K}({\mathbb{R}}^{N}_{+})}^{2}
&=\int_{\mathbb{R}^{N}_{+}}{\phi}^{2}\hat{U}_{\varepsilon}^{2}dx
=\int_{B_{2}^{+}}\hat{U}_\varepsilon^{2}dx+\int_{ B_{2}^{+}\setminus{B_{1}^{+}}}({\phi}^{2}-1)\hat{U}_\varepsilon^{2}dx\\
&=\int_{B_{2}^{+}}\hat{U}_\varepsilon^{2}dx+O({\varepsilon}^{N-2})\\
&={\varepsilon}^{2}\int_{{{B^{+}_{{2}/{\varepsilon}}}}}
\frac{1}{(|y{'}|^2+|y_{N}+1|^2)^{N-2}}dy+O({\varepsilon}^{N-2})\\
&={\varepsilon}^{2}\int_{B^{+}_{{2}/{\varepsilon}}\setminus{B_{1}^{+}}}
\frac{1}{(|y{'}|^2+|y_{N}+1|^2)^{N-2}}dy+O({\varepsilon}^{2})+O({\varepsilon}^{N-2}).
  \end{aligned}
\end{equation}
Since for $N=4$,
\begin{equation*}
\begin{aligned}
0&<\int_{B_{2/\varepsilon}^{+}\setminus{B_{1}^{+}}}\frac{1}{|y|^{4}}dy-\int_{B^{+}_{{2}/{\varepsilon}}\setminus{B_{1}^{+}}}
\frac{1}{{({|y{'}|^{2}+|y_{4}+1|^{2})^{2}}}}dy\\
&=\int_{B_{2/\varepsilon}^{+}\setminus{B_{1}^{+}}}
\frac{{(|y{'}|^{2}+|y_{4}+1|^{2})^{2}}-|y|^{4}}{|y|^{4}
{(|y{'}|^{2}+|y_{4}+1|^{2})^{2}}}dy=O_\varepsilon(1),
\end{aligned}
\end{equation*}
we deduce that
\begin{equation}\label{wzzz2}
\begin{aligned}
\int_{B^{+}_{{2}/{\varepsilon}}\setminus{B_{1}^{+}}}
\frac{1}{{(|y{'}|^{2}+|y_{4}+1|^{2})^{2}}}dy
&=\int_{B_{2/\varepsilon}^{+}\setminus{B_{1}^{+}}}\frac{1}{|y|^{4}}dy+O_\varepsilon(1)\\
&=\frac{\omega_4}{2}\int_{1}^{{{2}/\varepsilon}}r^{-1}dr+O_\varepsilon(1)\\
&=\frac{\omega_4}{2}(|\ln \varepsilon|+\ln 2)+O_\varepsilon(1),
\end{aligned}
\end{equation}
where $\omega_4$ is the area of unit sphere in $\mathbb{R}^{4}$ and  $O_\varepsilon(1)$ is a constant associated with $\varepsilon$. It is clearly that for $N=3$,
\begin{equation}\label{wzzz3}
\begin{aligned}
\int_{B^{+}_{{2}/{\varepsilon}}\setminus{B_{1}^{+}}}
\frac{1}{{(|y{'}|^{2}+|y_{3}+1|^{2})}}dy
=O\Big(\int_{B_{2/\varepsilon}^{+}\setminus{B_{1}^{+}}}\frac{1}{|y|^{2}}dy\Big)
=O(\varepsilon^{-1}).
\end{aligned}
\end{equation}
In view of \eqref{wzz5}-\eqref{wzzz3}, one has
\begin{equation}\label{wz5}
{\|\hat{u}_{{\varepsilon}}\|}_{L^{2}_{K}({\mathbb{R}}^{N}_{+})}^{2}=
\begin{cases}
\frac{{\omega_4}}{2}\varepsilon^2|\ln\varepsilon|+O({\varepsilon}^2),~&~~N=4,\\[1mm]
O({\varepsilon}),~&~~N=3.\\[1mm]
\end{cases}
\end{equation}

From \eqref{wz4} and \eqref{wz5}, we conclude that
\begin{equation*}
{\|\hat{u}_{{\varepsilon}}\|}_{L^{2}_{K}({\mathbb{R}}^{N}_{+})}^{2}=
\begin{cases}
D_{4,N}{{\varepsilon}}^{2}+o({{\varepsilon}}^{2}),~&~~N\geq5,\\[1mm]
\frac{\omega_4}{2}\varepsilon^2|\ln\varepsilon|+O({{\varepsilon}}^{2}),~&~~N=4,\\[1mm]
O({\varepsilon}),~&~~N=3.\\[1mm]
\end{cases}
\end{equation*}
On the other hand, from Lemma 5.4 in \cite{2021}, we have
\begin{equation*}
D_{4,N}=\frac{\omega_{N-1}}{2(N-4)}B\Big(\frac{N-1}{2}, \frac{N-3}{2}\Big).
\end{equation*}
Hence, the proof is completed.
\end{proof}

It is necessary to state more delicate estimates for $N=3$. To this end, we
introduce the function
\begin{equation*}\label{cw1} \hat{v}_{\varepsilon}(x):=K(x)^{-\frac{1}{2}}\psi(x)\hat{U}_{{\varepsilon}}(x),
\end{equation*}
where $\varepsilon>0$ and $\hat{U}_\varepsilon, \psi$ are defined by \eqref{2.9}, \eqref{xxx1}. The following estimates hold.
\begin{Lem}\label{lem4.4}
Suppose that $N=3$. Then we obtain
\begin{align*}
&{\|\hat{v}_{{\varepsilon}}\|}^{2}<A_N+\frac{(3+\sqrt{5})\varepsilon}{4} \int_{\mathbb{R}^{N}_{+}}
\frac{\psi^2}{|x|^{2}}dx,\\
&{\|\hat{v}_{{\varepsilon}}\|}_{L^{2}_{K}({\mathbb{R}}^{N}_{+})}^{2}> \varepsilon\int_{\mathbb{R}^{N}_{+}}
\frac{\psi^2}{|x|^{2}}dx-c_1\varepsilon^2|\ln\varepsilon|-c_2\varepsilon^2,
\end{align*}
where $\varepsilon>0$ is sufficiently small, $A_N, \psi$ are given by \eqref{aa1}, \eqref{xxx1} and $c_1$, $c_2$ are positive constants.
\end{Lem}
\begin{proof}
Simple computation shows that
  \begin{align*}
{\|\hat{v}_{{\varepsilon}} \|}^{2}
&=\varepsilon\int_{\mathbb{R}^{3}_{+}}\frac{|\nabla{\psi}|^{2}}{(|x{'}|^{2}+|x_{3}+{\varepsilon}|^{2})}dx
-2\varepsilon\int_{\mathbb{R}^{3}_{+}}\frac{\psi{\nabla{\psi}\cdot{(x{'},x_3+\varepsilon)}}}{(|x{'}|^{2}+|x_{3}+{\varepsilon}|^{2})^2}dx\\
&\quad-\frac{\varepsilon}{2}\int_{\mathbb{R}^{3}_{+}}\frac{\psi(x\cdot{\nabla{\psi}})}{(|x{'}|^{2}+|x_{3}+{\varepsilon}|^2)}dx
+\int_{\mathbb{R}^{3}_{+}}{\psi}^{2}|\nabla{{{\hat{U}}_{\varepsilon}}}|^{2}dx\\
&\quad+\frac{\varepsilon}{2}\int_{\mathbb{R}^{3}_{+}}\frac{{\psi}^2\big(|x{'}|^2+x_3(x_3+\varepsilon)\big)}{(|x{'}|^{2}+|x_{3}+{\varepsilon}|^2)^2}dx
+\frac{\varepsilon}{16}\int_{\mathbb{R}^{3}_{+}}\frac{{\psi}^2|x|^2}{(|x{'}|^{2}+|x_{3}+{\varepsilon}|^2)}dx\\
&<A_3+\varepsilon\int_{\mathbb{R}^{3}_{+}}\frac{|\nabla{\psi}|^{2}}{(|x{'}|^{2}+|x_{3}+{\varepsilon}|^{2})}dx
-2\varepsilon\int_{\mathbb{R}^{3}_{+}}\frac{\psi{\nabla{\psi}\cdot{(x{'},x_3+\varepsilon)}}}{(|x{'}|^{2}+|x_{3}+{\varepsilon}|^{2})^2}dx\\
&\quad-\frac{\varepsilon}{2}\int_{\mathbb{R}^{3}_{+}}\frac{\psi(x\cdot{\nabla{\psi}})}{(|x{'}|^{2}+|x_{3}+{\varepsilon}|^2)}dx
+\frac{\varepsilon}{2}\int_{\mathbb{R}^{3}_{+}}\frac{{\psi}^2}{(|x{'}|^{2}+|x_{3}+{\varepsilon}|^2)}dx+\frac{\varepsilon}{16}\int_{\mathbb{R}^{3}_{+}}{\psi}^2dx.
  \end{align*}
From the definition of $\psi$, we have
\begin{equation*}
\nabla\psi(x)=-\frac{1}{4\sqrt{5}}\psi x,
\end{equation*}
and then
\begin{equation*}\label{cw1}
 % \begin{aligned}
{\|\hat v_{{\varepsilon}} \|}^{2}
%&<K_1+\frac{\sqrt{3}\varepsilon}{80}\int_{\mathbb{R}^{3}_{+}}\frac{{\psi}^{2}|x|^2}{(\varepsilon^2+|x{'}|^{2}+|x_{3}+\sqrt{3}{\varepsilon}|^{2})}dx\\
%&\quad+\frac{\sqrt{3}\varepsilon}{2\sqrt{5}}\int_{\mathbb{R}^{3}_{+}}\frac{\psi^2(|x{'}|^2+x_3(x_3+\sqrt{3}\varepsilon))}{(\varepsilon^2+|x{'}|^{2}+|x_{3}+\sqrt{3}{\varepsilon}|^{2})^2}dx\\
%&\quad+\frac{\sqrt{3}\varepsilon}{8\sqrt{5}}\int_{\mathbb{R}^{3}_{+}}\frac{\psi^2|x|^2}{(\varepsilon^2+|x{'}|^{2}+|x_{3}+\sqrt{3}{\varepsilon}|^2)}dx\\
%&\quad+\frac{\sqrt{3}\varepsilon}{2}\int_{\mathbb{R}^{3}_{+}}\frac{{\psi}^2}{(\varepsilon^2+|x{'}|^{2}+|x_{3}+\sqrt{3}{\varepsilon}|^2)}dx\\
%&\quad+\frac{\sqrt{3}\varepsilon}{16}\int_{\mathbb{R}^{3}_{+}}{\psi}^2dx\\
%&<K_1+\frac{\sqrt{3}\varepsilon}{80}\int_{\mathbb{R}^{3}_{+}}{\psi}^{2}dx
%+\frac{\sqrt{3}\varepsilon}{2\sqrt{5}}\int_{\mathbb{R}^{3}_{+}}\frac{\psi^2}{(\varepsilon^2+|x{'}|^{2}+|x_{3}+\sqrt{3}{\varepsilon}|^{2})}dx\\
%&\quad+\frac{\sqrt{3}\varepsilon}{8\sqrt{5}}\int_{\mathbb{R}^{3}_{+}}\psi^2dx+\frac{\sqrt{3}\varepsilon}{2}\int_{\mathbb{R}^{3}_{+}}\frac{{\psi}^2}{(\varepsilon^2+|x{'}|^{2}+|x_{3}+\sqrt{3}{\varepsilon}|^2)}dx\\
%&\quad+\frac{\sqrt{3}\varepsilon}{16}\int_{\mathbb{R}^{3}_{+}}{\psi}^2dx\\
<A_3+\frac{(3+\sqrt{5})\varepsilon}{40}\int_{\mathbb{R}^{3}_{+}}{\psi}^{2}dx
+\frac{(5+\sqrt{5})\varepsilon}{10}\int_{\mathbb{R}^{3}_{+}}\frac{\psi^2}{(|x{'}|^{2}+|x_{3}+{\varepsilon}|^{2})}dx.
%  \end{aligned}
\end{equation*}
This inequality and \eqref{w2} imply that
%\begin{equation}\label{w2}
 % \begin{aligned}
%\int_{\mathbb{R}^{3}_{+}}{\psi}^2dx
%&=\int_{\mathbb{R}^{3}_{+}}e^{-\frac{|x|^2}{4\sqrt{5}}}dx
%=\frac{\omega_3}{2}\int_{0}^{\infty}e^{-\frac{r^2}{4\sqrt{5}}}r^{2}dr
%=\sqrt{5}\omega_3\int_{0}^{\infty}e^{-\frac{r^2}{4\sqrt{5}}}dr\\
%&=2\sqrt{5}\int_{\mathbb{R}^{3}_{+}}\frac{e^{-\frac{|x|^2}{4\sqrt{5}}}}{|x|^2}dx
%=2\sqrt{5}\int_{\mathbb{R}^{3}_{+}}\frac{\psi^2}{|x|^2}dx,
%  \end{aligned}
%\end{equation}
%where $\omega_3$ is the area of unit sphere in $\mathbb{R}^{3}$.
%Combining \eqref{w1} with \eqref{w2}, we get that
\begin{equation}\label{cw2}
{\|\hat{v}_{{\varepsilon}}\|}^{2}<A_3+\frac{(3+\sqrt{5})\varepsilon}{4} \int_{\mathbb{R}^{3}_{+}}
\frac{e^{-\frac{|x|^2}{4\sqrt{5}}}}{|x|^{2}}dx.
\end{equation}

Notice that
\begin{equation}\label{aA1}
  \begin{aligned}
{\|\hat{v}_{{\varepsilon}}\|}_{L^{2}_{K}({\mathbb{R}}^{3}_{+})}^{2}
&=\int_{\mathbb{R}^{3}_{+}}{\psi}^{2}\hat{U}_{\varepsilon}^{2}dx
=\varepsilon\int_{\mathbb{R}^{3}_{+}}\frac{e^{-\frac{|x|^2}{4\sqrt{5}}}}{(|x{'}|^{2}+|x_{3}+{\varepsilon}|^2)}dx,
 \end{aligned}
\end{equation}
and
\begin{equation}\label{aA2}
  \begin{aligned}
0&<\int_{\mathbb{R}^{3}_{+}}\frac{e^{-\frac{|x|^2}{4\sqrt{5}}}}{|x|^2}dx-
\int_{\mathbb{R}^{3}_{+}}\frac{e^{-\frac{|x|^2}{4\sqrt{5}}}}{(|x{'}|^{2}+|x_{3}+{\varepsilon}|^2)}dx\\
&=\int_{\mathbb{R}^{3}_{+}}\frac{e^{-\frac{|x|^2}{4\sqrt{5}}}(2{\varepsilon}x_3+\varepsilon^2)}{|x|^2(|x{'}|^{2}+|x_{3}+{\varepsilon}|^2)}dx\\
&<2{\varepsilon}\int_{\mathbb{R}^{3}_{+}}\frac{e^{-\frac{|x|^2}{4\sqrt{5}}}}{|x|(\varepsilon^2+|x|^{2})}dx
+\varepsilon^2\int_{\mathbb{R}^{3}_{+}}\frac{e^{-\frac{|x|^2}{4\sqrt{5}}}}{|x|^2(\varepsilon^2+|x|^{2})}dx.
 \end{aligned}
\end{equation}
Let
\begin{equation*}
I_3:=2{\varepsilon}\int_{\mathbb{R}^{3}_{+}}\frac{e^{-\frac{|x|^2}{4\sqrt{5}}}}{|x|(\varepsilon^2+|x|^{2})}dx,
\quad
I_4:=\varepsilon^2\int_{\mathbb{R}^{3}_{+}}\frac{e^{-\frac{|x|^2}{4\sqrt{5}}}}{|x|^2(\varepsilon^2+|x|^{2})}dx.
\end{equation*}
%Now, we proceed with the computation of the two integrals above.
Using the same techniques as \eqref{A3} and \eqref{A4}, one has
\begin{equation}\label{aA3}
  %\begin{aligned}
I_3%&=\omega_3{\varepsilon}\int_{0}^{\infty}\frac{e^{-\frac{r^2}{4\sqrt{5}}}r}{\varepsilon^2+r^{2}}dr\\
%&=\sqrt{3}\omega_3{\varepsilon}\Big(\int_{0}^{\varepsilon}\frac{e^{-\frac{r^2}{4\sqrt{5}}}r}{\varepsilon^2+r^{2}}dr
%+\int_{\varepsilon}^{1}\frac{e^{-\frac{r^2}{4\sqrt{5}}}r}{\varepsilon^2+r^{2}}dr
%+\int_{1}^{\infty}\frac{e^{-\frac{r^2}{4\sqrt{5}}}r}{\varepsilon^2+r^{2}}dr\Big)\\
%&<\sqrt{3}\omega_3{\varepsilon}\Big(\int_{0}^{\varepsilon}\frac{1}{\varepsilon}dr
%+\int_{\varepsilon}^{1}\frac{1}{r}dr
%+\int_{1}^{\infty}e^{-\frac{r^2}{4\sqrt{5}}}rdr\Big)\\
=\omega_3{\varepsilon}+\omega_3{\varepsilon}|\ln\varepsilon|+\bar{c}_1\varepsilon,
  %\end{aligned}
%\end{equation}
\quad
%\begin{equation}\label{A4}
 % \begin{aligned}
I_4%&=2\omega_3{\varepsilon}^2\int_{0}^{\infty}\frac{e^{-\frac{r^2}{4\sqrt{5}}}}{\varepsilon^2+r^{2}}dr\\
%&=2\omega_3{\varepsilon}^2\Big(\int_{0}^{\varepsilon}\frac{e^{-\frac{r^2}{4\sqrt{5}}}}{\varepsilon^2+r^{2}}dr
%+\int_{\varepsilon}^{\infty}\frac{e^{-\frac{r^2}{4\sqrt{5}}}}{\varepsilon^2+r^{2}}dr
%\Big)\\
%&<2\omega_3{\varepsilon}^2\Big(\int_{0}^{\varepsilon}\frac{1}{\varepsilon^2}dr
%+\int_{\varepsilon}^{\infty}\frac{1}{r^2}dr\Big)\\
%&
=\omega_3{\varepsilon},
  %\end{aligned}
\end{equation}
where $\omega_3$ is the area of unit sphere in $\mathbb{R}^{3}$ and $\bar{c}_1>0$. In view of \eqref{aA2} and \eqref{aA3}, there holds
\begin{equation}\label{aA5}
\int_{\mathbb{R}^{3}_{+}}\frac{e^{-\frac{|x|^2}{4\sqrt{5}}}}{(|x{'}|^{2}+|x_{3}+{\varepsilon}|^2)}dx
>\int_{\mathbb{R}^{3}_{+}}\frac{e^{-\frac{|x|^2}{4\sqrt{5}}}}{|x|^2}dx-\bar{c}_2{\varepsilon}|\ln\varepsilon|-\bar{c}_3\varepsilon,
\end{equation}
where $\bar{c}_2, \bar{c}_3>0$.
It follows from \eqref{aA1} and \eqref{aA5} that
\begin{equation}\label{cw3}
{\|\hat{v}_{{\varepsilon}}\|}_{L^{2}_{K}({\mathbb{R}}^{3}_{+})}^{2}> \varepsilon \int_{\mathbb{R}^{3}_{+}}
\frac{e^{-\frac{|x|^2}{4\sqrt{5}}}}{|x|^{2}}dx-\bar{c}_2\varepsilon^2|\ln\varepsilon|-\bar{c}_3\varepsilon^2.
\end{equation}

Hence, the proof is completed from \eqref{cw2} and \eqref{cw3}.
\end{proof}

The following Lemma deals with the ${L^{4}_{K}({\mathbb{R}}^{2})}$-norm of $\hat{v}_{{\varepsilon}}$ as  $\varepsilon\rightarrow0$.
\begin{Lem}\label{lem4.5}
If $N=3$, one has
\begin{align*}
%{\|v_{{\varepsilon}}\|}_{L^{{2}^{*}}_{K}({\mathbb{R}}^{N}_{+})}^{2^{*}}&=K_{2}+O({\varepsilon}^{2}),\\
{\|\hat{v}_{{\varepsilon}}\|}_{L^{{2}_{*}}_{K}({\mathbb{R}}^{N-1})}^{2_{*}}&=B_{N}^{{2_*}/2}+O({\varepsilon}^{2}|\ln\varepsilon|),
\end{align*}
where $\varepsilon>0$ is sufficiently small and $B_N$ is defined in \eqref{aa1}.
\end{Lem}
\begin{proof}
Observe that
\begin{equation*}\label{ww10}
  \begin{aligned}
{\|\hat{v}_{{\varepsilon}}\|}_{{L}_{K}^{4}({\mathbb{R}}^{2})}^{4}
&=\int_{\mathbb{R}^{2}}K(x{'},0)^{{-1}}{\psi}^{4}\hat{U}^{4}_{\varepsilon}dx{'}
=\int_{\mathbb{R}^{2}}e^{-\alpha_2|x'|^2}\hat{U}^{4}_{\varepsilon}dx{'}\\
&=\int_{\mathbb{R}^{2}}\hat{U}^{4}_{\varepsilon}dx{'}+\int_{\mathbb{R}^{2}}(e^{-\alpha_2|x'|^2}-1)\hat{U}^{4}_{\varepsilon}dx{'}\\
&=B_{3}^2+\int_{\mathbb{R}^{2}}\frac{e^{-\alpha_2{{\varepsilon}^{2}|y'|^{2}}-1}}
{{{{(|y{'}|^{2}+1)^{2}}}}}dy{'}\\
&=B_{3}^2+\int_{\hat{B}_{1/\varepsilon}}\frac{e^{-\alpha_2{\varepsilon}^{2}|y'|^{2}}-1}
{{{{(|y{'}|^{2}+1)^{2}}}}}dy{'}
+\int_{\mathbb{R}^{2}\setminus \hat{B}_{1/\varepsilon}}\frac{e^{-\alpha_2{\varepsilon}^{2}|y'|^{2}}-1}
{{{{(|y{'}|^{2}+1)^{2}}}}}dy{'},
  \end{aligned}
\end{equation*}
where $\alpha_2=\frac{1}{4}+\frac{1}{2\sqrt{5}}$. %It follows from Taylor's formula that
%\begin{equation*}\label{ww11}
 % \begin{aligned}
%\int_{\hat{B}_{1/\varepsilon}}\frac{ e^{-\alpha_2{\varepsilon}^{2}|y{'}|^{2}}-1}
%{(|y{'}|^{2}+1)^{2}}dy{'}
%&=-\alpha_2{\varepsilon^{2}}\int_{\hat{B}_{1/\varepsilon}}\frac{|y{'}|^{2}}
%{(|y{'}|^{2}+1)^{2}}dy{'}
%+O\Big(\varepsilon^4\int_{\hat{B}_{1/\varepsilon}}\frac{ |y{'}|^{4}}
%{(|y{'}|^{2}+1)^{2}}dy{'}\Big)\\
%&=-\alpha_2\varepsilon^{2}\int_{\hat{B}_{1/\varepsilon}\setminus \hat{B}_{1}}\frac{|y{'}|^{2}}
%{(|y{'}|^{2}+1)^{2}}dy{'}
%+O\Big(\varepsilon^{4}\int_{\hat{B}_{1/\varepsilon}\setminus \hat{B}_{1}}\frac{ |y{'}|^{4}}
%{(|y{'}|^{2}+1)^{2}}dy{'}\Big)\\
%&\quad+O(\varepsilon^{2})+O(\varepsilon^4).
 % \end{aligned}
 % \end{equation*}
Arguing as \eqref{w12} and \eqref{w13}, we obtain
\begin{equation*}
{\|\hat{v}_{{\varepsilon}}\|}_{L^{4}_{K}({\mathbb{R}}^{2})}^{4}=
B_{3}^2+O({\varepsilon}^2|\ln\varepsilon|).
\end{equation*}
We complete the proof of this Lemma.
\end{proof}

Now, we are ready to verify \eqref{4.0}.
\begin{Lem}\label{lem4.6}  The inequality (\ref {4.0}) holds %and (\ref {2.12}) is naturally obtained,
if one of the following assumptions holds:
\vskip 0.2cm
	
	$( \romannumeral 1)$  $~N=3,\lambda>\frac{3+\sqrt{5}}{4}$;
\vskip 0.2cm
	
	$( \romannumeral 2)$  $~N\geq4, \lambda>\frac{N}{4}+\frac{N-4}{8}$.
	\vskip 0.2cm
\end{Lem}	
\begin{proof}
For $N=3$, by the definition of $\hat{c}_{\lambda, \mu}$ and a straightforward calculation, we have
\begin{equation*}\label{4.1}
\hat{c}_{\lambda, \mu}\leq\sup\limits_{s>0}I^0_{\lambda,0}(s\hat v_\varepsilon)
=\frac{1}{4}
\Bigg(\frac{{\|\hat v_{{\varepsilon}}\|}^{2}-\lambda \|\hat v_{{\varepsilon}}\|_{L^{{2}}_{K}({\mathbb{R}}^{3}_{+})}^{2}}
{\|\hat v_\varepsilon\|_{L^{4}_{K}({\mathbb{R}}^{2})}^{2}}\Bigg)^{2}.
\end{equation*}
The inequality (\ref {4.0}) holds if we claim
\begin{equation}\label{4.11}
\frac{{\|\hat v_\varepsilon\|}^{2}-\lambda \|\hat v_\varepsilon\|_{L^{{2}}_{K}({\mathbb{R}}^{3}_{+})}^{2}}
{\|\hat v_\varepsilon\|_{L^{4}_{K}({\mathbb{R}}^{2})}^{2}}<S_0.
\end{equation}
It follows from Lemma \ref{lem4.4} and Lemma \ref{lem4.5} that
\begin{equation*}\label{4.111}
  \begin{aligned}
\frac{{\|\hat v_\varepsilon \|}^{2}-\lambda \|\hat v _\varepsilon\|_{L^{{2}}_{K}({\mathbb{R}}^{3}_{+})}^{2}}
{\|\hat v_\varepsilon\|_{L^{4}_{K}({\mathbb{R}}^{2})}^{2}}
&<
\frac{A_3+\varepsilon\big(\frac{3+\sqrt{5}}{4}-\lambda\big) \int_{\mathbb{R}^{3}_{+}}
\frac{\psi^2}{|x|^{2}}dx+\lambda
c_1\varepsilon^2|\ln\varepsilon|+\lambda c_2\varepsilon^2}{B_{3}+O({\varepsilon}^{2}|\ln\varepsilon|)}\\
&=\frac{A_3}{B_{3}}+\varepsilon\Big(\frac{3+\sqrt{5}}{4}-\lambda\Big)B_3^{-1} \int_{\mathbb{R}^{3}_{+}}\frac{\psi^2}{|x|^{2}}dx
+O({\varepsilon}^{2}|\ln\varepsilon|)\\
&=S_0+\varepsilon\Big(\frac{3+\sqrt{5}}{4}-\lambda\Big){B_3^{-1}} \int_{\mathbb{R}^{3}_{+}}\frac{\psi^2}{|x|^{2}}dx+
O({\varepsilon}^{2}|\ln\varepsilon|)<S_0,
  \end{aligned}
\end{equation*}
since $\lambda>\frac{3+\sqrt{5}}{4}$ and $\varepsilon$ small enough. Hence, \eqref{4.11} holds.

For $N\geq4$, simple computation yields that
\begin{equation*}\label{4.1}
\hat{c}_{\lambda, \mu}\leq\sup\limits_{s>0}I^0_{\lambda,0}(s\hat u_\varepsilon)
=\frac{1}{2(N-1)}
\Bigg(\frac{{\|\hat u_\varepsilon\|}^{2}-\lambda \|\hat u_\varepsilon\|_{L^{{2}}_{K}({\mathbb{R}}^{N}_{+})}^{2}}
{\|\hat u_\varepsilon\|_{L^{2_*}_{K}({\mathbb{R}}^{N-1})}^{2}}\Bigg)^{N-1}.
\end{equation*}
This means that we just need to show
\begin{equation}\label{4.12}
\frac{{\|\hat u_\varepsilon\|}^{2}-\lambda \|\hat u_\varepsilon\|_{L^{{2}}_{K}({\mathbb{R}}^{N}_{+})}^{2}}
{\|\hat u_\varepsilon\|_{L^{2_*}_{K}({\mathbb{R}}^{N-1})}^{2}}<S_0.
\end{equation}
For $N=4$, it follows from Lemmas \ref{lem4.1}-\ref{lem4.3} that
\begin{equation}\label{4.2}
  \begin{aligned}
\frac{{\|\hat u_\varepsilon \|}^{2}-\lambda \|\hat u_\varepsilon \|_{L^{{2}}_{K}({\mathbb{R}}^{4}_{+})}^{2}}
{\|\hat u_\varepsilon \|_{L^{3}_{K}({\mathbb{R}}^{3})}^{2}}
&=\frac{A_4+\frac{\omega_4}{2}(1-\lambda){{\varepsilon}}^{2}|\ln{\varepsilon}|+O({\varepsilon}^{2})}
{B_4+O(\varepsilon^2)}\\
&=\frac{A_4}{B_4}+\frac{\omega_4}{2}B_4^{-1}(1-\lambda){{\varepsilon}}^{2}|\ln{\varepsilon}|+O({\varepsilon}^{2})\\
&=S_0+\frac{\omega_4}{2}B_4^{-1}(1-\lambda){{\varepsilon}}^{2}|\ln{\varepsilon}|+O({\varepsilon}^{2})
<S_0,
\end{aligned}
\end{equation}
since $\lambda>1$ and $\varepsilon$ small enough. For $N\geq5$, in view of Lemma \ref{lem4.2} and mean value Theorem, there holds
\begin{equation*}
{\|\hat{u}_{{\varepsilon}}\|}_{L^{{2}_{*}}_{K}({\mathbb{R}}^{N-1})}^{2}
=
B_N-\frac{2}{2_*}B_N^{\frac{2-2_*}{2}}\hat{\gamma}_{N}{{\varepsilon}}^{2}+o({\varepsilon}^{2})
\end{equation*}
and then
\begin{equation}\label{4.33}
\frac{1}{{\|\hat{u}_{{\varepsilon}}\|}_{L^{{2}_{*}}_{K}({\mathbb{R}}^{N-1})}^{2}}
=
\frac{1}{B_N}+\frac{2}{2_*}B_N^{-1-\frac{2_*}{2}}\hat{\gamma}_{N}{{\varepsilon}}^{2}+o({\varepsilon}^{2}).
\end{equation}
From Lemmas \ref{lem4.1}-\ref{lem4.3} and \eqref{4.33}, we conclude \begin{equation}\label{4.3}
  %\begin{aligned}
\frac{{\|\hat u_\varepsilon \|}^{2}-\lambda \|\hat u_\varepsilon \|_{L^{{2}}_{K}({\mathbb{R}}^{N}_{+})}^{2}}
{\|\hat u_\varepsilon \|_{L^{2_*}_{K}({\mathbb{R}}^{N-1})}^{2}}
=\frac{A_N}{B_N}+\frac{1}{B_N}(\hat{\alpha}_N-\lambda\hat{d}_N+\xi_N){\varepsilon}^{2}
+o({\varepsilon}^{2}),
%  \end{aligned}
\end{equation}
where
\begin{equation*}
\xi_N:=\frac{2}{2_*}A_NB_N^{-\frac{2_*}{2}}\hat{\gamma}_{N}.
\end{equation*}
On the other hand, from the proof of Proposition 5.2 in \cite{2021}, we get
\begin{equation*}
\frac{\hat{\alpha}_{N}+\xi_N}{\hat{d}_{N}}=\frac{N}{4}+\frac{N-4}{8}.
\end{equation*}
This, together with \eqref{4.3}, one has that for $\varepsilon$ sufficiently small,
\begin{equation}\label{4.4}
  %\begin{aligned}
\frac{{\|\hat u_\varepsilon \|}^{2}-\lambda \|\hat u_\varepsilon \|_{L^{{2}}_{K}({\mathbb{R}}^{N}_{+})}^{2}}
{\|\hat u_\varepsilon \|_{L^{2_*}_{K}({\mathbb{R}}^{N-1})}^{2}}
<S_0,
%  \end{aligned}
\end{equation}
since
\begin{equation*}
\lambda>\frac{N}{4}+\frac{N-4}{8}.
\end{equation*}
Therefore, \eqref{4.12} holds from \eqref{4.2} and \eqref{4.4}.
\end{proof}

\subsection{Verification of the condition (5.1) when $\lambda\geq0$ and $\mu>0$ }\label{S3.2}

In this subsection, we are committed to checking condition \eqref{4.0} under the assumptions $(2)$-$(3)$ of Theorem \ref{Th1.2}. To this end, we introduce the energy function corresponding to equation \eqref{1.5} with $ \lambda\geq0, a=0$, $\mu>0$ and $2\leq q<2_*$ as follows
\begin{equation}
	I^0_{\lambda, \mu}(u):={\frac{1}{2}}{\|{u}\|}^{2}-{\frac{\lambda }{2}}\|u_+\|_{L^{{2}}_{K}({\mathbb{R}}^{N}_{+})}^{2}
-{\frac{1 }{2_{*}}}\|{u_{+}}\|_{L^{{2}_{*}}_{K}({\mathbb{R}}^{N-1})}^{2_{*}}-
{\frac{\mu}{q}}\|{u_{+}}\|_{L^{q}_{K}({\mathbb{R}}^{N-1})}^{q}.\notag
\end{equation}
 The estimate for $ {\|\hat{u}_{{\varepsilon}}\|}_{L^{q}_{K}({\mathbb{R}}^{N-1})}^{q}$ as $\varepsilon \rightarrow 0$ was given in \cite{116} when $N\ge 4$. In the following, we provide the estimate of  $ {\|\hat{u}_{{\varepsilon}}\|}_{L^{q}_{K}({\mathbb{R}}^{N-1})}^{q}$ for all $N\ge 3$,
which is necessary to verify condition \eqref{4.0}.
\begin{Lem}\label{lem4.7}
Assume that $q\in [2,2^{*})$. There holds
\begin{equation*}\label{bc0}
{\|\hat{u}_{{\varepsilon}}\|}_{L^{q}_{K}({\mathbb{R}}^{N-1})}^{q}
\geq
\begin{cases}
l_1{\varepsilon}^{\theta_N}+o({\varepsilon}^{\theta_N}),~&~~N\geq3,~2< q<2_*,\\[1mm]
l_1{\varepsilon}+o({\varepsilon}),~&~~N\geq4,~q=2,\\[1mm]
l_2{\varepsilon}|\ln{\varepsilon}|+l_3{\varepsilon},~&~~N=3,~q=2\\[1mm]
\end{cases}
\end{equation*}
as $\varepsilon\rightarrow 0$, where $l_{i}, i=1, 2, 3$ are positive constants independent of $\varepsilon$ and  $\theta _{N}$ is given in Lemma \ref{lem3.6}.
\end{Lem}

\begin{proof}
By the definitions of $\hat{u}_{{\varepsilon}}$ and $\phi$, we have
\begin{align}
{\|\hat{u}_{{\varepsilon}}\|}_{L^{q}_{K}({\mathbb{R}}^{N-1})}^q\notag
&=\int_{\mathbb{R}^{N-1}}K(x',0)\hat{u}_{\varepsilon}^qdx'
=\int_{\mathbb{R}^{N-1}}\frac{K(x',0)^{1-\frac{q}{2}}\phi(x',0)^{q}\varepsilon^{\frac{(N-2)q}{2}}}{(|x'|^2+\varepsilon^{2})^{\frac{(N-2)q}{2}}}
dx'\\ \notag
&\geq e^{-\frac{q-2}{2}}\varepsilon^{\frac{(N-2)q}{2} } \int_{\hat{B}_{2}} \frac{\phi(x',0)^{q}}{(|x'|^{2}+\varepsilon^{2})^{\frac{(N-2)q}{2}}} dx' \\ \notag
&\geq e^{-\frac{q-2}{2}}\varepsilon^{\frac{(N-2)q}{2}}\int_{\hat{B}_{1}} \frac{1}{(|x'|^{2}+\varepsilon^{2})^{\frac{(N-2)q}{2}}} dx'\\ \label{ll1}
&=e^{-\frac{q-2}{2}}\varepsilon^{\theta _{N}}\int_{\hat{B}_{1/\varepsilon}} \frac{1}{(|y'|^{2}+1)^{\frac{(N-2)q}{2}}} dy'\\ \notag
&=e^{-\frac{q-2}{2}}\varepsilon^{\theta _{N}}\int_{\hat{B}_{1}} \frac{1}{(|y'|^{2}+1)^{\frac{(N-2)q}{2}}} dy'\\ \notag
&\quad+e^{-\frac{q-2}{2}} \varepsilon^{\theta _{N}}\int_{\hat{B}_{1/\varepsilon} \setminus \hat{B}_{1}} \frac{1}{(|y'|^{2}+1)^{\frac{(N-2)q}{2}}} dy'\\ \notag
&\geq \bar{l}_1\varepsilon^{\theta _{N}}+\bar{l}_2\varepsilon^{\theta _{N}}\int_{\hat{B}_{1/\varepsilon} \setminus \hat{B}_{1}} \frac{1}{|y'|^{(N-2)q}} dy'\\ \notag
&=\bar{l}_1\varepsilon^{\theta _{N}}+\bar{l}_{2}{\omega_{N-1}}\varepsilon^{\theta_{N}} \int_{1}^{{{1}/\varepsilon}}r^{N-2-(N-2)q}dr, \notag
\end{align}
where $\theta _{N}=N-1-\frac{(N-2)q}{2}$, $\bar{l}_1, \bar{l}_2>0$ and $\omega_{N-1}$ is the area of unit sphere in $\mathbb{R}^{N-1}$. Combining \eqref{c2}-\eqref{c4} with \eqref{ll1}, we have
\begin{equation}\label{lc3}
{\|\hat{u}_{{\varepsilon}}\|}_{L^{q}_{K}({\mathbb{R}}^{N-1})}^{q}
\geq
\begin{cases}
\bar{l}_{3}\varepsilon^{\frac{(N-2)q}{2}}+o(\varepsilon^{\frac{(N-2)q}{2}}),~&~~q<\frac{N-1}{N-2},\\[1mm]
\bar{l}_{4} \varepsilon^{\frac{N-1}{2}}|\ln \varepsilon| +\bar{l}_{1}\varepsilon^{\frac{N-1}{2}} ,~&~~q=\frac{N-1}{N-2},\\[1mm]
\bar{l}_{5}\varepsilon^{\theta_{N}}+o(\varepsilon^{\theta_{N}}), ~&~~q>\frac{N-1}{N-2},
\end{cases}
\end{equation}
where constants $\bar{l}_3, \bar{l}_4, \bar{l}_5$ are positive. Thus, the proof is completed from \eqref{lc3}.
\end{proof}

In the following, we are going to verify \eqref{4.0}.

\begin{Lem}\label{lem4.8}
If $N\geq3, \lambda\geq0$, $\mu>0$ and $2\leq q<2_*$, then the inequality \eqref{4.0} holds.
\end{Lem}
\begin{proof}
Define the function
\begin{equation*}
{\hat{g}}_{{\varepsilon}}(s):%=I^0_{\lambda, \mu}(s\hat u _\varepsilon)
={\frac{s^2}{2}}\Big({\|\hat{u}_\varepsilon\|}^{2}-\lambda \|\hat{u}_\varepsilon\|_{L^{{2}}_{K}({\mathbb{R}}^{N}_{+})}^{2}\Big)
-{\frac{s^{2_{*}} }{2_{*}}}\|\hat{u}_\varepsilon\|_{L^{{2}_{*}}_{K}({\mathbb{R}}^{N-1})}^{2_{*}}
-{\frac{\mu s^q}{q}}\|\hat{u}_\varepsilon\|_{L^{q}_{K}({\mathbb{R}}^{N-1})}^{q},
\end{equation*}
where $s>0$. This means that we only need to verify
\begin{equation}\label{bc3}
\sup\limits_{s>0}\hat{g}_{{\varepsilon}}(s)<\frac{1}{2(N-1)}S_0^{N-1}.
\end{equation}
Let $s_{\varepsilon}>0$ be a constant such that ${\hat g}_{{\varepsilon}}(s)$ attains its maximum value. One has
\begin{equation*}
{\|\hat{u}_\varepsilon\|}^{2}-\lambda \|\hat{u}_\varepsilon\|_{L^{{2}}_{K}({\mathbb{R}}^{N}_{+})}^{2}-
\|\hat{u}_\varepsilon\|_{L^{{2}_{*}}_{K}({\mathbb{R}}^{N-1})}^{2_{*}}
s^{2_{*}-2}-
{\mu}\|\hat{u}_\varepsilon\|_{L^{q}_{K}({\mathbb{R}}^{N-1})}^{q}s^{q-2}=0.\notag
\end{equation*}
For any fixed $\lambda\geq0$, it follows from Lemmas \ref{lem4.1}-\ref{lem4.3} and Lemma \ref{lem4.7} that there exists $a_{2}>0$, independent of ${\varepsilon}$, such that ${s_{\varepsilon}}\geq a_{2}$ for any ${\varepsilon}>0$ small enough. Then
  \begin{align}
\hat{g}_{{\varepsilon}}(s_{\varepsilon})
&\leq\sup\limits_{s>0}\Bigg( \notag {\frac{s^2}{2}}\Big({\|\hat{u}_\varepsilon\|}^{2}-\lambda \|\hat{u}_\varepsilon\|_{L^{{2}}_{K}({\mathbb{R}}^{N}_{+})}^{2}\Big)
-{\frac{s^{2_{*}} }{2_{*}}}\|\hat{u}_\varepsilon\|_{L^{{2}_{*}}_{K}({\mathbb{R}}^{N-1})}^{2_{*}}\Bigg)
-{\frac{\mu{s^q_\varepsilon}}{q}}\|\hat{u}_\varepsilon\|_{L^{q}_{K}({\mathbb{R}}^{N-1})}
^{q}\\
&\leq\sup\limits_{s>0}\Bigg( \label{bc4} {\frac{s^2}{2}}\Big({\|\hat{u}_\varepsilon\|}^{2}-\lambda \|\hat{u}_\varepsilon\|_{L^{{2}}_{K}({\mathbb{R}}^{N}_{+})}^{2}\Big)
-{\frac{s^{2_*} }{2_{*}}}\|\hat{u}_\varepsilon\|_{L^{{2}_{*}}_{K}({\mathbb{R}}^{N-1})}^{2_{*}}\Bigg)
-{\frac{\mu{a^{q}_2}}{q}}\|\hat{u}_\varepsilon\|_{L^{q}_{K}({\mathbb{R}}^{N-1})}^{q}\\
&=\frac{1}{2(N-1)}
\Bigg(\frac{{\|\hat{u}_\varepsilon\|}^{2}-\lambda \|\hat u_\varepsilon \|_{L^{{2}}_{K}({\mathbb{R}}^{N}_{+})}^{2}}
{\|\hat u_\varepsilon \|_{L^{2_*}_{K}({\mathbb{R}}^{N-1})}^{2}}\Bigg)^{N-1}
-{\frac{\mu{a^{q}_2}}{q}}\|\hat u _\varepsilon\|_{L^{q}_{K}({\mathbb{R}}^{N-1})}^{q}.\notag
  \end{align}
Similar to the proof of Lemma \ref{lem4.6}, we conclude
\begin{equation}\label{bc5}
\frac{{\|\hat u_\varepsilon \|}^{2}-\lambda \|\hat u_\varepsilon \|_{L^{{2}}_{K}({\mathbb{R}}^{N}_{+})}^{2}}
{\|\hat u_\varepsilon \|_{L^{2_*}_{K}({\mathbb{R}}^{N-1})}^{2}}
=
S_0+B_\varepsilon,
\end{equation}
where $B_\varepsilon$ is given by \eqref{00}.
Now, we are ready to verify \eqref{bc3}.
For $N\geq3$ and $2<q<2_*$, in view of Lemma \ref{lem4.7}, \eqref{bc4}, \eqref{bc5}, we have that for $\varepsilon$ sufficiently small,
\begin{equation}\label{ss1}	
\hat{g}_{{\varepsilon}}({s_{\varepsilon}})\leq
\frac{1}{2(N-1)}S_0^{N-1}+B_\varepsilon-\frac{\mu a_{2}^q l_{1}}{q}{{\varepsilon}}^{\theta _{N}}+o({{\varepsilon}}^{\theta _{N}})<\frac{1}{2(N-1)}S_0^{N-1},
\end{equation}
since $\theta_{N}\in (0,1)$ and $\mu>0$. For $q=2$, from Lemma \ref{lem4.7}, \eqref{bc4}, \eqref{bc5}, we proceed as follows:
\vskip 1mm
(1) If $N\geq4$, then for any $\mu>0$, $\varepsilon$ sufficiently small,
\begin{equation}\label{ss2}
 \hat{g}_{{\varepsilon}}({s_{\varepsilon}})\leq
\frac{1}{2(N-1)}S_0^{N-1}-\frac{\mu a_{2}^q l_{1}}{q}{{\varepsilon}}+o({{\varepsilon}})<\frac{1}{2(N-1)}S_0^{N-1}.
\end{equation}

(2) If $N=3$, one has
\begin{equation}\label{ss3}
 \hat{g}_{{\varepsilon}}({s_{\varepsilon}})\leq
\frac{1}{4}S_0^{2}-\frac{\mu a_{2}^q l_{2}}{q}{{\varepsilon}}|\ln \varepsilon| +O({{\varepsilon}})<\frac{1}{4}S_0^{2},
\end{equation}
for small $\varepsilon$ and $\mu>0$.

Therefore, we conclude that \eqref{bc3} holds for $\varepsilon>0$ sufficiently small from \eqref{ss1}-\eqref{ss3}.
\end{proof}

Now we are ready to prove our Theorem \ref{Th1.2}.

{\bf Proof of Theorem \ref{Th1.2}.}
From Lemma \ref{lem2.5}, Lemma \ref{lem2.7}, Lemma \ref{lem4.6}, Lemma \ref{lem4.8} and Mountain Pass Lemma, we obtain that \eqref{1.5} has a
nonnegative weak solution $u$ under the assumptions of Theorem \ref{Th1.2}. Moreover, from the Brezis-Kato Theorem and maximum principle, $u$ is
a positive solution of \eqref{1.5}. That is, $u$ is a positive solution of \eqref{1.1}. %The proof of Theorem \ref{Th1.2} is completed.
 \qed

\vs{2mm}
\noindent\textbf {Acknowledgments}
The work is supported by NSFC [grant numbers 12271196, 11931012].

\vskip0.2cm
\noindent{\bf Statements and Declarations}

\noindent{\bf Competing Interests}
%The authors declare no competing interests
On behalf of all authors, the corresponding author states that there is no conflict of interest.

\noindent{\bf Data availability}
No data was used for the research described in the article.

\end{document}